\documentclass[11pt,letterpaper]{amsart}
\usepackage{graphicx,subcaption}
\usepackage{amsmath} 
\usepackage{amsthm,amsfonts,amssymb,mathrsfs,amscd,amstext,amsbsy} 
\usepackage{epic,eepic} 
\usepackage{yfonts}
\usepackage{paralist,enumerate}
\usepackage[all]{xy}
\usepackage{hyperref}
\hypersetup{colorlinks}
\usepackage{tikz}
\usetikzlibrary {positioning}
\definecolor {processblue}{cmyk}{0.96,0,0,0}
\usetikzlibrary{arrows}
\tikzset{    vertex/.style={circle,draw,minimum size=1.5em},    edge/.style={->,> = latex'}}

\newtheorem{theorem}{Theorem}[section]
\newtheorem{cor}[theorem]{Corollary}
\newtheorem{lemma}[theorem]{Lemma}

\newtheorem{theo}[theorem]{Theorem}
\newtheorem{lem}[theorem]{Lemma}
\newtheorem{pro}[theorem]{Proposition}

\newtheorem{exa}[theorem]{Example}

\newtheorem{Definition}[theorem]{Definition}
\newtheorem*{Definition*}{Definition}

\begin{document}
\title[Circle action on oriented 6-manifold with isolated fixed points]{Graphs for torus actions on oriented manifolds with isolated fixed points and classification in dimension 6}
\author{Donghoon Jang}
\thanks{MSC 2020: primary 58C30, secondary 57M60}
\address{Department of Mathematics, Pusan National University, Pusan, South Korea}
\email{donghoonjang@pusan.ac.kr}

\begin{abstract}
Let a torus act on a compact oriented manifold $M$ with isolated fixed points, with an additional mild assumption that its isotropy submanifolds are orientable. We associate a signed labeled multigraph encoding the fixed point data (weights and signs at fixed points and isotropy submanifolds) of the manifold. We study operations on $M$ and its multigraph, (self) connected sum and blow up, etc.

When the circle group acts on a 6-dimensional $M$, we classify such a multigraph by proving that we can convert it into the empty graph by successively applying two types of operations. In particular, this classifies the fixed point data of any such manifold.

We prove this by showing that for any such manifold, we can successively take equivariant connected sums at fixed points with itself, $\mathbb{CP}^3$, and 6-dimensional analogue $Z_1$ and $Z_2$ of the Hirzebruch surfaces (and these with opposite orientations)
 to a fixed point free action on a compact oriented 6-manifold.

We also classify a multigraph for a torus action on a 4-dimensional $M$.
\end{abstract}

\maketitle

\tableofcontents

\section{Introduction}

Torus actions on manifolds have been studied in low dimensions. In dimensions 1 and 2 the classification results are known. Circle actions and torus actions on 3- and 4-manifolds have been studied and classified in 1960's and 1970's. Raymond classified circle actions on compact 3-manifolds \cite{R}.
Orlik and Raymond proved that a 2-torus action on a simply connected closed orientable 4-manifold is $S^4$, or a connected sum of $\mathbb{CP}^2$, $\overline{\mathbb{CP}^2}$, and the Hirzebruch surfaces \cite{OR1}. Fintushel proved an analogous result that a circle action on a simply connected closed oriented 4-manifold is a connected sum of $S^4$, $\mathbb{CP}^2$, $\overline{\mathbb{CP}^2}$, and $S^2 \times S^2$ \cite{F2}. 
For non-simply connected case, Orlik and Raymond showed that a 2-torus action on a closed orientable 4-manifold is determined by its orbit data \cite{OR2}. 
Fintushel showed that a circle action on a closed oriented 4-manifold is also determined by its orbit data \cite{F3}.
Also see \cite{F1}, \cite{J1}, \cite{P} for some classification results on certain 4-dimensional oriented $S^1$-manifolds.

Circle actions on different types of 4-manifolds with fixed points have been also classified. Carrell, Howard, and Kosniowski studied complex surfaces \cite{CHK}. Ahara and Hattori \cite{AH}, Audin \cite{Au}, and Karshon \cite{Ka} studied symplectic 4-manifolds. 
For both complex 4-manifolds and symplectic 4-manifolds, there is a common phenomenon; we can successively blow down any such manifold to one of minimal manifolds, which are $\mathbb{CP}^2$, the Hirzebruch surfaces, and ruled surfaces. Also see a generalization of these results to almost complex 4-manifolds when there are finite fixed points \cite{J3}.

Consider a torus action on a manifold that has a non-empty finite fixed point set. Because the dimension of a manifold and the dimension of its fixed point set have the same parity, the next dimension to consider is dimension 6. Unlike the classification results in dimension 4, for any type of manifold classifying torus actions on 6-manifolds becomes more demanding. Part of the reason why a classification is more complicated in dimension 6 is that (1) the orbit space is too big for an action of a torus on a 6-manifold, and (2) for a 4-dimensional complex or symplectic manifold detecting and blowing down a 2-sphere with self-intersection number -1 to a minimal manifold is relatively simple as done in the above results for complex and symplectic manifolds, but for this kind of argument in dimension 6 we need to blow down $\mathbb{CP}^2$ containing 3 fixed points whose weights must cooperate to be able to blow down equivariantly, and this is difficult to realize.

Known classification results for torus actions on 6-manifolds have strong assumptions, on actions and/or manifolds. For instance, McGavran and Oh considered $T^3$-actions on closed orientable 6-manifolds that have orbit space $D^3$ \cite{MO}, Kuroki considered $T^3$-actions on closed oriented 6-manifolds with fixed points and with vanishing odd cohomology groups \cite{Ku}, Tolman considered Hamiltonian $S^1$-actions on symplectic manifolds with minimal cohomology groups \cite{T}, and Ahara and the author considered $S^1$-actions on almost complex 6-manifolds with small number of fixed points \cite{Ah}, \cite{J2}.

Let a $k$-dimensional torus $T^k$ act effectively on a $2n$-dimensional manifold $M$ and suppose that there is a fixed point. The \textbf{complexity} of the action is $n-k$. Since there is a fixed point, $k$ can be at most $n$, and if $k=n$, then there are only finitely many fixed points. The bigger the complexity is, the more difficult a classification is. People have studied torus actions of minimal complexity (complexity zero, $k=n$), for instance quasitoric manifolds \cite{DJ}, topological toric manifolds \cite{IFM}, unitary toric manifolds \cite{Ma}, etc. For instance, a Hamiltonian $T^n$-action on a $2n$-dimensional compact symplectic manifold is determined by its moment image  \cite{D}. For the case of complexity one, see \cite{KT1}, \cite{KT2}, and \cite{KT3} for Hamiltonian actions on compact symplectic manifolds of complexity one.

People have used graphs to study torus actions on manifolds, where a graph encodes the data on the fixed point set of a given manifold. For instance, Karshon classified 4-dimensional Hamiltonian $S^1$-spaces in terms of graphs for the aforementioned result \cite{Ka}. Godinho and Sabatini \cite{GS} and Tolman \cite{T} used multigraphs to study Hamiltonian $S^1$-actions in higher dimensions. Musin studied Petrie's conjecture in low dimensions via labeled multigraphs \cite{Mu}. The author studied $S^1$-actions on almost complex manifolds via multigraphs \cite{J4}.

In this paper, we study multigraphs encoding the fixed point data of torus actions on compact oriented manifolds with isolated fixed points that satisfy a mild condition that isotropy submanifolds are orientable. For instance, the condition is satisfied if a given manifold admits a unitary structure, an (almost) complex structure, or a symplectic structure, and a torus group acts on the manifold, preserving the given structure. Let a torus $T^k$ act on a compact oriented manifold $M$ with isolated fixed points, whose isotropy submanifolds are orientable. By the fixed point data of $M$ we shall mean weights and signs at the fixed points; see Definition \ref{d21}. We show that there always exists a multigraph that encodes the fixed point data (Proposition \ref{p34}). Moreover, the multigraph also encodes which fixed points are in the same isotropy submanifold and hence the multigraph also encodes geometry (isotropy submanifolds) of a given manifold; see Definition \ref{d32}. Such a multigraph is signed and labeled, meaning that each vertex has sign $\pm$ and each edge is labeled by an element of $\mathbb{Z}^k$ (Definition \ref{d31}) whose label is a weight of a vertex (fixed point). We also introduce several operations on a multigraph, such as reversing an edge, a connected sum of two multigraph, a self connected sum of a multigraph, blow up of a multigraph, etc. We shall see that our multigraphs and these operations on multigraphs behave well with operations on manifolds, implying that our multigraphs are proper ones to consider.


We classify multigraphs for $S^1$-actions on 6-dimensional oriented manifolds with isolated fixed points, by showing that there is a systematic way of converting any such multigraph into the empty graph, and any multigraph between the steps are realized as one encoding the fixed point data of some $S^1$-manifold. Moreover, we only need two types of operations, called reversing an edge and connected sum. We can roughly state our multigraph version of main result of this paper as follows.

\begin{theo}[\textbf{Theorem \ref{t61}}] \label{t11}
For a circle action on a 6-dimensional compact oriented manifold with isolated fixed points whose isotropy submanifolds are orientable, there exists a signed labeled multigraph encoding the fixed point data of the manifold, and we can successively apply two types of operations, called reversing edge and connected sum, where we take connected sums with the multigraph itself, or with 3 types of standard multigraphs for the manifolds $\mathbb{CP}^3$, $Z_1$, and $Z_2$, to convert the multigraph into the empty graph. There is a definite procedure that this ends in a finite number of steps.
\end{theo}

For the precise statement of Theorem \ref{t11}, see Theorem \ref{t61}. A new type of manifolds $Z_1$ and $Z_2$, which did not appear in classifications of circle actions and torus actions on lower dimensional manifolds, play an important role. For the manifolds $Z_n$ and actions on $S^6$, $\mathbb{CP}^3$, and $Z_n$ that we use, see Section \ref{s7}; the actions are all standard.

Moreover, along the way of proving Theorem \ref{t11}, we also classify multigraphs for torus actions on 4-dimensional compact oriented manifolds with isolated fixed points; see Theorem \ref{t51}. 

We show how Theorem \ref{t11} works with an example, though we have not introduced necessary terminologies yet.

Let $0<a<b<c$ be integers. Let $S^1$ act on the complex projective space $\mathbb{CP}^3$ by
\begin{center}
$g \cdot [z_0:z_1:z_2:z_3]=[z_0:g^a z_1: g^b z_2: g^c z_3]$
\end{center}
for all $g \in S^1 \subset \mathbb{C}$. The action has 4 fixed points $p_1=[1:0:0:0]$, $\cdots$, $p_4=[0:0:0:1]$ that have fixed point data
\begin{center}
$[+,a,b,c]$, $[-,a,b-a,c-a]$, $[+,b,b-a,c-b]$, $[-,c,c-a,c-b]$,
\end{center}
see Definition \ref{d21} and Example \ref{e412} for terminologies and details. The left multigraph of Figure \ref{f1-1} encodes this fixed point data. If we reverse the orientation of $\mathbb{CP}^3$ then the right multigraph of Figure \ref{f1-1} encodes the fixed point data of $\overline{\mathbb{CP}^3}$ where we denote its fixed points by $q_1,\cdots,q_4$. First, we perform an operation of multigraphs, called a connected sum, at $p_1$ and $q_1$ to get a multigraph $\Gamma_1$, which is Figure \ref{f1-2}. Next we perform an operation, called a self connected sum, of $\Gamma_1$ at $p_2$ and $q_2$ to get $\Gamma_2$ (Figure \ref{f1-3}), and then on $\Gamma_2$ we take a self connected sum at $p_3$ and $q_3$ to get $\Gamma_3$ (Figure \ref{f1-4}). Finally if we take a self connected sum of $\Gamma_3$ at $p_4$ and $q_4$, we get the empty graph. Figure \ref{f1} illustrates the whole procedure.

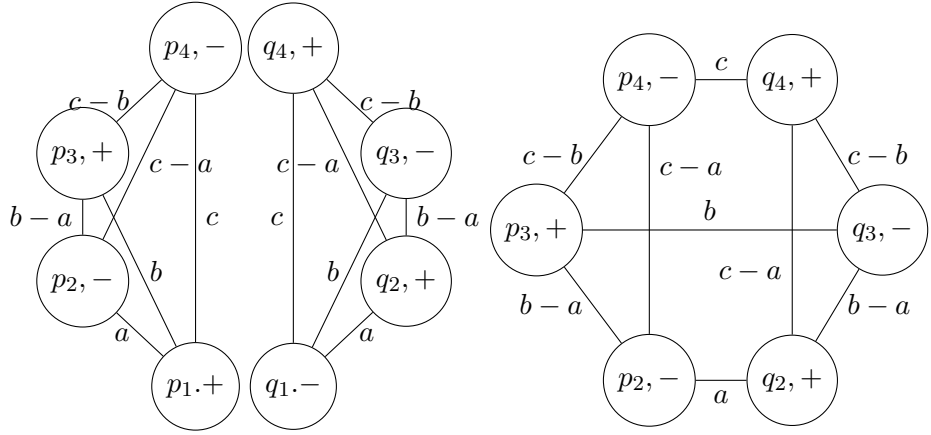
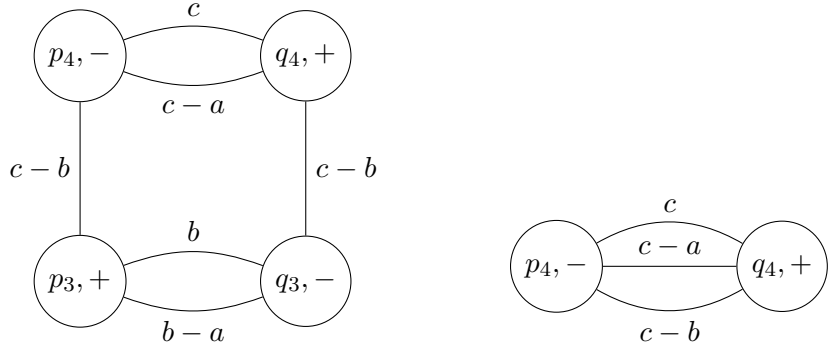
\begin{figure} 
\centering
\begin{subfigure}[b][6.8cm][s]{.49\textwidth}
\centering
\vfill
\begin{tikzpicture}[state/.style ={circle, draw}]
\node[state] (a) at (0,0) {$p_1.+$};
\node[state] (b) at (-1.5,1.4) {$p_2,-$};
\node[state] (c) at (-1.5,3.1) {$p_3,+$};
\node[state] (d) at (0,4.5) {$p_4,-$};
\path (a) edge node[left] {$a$} (b);
\path (a) edge node[right] {$b$} (c);
\path (a) edge node[right] {$c$} (d);
\path (b) edge node [left] {$b-a$} (c);
\path (b) edge node [right] {$c-a$} (d);
\path (c) edge node [left] {$c-b$} (d);
\node[state] (e) at (1.3,0) {$q_1.-$};
\node[state] (f) at (2.8,1.4) {$q_2,+$};
\node[state] (g) at (2.8,3.1) {$q_3,-$};
\node[state] (h) at (1.3,4.5) {$q_4,+$};
\path (e) edge node[right] {$a$} (f);
\path (e) edge node[left] {$b$} (g);
\path (e) edge node[left] {$c$} (h);
\path (f) edge node [right] {$b-a$} (g);
\path (f) edge node [left] {$c-a$} (h);
\path (g) edge node [right] {$c-b$} (h);
\end{tikzpicture}
\caption{Multigraph for $\mathbb{CP}^3$ (left) and $\overline{\mathbb{CP}^3}$ (right)} \label{f1-1}
\vfill
\end{subfigure}
\begin{subfigure}[b][6.8cm][s]{.49\textwidth}
\centering
\vfill
\begin{tikzpicture}[state/.style ={circle, draw}]
\node[state] (b) at (0,0) {$p_2,-$};
\node[state] (c) at (-1.5,2) {$p_3,+$};
\node[state] (d) at (0,4) {$p_4,-$};
\path (b) edge node [left] {$b-a$} (c);
\path (b) edge node [pos=0.8, right] {$c-a$} (d);
\path (c) edge node [left] {$c-b$} (d);
\node[state] (f) at (1.9,0) {$q_2,+$};
\node[state] (g) at (3.1,2) {$q_3,-$};
\node[state] (h) at (1.9,4) {$q_4,+$};
\path (b) edge node[below] {$a$} (f);
\path (c) edge node[above] {$b$} (g);
\path (d) edge node[above] {$c$} (h);
\path (f) edge node [right] {$b-a$} (g);
\path (f) edge node [pos=0.3, left] {$c-a$} (h);
\path (g) edge node [right] {$c-b$} (h);
\end{tikzpicture}
\caption{$\Gamma_1$, connected sum of $\mathbb{CP}^3$ and $\overline{\mathbb{CP}^3}$ at $p_1$ and $q_1$} \label{f1-2}
\vfill
\end{subfigure}
\begin{subfigure}[b][6cm][s]{.49\textwidth}
\centering
\vfill
\begin{tikzpicture}[state/.style ={circle, draw}]
\node[state] (c) at (0,0) {$p_3,+$};
\node[state] (d) at (0,3) {$p_4,-$};
\path (c) edge node [left] {$c-b$} (d);
\node[state] (g) at (3,0) {$q_3,-$};
\node[state] (h) at (3,3) {$q_4,+$};
\path (c) [bend left =20] edge node[above] {$b$} (g);
\path (d) [bend left =20] edge node[above] {$c$} (h);
\path (c) [bend right =20] edge node [below] {$b-a$} (g);
\path (d) [bend right =20] edge node [below] {$c-a$} (h);
\path (g) edge node [right] {$c-b$} (h);
\end{tikzpicture}
\caption{$\Gamma_2$, self connected sum of $\Gamma_1$ at $p_2$ and $q_2$} \label{f1-3}
\vfill
\end{subfigure}
\begin{subfigure}[b][4cm][s]{.49\textwidth}
\centering
\vfill
\begin{tikzpicture}[state/.style ={circle, draw}]
\node[state] (d) at (0,0) {$p_4,-$};
\node[state] (h) at (3,0) {$q_4,+$};
\path (d) [bend left =30] edge node[above] {$c$} (h);
\path (d) edge node[above] {$c-a$} (h);
\path (d) [bend right =30] edge node [below] {$c-b$} (h);
\end{tikzpicture}
\caption{$\Gamma_3$, self connected sum of $\Gamma_2$ at $p_3$ at $q_3$} \label{f1-4}
\vfill
\end{subfigure}
\caption{Converting a multigraph for $\mathbb{CP}^3$ to the empty graph}\label{f1}
\end{figure}

A geometric statement of our main result is that for a 6-dimensional compact oriented $S^1$-manifold with isolated fixed points, we can successively take equivariant connected sums to a fixed point free action, where we take each connected sum with itself, or with $\mathbb{CP}^3$, $Z_1$ and $Z_2$.

\begin{theorem}\label{t12}
Let the circle act on a 6-dimensional compact connected oriented manifold $M$ with isolated fixed points whose isotropy submanifolds are orientable. Then we can successively take equivariant connected sums with itself, or with $\mathbb{CP}^3$, $Z_1$, and $Z_2$ (and these with opposite orientations) to construct another 6-dimensional compact connected oriented manifold, which is equipped with a fixed point free $S^1$-action. There is a definite procedure that this ends in a finite number of steps. The circle actions on $\mathbb{CP}^3$, $Z_1$, and $Z_2$ all have non-empty finite fixed point sets.
\end{theorem}

We can state a combinatorial statement of our main result, which classifies the fixed point data of a 6-dimensional compact oriented $S^1$-manifold with isolated fixed points, roughly as follows.

\begin{theo}[\textbf{Theorem \ref{t62}}] \label{t14}
For a circle action on a 6-dimensional compact oriented manifold $M$ with isolated fixed points whose isotroy submanifolds are orientable, we can successively apply a combination of 5 types of operations to convert the fixed point data of $M$ to the empty collection. There is a definite procedure that ends in a finite number of steps.
\end{theo}

For the precise statement of Theorem \ref{t14}, see Theorem \ref{t62}. Theorem \ref{t14} means that given any such manifold $M$, we must be able to perform one of the 5 operations to the fixed point data of $M$. After performing one of those, again we can perform one of those (unless we already have the empty collection), and so on until we get the empty collection.

The classification results of this paper are hints for classifications of torus actions on other types of 6-manifolds, such as (almost) complex manifolds and symplectic manifolds. For a classification of a torus action on a manifold that has fixed points, the fixed point data is a necessary information to know; if for instance we want to classify an action on an (almost) complex 6-manifold $M$ with finite fixed point set, one can give a multigraph for an almost complex manifold (a directed labeled multigraph instead of a signed labeled multigraph; see Definition \ref{d410}), convert it to a signed labeled multigraph regarding $M$ as an oriented manifold, and check if it satisfies the conclusion of Theorem \ref{t61}. Or, we can check if complex weights at the fixed points satisfy the conclusion of Theorem \ref{t62} when converted to the fixed point data as an oriented manifold $M$.

The structure of this paper is as follows. In section \ref{s2}, we give necessary background. In Section \ref{s3}, we define a multigraph encoding the fixed point data of a torus action on an oriented manifold with isolated fixed points, and show that such a multigraph always exists. In Section \ref{s4}, we define a notion of a connected sum of a multigraph and show it behaves well with a connected sum of manifolds at fixed points. In Section \ref{s5}, we classify such a multigraph if the dimension of the manifold is 4, with a classification of GKM graphs of 4-dimensional GKM manifolds as a corollary. In Section \ref{s6}, we give precise statements of Theorems \ref{t11} and \ref{t14}. In Section \ref{s7} we give examples of torus actions on 6-dimensional manifolds and their multigraphs, which are standard (minimal) models for our results. In Section \ref{s8} we discuss a relation of fixed point data of fixed points in the isotropy submanifold. After dealing with some base cases in Section \ref{s9}, in Section \ref{s10} we prove Theorems \ref{t11}, \ref{t12}, and \ref{t14} altogether.

\section{Background} \label{s2}

Let a torus $T^k$ act on a $2n$-dimensional oriented manifold $M$. Let $p$ be an isolated fixed point. The tangent space $T_pM$ of $p$ at $M$ decomposes into the sum of $n$ complex 1-dimensional vector spaces $L_{p,1}$, $\cdots$, $L_{p,n}$, where on each $L_{p,i}$ the torus acts by multiplication by $g^{w_{p,i}}$ for some non-zero element $w_{p,i}$ of $\mathbb{Z}^k$. These elements $w_{p,1}$, $\cdots$, $w_{p,n}$ are called \textbf{weights} at $p$. The sign of each weight is not well-defined, but the product of the signs of the weights is well-defined. For each $i$, choose an orientation of $L_{p,i}$, that is, choose a sign of $w_{p,n}$. The tangent space $T_pM$ has two orientations, one induced from the orientation on $M$ and the other induced from the orientation on the representation space $L_{p,1} \oplus \cdots \oplus L_{p,n}$ with the chosen orientations on $L_{p,i}$'s. We define $\epsilon_M(p)$ to be $+1$ if these orientations agree and $-1$ otherwise, and call it \textbf{sign} of $p$. We shall use $\epsilon(p)$ instead of $\epsilon_M(p)$ if there is no confusion.

The sign $\epsilon(p)$ of $p$ depends on the choice of orientations of $L_{p,i}$'s; reversing the orientation of some $L_{p,i}$ reverses the sign of the corresponding weight $w_{p,i}$ by its negative $-w_{p,i}$, and this results in changing the sign $\epsilon(p)$ of $p$ by its negative $-\epsilon(p)$. Also note that if the $T^k$-action on $M$ is effective, $w_{p,1},\cdots,w_{p,n}$ span $\mathbb{Z}^k$.

We define an ordered pair $(\epsilon(p),\{w_{p,1},\cdots,w_{p,n}\})$, where $\epsilon(p)$ is the sign of $p$ and $\{w_{p,1},\cdots,w_{p,n}\}$ is a multiset of the weights at $p$.

For an ordered pair $(a,\{a_1,\cdots,a_n\})$ where $a \in \{-1,+1\}$ and $\{a_1,\cdots,a_n\}$ is a multiset of non-zero elements of $\mathbb{Z}^k$, we define an equivalence relation as follows:
\begin{itemize}
\item $(a,\{a_1,\cdots,a_i,\cdots,a_n\})$ is equivalent to $(-a,\{a_1,\cdots,-a_i,\cdots,a_n\})$.
\end{itemize}

For instance, $(+1,\{1,2,3\})$ is equivalent to $(-1,\{1,2,-3\})$ when $k=1$. Because the second component is a multiset, we have $(-1,\{1,2,-3\})=(-1,\{1,-3,2\})$.

Let $p$ be a fixed point and let $w_{p,1}$, $\cdots$, $w_{p,n}$ be the weights at $p$. Let $\epsilon(p)=\pm$ be the sign of $p$ corresponding to this choice of weights $w_{p,1},\cdots,w_{p,n}$ (this choice on orientations of $L_{p,1}\oplus \cdots \oplus L_{p,n}$). Suppose we replace some $w_{p,i}$ with $-w_{p,i}$ (suppose we reverse the orientation of $L_{p,i}$). Then this reverses the sign of $p$. As oriented $T^k$-representations, $(\pm,\{w_{p,1},\cdots,w_{p,i},\cdots,w_{p,n}\})$ is isomorphic to $(\mp,\{w_{p,1},\cdots,-w_{p,i},\cdots,w_{p,n}\})$. Therefore, the above definition of equivalence relation is natural.

\begin{Definition} \label{d21}
Let a torus $T^k$ act on a $2n$-dimensional oriented manifold $M$ with isolated fixed points. 
\begin{enumerate}[(1)]
\item For an isolated fixed point $p$, the \textbf{fixed point data} of $p$, denoted by $\Sigma_p$, is the equivalence class $[(\epsilon(p), \{w_{p,1},\cdots,w_{p,n}\})]$. For simplicity, we write $[\epsilon(p), w_{p,1},\cdots,w_{p,n}]$ instead of $[(\epsilon(p), \{w_{p,1},\cdots,w_{p,n}\})]$ if this causes no confusion. Moreover, when we write $\epsilon(p)$ inside the fixed point data of $p$, we shall omit 1 and only write its sign $\pm$. 
\item We define $-\Sigma_p$ to be the equivalence class of $[-\epsilon(p),w_{p,1},\cdots,w_{p,n}]$, where $\Sigma_p=[\epsilon(p),w_{p,1},\cdots,w_{p,n}]$. 
\item We define the \textbf{fixed point data} of $M$, denoted by $\Sigma_M$, to be a collection of the equivalence classes of fixed point daum of all fixed points of $M$. That is,
\begin{center}
$\displaystyle \Sigma_M=\cup_{p \in M^{T^k}} \Sigma_p$.
\end{center}
\end{enumerate}
\end{Definition}

\begin{exa}[The $2n$-dimensional sphere $S^6$] \label{e1}
Let $n$ be a positive integer and let $k \in \{1,\cdots,n\}$. Let $a_1$, $\cdots$, $a_n$ be non-zero elements in $\mathbb{Z}^k$. Let a torus $T^k$ act on $S^{2n}$ by
\begin{center}
$g \cdot (z_1,\cdots,z_n,x)=(g^{a_1} z_1, \cdots, g^{a_n} z_n,x)$
\end{center}
for all $g \in T^k \subset \mathbb{C}^k$ and for all $(z_1,\cdots,z_n,x) \in S^{2n}$, where 
\begin{center}
$S^{2n}=\{(z_1,\cdots,z_n,x) \in \mathbb{C}^n \times \mathbb{R} \, : \, x^2+\sum_{i=1}^n |z_i|^2=1\}$.
\end{center}
The action has two fixed points $q_1=(0,\cdots,0,1)$ and $q_2=(0,\cdots,0,-1)$. The weights at $q_i$ are both $\{a_1,\cdots,a_n\}$, and $\epsilon(q_1)=-\epsilon(q_2)=1$. The fixed point data of this action on $S^{2n}$ is 
\begin{center}
$[+,a_1,\cdots,a_n]$, $[-,a_1,\cdots,a_n]$.
\end{center}
\end{exa}

For an action of a group $G$ on a manifold $M$, we denote by $M^G$ its fixed point set, the set of points in $M$ that are fixed by the $G$-action on $M$. That is,
\begin{center}
$M^G=\{m \in M \, | \, g \cdot m=m\textrm{ for all }g \in G\}$.
\end{center}

Let $w=(w_1,\cdots,w_k)$ be an element of $\mathbb{Z}^k$. We define $\ker w$ to be the subgroup of $T^k$ whose elements fix $w$. That is,
\begin{center}
$\ker w=\{g=(g_1,\cdots,g_k) \in T^k \subset \mathbb{C}^k \, | \, g^w:=g_1^{w_1} \cdots g_k^{w_k}=1\}$.
\end{center}

Let a torus $T^k$ act effectively on a $2n$-dimensional compact oriented manifold $M$ with isolated fixed points. Let $w$ be a weight that occurs at some fixed point, such that $\ker w$ is a non-trivial subset of $T^k$. Let $F$ be a component of $M^{\ker w}$ such that $F \cap M^{T^k} \neq \emptyset$. By Lemma \ref{l29} below, $F$ is orientable; hence the normal bundle $NF$ is also orientable. Choose an orientation of $F$ and that of $NF$, so that the induced orientation on $TF \oplus NF$ agrees with the orientation of $M$. Let $p\in F \cap M^{T^k}$ be a $T^k$-fixed point. By permuting $L_{p,i}$'s, we may let
\begin{center}
$T_pM=L_{p,1} \oplus \cdots \oplus L_{p,m} \oplus L_{p,m+1} \oplus \cdots \oplus L_{p,n}$,
\end{center}
where $T_pF=L_{p,1} \oplus \cdots \oplus L_{p,m}$ and $N_pF=L_{p,m+1} \oplus \cdots \oplus L_{p,n}$, and the torus $T^k$ acts on each $L_{p,i}$ with weight $w_{p,i}$. 

\begin{Definition} \label{d23} 
\begin{enumerate}[(1)]
\item $\epsilon_F(p)=+1$ if the orientation on $F$ agrees with the orientation on $L_{p,1}\oplus \cdots \oplus L_{p,m}$, and $\epsilon_F(p)=-1$ otherwise.
\item $\epsilon_N(p)=+1$ if the orientation on $NF$ agrees with the orientation on $L_{p,m+1}\oplus \cdots \oplus L_{p,n}$, and $\epsilon_N(p)=-1$ otherwise.
\end{enumerate} 
\end{Definition}

By definition, $\epsilon(p)=\epsilon_F(p) \cdot \epsilon_N(p)$. As $\epsilon(p)$ does, the numbers $\epsilon_F(p)$ and $\epsilon_N(p)$ depend on the choice of the orientations on $L_{p,i}$'s. We need these terminologies for instance, to define a multigraph encoding the fixed point data of $M$; see Definition \ref{d32} and thereafter.

Let a torus $T^k$ act on a manifold $M$. The \textbf{equivariant cohomology} of $M$ is 
\begin{center}
$\displaystyle H_{T^k}^{*}(M)=H^{*}(M \times_{T^k} ET^k)$, 
\end{center}
where $ET^k$ is a contractible space on which $T^k$ acts freely. 
Suppose that $M$ is oriented and compact. The projection map $\pi : M \times_{T^k} ET^k \rightarrow BT^k$ induces a natural push-forward map
\begin{center}
$\displaystyle \pi_{*} : H_{S^{1}}^{i}(M;\mathbb{Z}) \rightarrow H^{i-\dim M}(BT^k;\mathbb{Z})$
\end{center}
for all $i \in \mathbb{Z}$, where $BT^k$ is the classifying space of $T^k$. This map is given by integration over the fiber $M$, and is also denoted by $\int_M$.

\begin{theo} \emph{[Atiyah-Bott-Berline-Vergne localization theorem]} \cite{AB} \label{t24}
Let a torus $T^k$ act on a compact oriented manifold $M$. Let $ \alpha \in H_{T^k}^{\ast}(M;\mathbb{Q})$. As an element of $\mathbb{Q}(t)$,
\begin{center}
$\displaystyle \int_{M} \alpha = \sum_{F \subset M^{T^k}} \int_{F} \frac{\alpha|_{F}}{e_{T^k}(NF)}$,
\end{center}
where the sum is taken over all fixed components, and $e_{T^k}(NF)$ denotes the equivariant Euler class of the normal bundle to $F$ in $M$.
\end{theo}

For a compact oriented $S^1$-manifold with isolated fixed points, the Atiyah-Singer index theorem yields the following formula.

\begin{theo} \emph{[Atiyah-Singer index theorem]} \cite{AS} \label{t25} Let the circle group $S^1$ act on a $2n$-dimensional compact oriented manifold $M$ with isolated fixed points. Then the signature of $M$ is
\begin{center}
$\displaystyle \textrm{sign}(M) = \sum_{p \in M^{S^1}} \epsilon(p) \prod_{i=1}^{n} \frac{1+t^{w_{p,i}}}{1-t^{w_{p,i}}}$
\end{center}
for all indeterminate $t$, and is a constant. \end{theo}

As an application of Theorem \ref{t25}, we obtain the following result.

\begin{pro} \label{p26}
Let the circle act on a compact oriented manifold $M$ with isolated fixed points. Suppose that $\dim M \equiv 2 \mod 4$. Then the signature of $M$ vanishes. For each fixed point $p$, by reversing the orientation of each $L_{p,i}$ if necessary, where $T_pM=L_{p,1}\oplus \cdots \oplus L_{p,n}$, assume that all weights $w_{p,i}$ are positive. Then the number of fixed points $p$ with $\epsilon(p)=+1$ and the number of fixed point $p$ with $\epsilon(p)=-1$ are equal.
\end{pro}

\begin{proof}
By the Hirzebruch signature theorem \cite{H}, the signature of $M$ is equal to the $L$-genus of $M$. Since $\dim M \neq 0 \mod 4$, the $L$-genus of $M$ vanishes; thus $\mathrm{sign}(M)=0$. Taking $t=0$ in Theorem \ref{t25},
\begin{center}
$\displaystyle \textrm{sign}(M) = \sum_{p \in M^{S^1}} \epsilon(p)$
\end{center}
and this proposition follows. \end{proof}

Proposition \ref{p26} implies that if $\dim M \equiv 2 \mod 4$, then the number of fixed points must be even.

\begin{cor} \label{c27} Let a torus $T^k$ act on a compact oriented manifold $M$. If the number of fixed points is odd, the dimension of $M$ is divisible by four. \end{cor}

There is a more elementary way to prove Corollary \ref{c27}; for a torus action on a compact manifold, its Euler number is equal to the sum of the Euler numbers of its fixed components \cite{K}. Then Poincar\'e duality implies Corollary \ref{c27}.



Comparing the coefficients of the smallest positive weight in Theorem \ref{t25} (after we choose all weights to be positive), the following lemma holds.

\begin{lem} \label{l28} \cite{J1} Let the circle act on a $2n$-dimensional compact oriented manifold $M$ with isolated fixed points. For each fixed point $p$, by reversing the orientation of each $L_{p,i}$ if necessary, where $T_pM=L_{p,1}\oplus \cdots \oplus L_{p,n}$, assume that all weights $w_{p,i}$ are positive. Let $w$ be the smallest positive weight that occurs among all the fixed points. Then the number of times weight $w$ occurs at fixed points of sign $+1$, counted with multiplicity, is equal to the number of times weight $w$ occurs at fixed points of sign $-1$, counted with multiplicity. That is,
\begin{center}
$\displaystyle \sum_{p \in M^{S^1},\epsilon(p)=+1} N_p(w)=\sum_{p \in M^{S^1},\epsilon(p)=-1} N_p(w)$,
\end{center}
where $N_p(w)=|\{i : w_{p,i}=w\}|$ is the number of times weight $w$ occurs at $p$.
\end{lem}

We need the below lemma to associate a multigraph to a torus action on an oriented manifold with isolated fixed points, in which we extend Lemma \ref{l28} to torus actions.

\begin{lem} \label{l210}
Let a torus $T^k$ act on a compact oriented manifold $M$ with isolated fixed points whose isotropy submanifolds are orientable. Let $w$ be a weight at some fixed point $p$ and $F$ a connected component of $M^{\ker w}$ containing $p$. Reversing the orientation of each $L_{q,i}$ if necessary, assume that all weights in $T_qF$ are positive multiples of $w$ for all $q \in F \cap M^{T^k}$. Then
\begin{center}
$\displaystyle \sum_{q \in F \cap M^{T^k},\epsilon_F(q)=+1} N_q(w)=\sum_{q \in F \cap M^{T^k},\epsilon_F(q)=-1} N_q(w)$,
\end{center}
where $F$ is chosen an orientation, and $N_q(w)=|\{i : w_{p,i}=w\}|$ is the multiplicity of $w$ in $T_qF$.
\end{lem}

\begin{proof}
Choose an orientation of $F$ as it is orientable by the assumption.
The $T^k$-action on $M$ restricts to act on $F$, and the fixed point set of this action on $F$ is equal to $F \cap M^{T^k}$, that is, $F^{T^k}=F \cap M^{T^k}$. Let $q \in F^{T^k}$. By definition of $\ker w$, every weight in $T_qF$ is an integer multiple of $w$. Replacing a weight with its negative (reversing of the orientation of $L_{q,i}$) if necessary, we may assume that every weight in $T_qF$ is a positive multiple of $w$.

Let $\xi \in (\mathbb{Z}^k \cap \mathfrak{t}) \setminus \ker w$ be a non-zero integral element of the Lie algebra $\mathfrak{t}$ of the torus $T^k$ such that $\langle w_{q,i}, \xi \rangle \neq 0$ for all weights $w_{q,i}$; such an element $\xi$ exists because $M$ is compact and hence there is a finite number of orbit types. If $\langle w,\xi \rangle$ is negative, we replace $\xi$ by $-\xi$ so that $\langle w,\xi \rangle>0$. Let $S$ be a subcircle of $T^k$ generated by $\xi$. Then the $T^k$-fixed point set of $M$ is equal to the $S$-fixed point set of $M$, that is, $M^{T^k}=M^S$; also $F^{T^k}=F^S$. Moreover, if $q \in F^{T^k}$ has a $T^k$-weight $aw$ for some positive integer $a$, the corresponding $S$-weight is $\langle aw, \xi \rangle$, which is a positive integer.

For $q \in F^{T^k}$, because every $T^k$-weight in $T_qF$ is a positive integer multiple of $w$, every $S$-weight in $T_qF$ is a positive integer multiple of $\langle w,\xi \rangle$. Let $x:=\langle w, \xi \rangle$. Then $x$ is the smallest positive weight for the $S$-action on $F$. Moreover, since we have chosen weights so that every $T^k$-weight in $T_qF$ is a positive multiple $aw$ of $w$ and its corresponding $S$-weight is a positive muleiple $ax$ of $x$, the $T^k$-action and the $S$-action give the same sign $\epsilon_F(q)$ to $q$.

Because the fixed point $p$ has $T^k$-weight $w$, $p$ has $S$-weight $x$. Applying Lemma \ref{l28} to the smallest positive weight $x$ for the $S$-action on $F$,
\begin{center}
$\displaystyle \sum_{p \in F^{S},\epsilon_F(p)=+1} N_p(x)=\sum_{p \in F^{S},\epsilon_F(p)=-1} N_p(x)$.
\end{center}
The $T^k$-weight $w$ only comes from the $S$-weight $x$. Therefore, this implies that
\begin{center}
$\displaystyle \sum_{p \in F^{T^k},\epsilon_F(p)=+1} N_p(w)=\sum_{p \in F^{T^k},\epsilon_F(p)=-1} N_p(w)$.
\end{center}
Hence the lemma holds. \end{proof}

It is well known that if a 2-dimensional compact connected oriented manifold admits a torus action with a fixed point, it must be the 2-sphere and there are exactly 2 fixed points; see \cite[Lemma 2.12]{J1} for a proof.

\begin{lem} \label{l211}
Let a torus $T^k$ act non-trivially on a 2-dimensional compact connected oriented manifold $M$. If there is a fixed point, $M$ is the 2-sphere and the action on $M$ has exactly two fixed points.
\end{lem}

\section{Multigraphs for torus actions} \label{s3}

In this section, we show, as for a circle action, that for a torus action on a compact oriented manifold with isolated fixed points there exists a multigraph that encodes the fixed point data of the manifold. For this, we need several terminologies.

\begin{Definition} \label{d31}
Let $k \in \mathbb{N}^+$. A \textbf{(signed labeled $k$-)multigraph} consists of
\begin{enumerate}
\item a set $V$ of vertices,
\item a set $E$ of edges,
\item a map $\epsilon:V \to \{-1,+1\}$, and
\item a map $w:E \to \mathbb{Z}^k$.
\end{enumerate}
Let $\Gamma$ be a signed labeled $k$-multigraph. For each vertex $v \in V$, the image $\epsilon(v)$ is called \textbf{sign} of $v$. For each edge $e \in E$, the image $w(e)$ is called \textbf{label} of $e$. For a vertex $v$, we define the \textbf{fixed point data} of $v$, denoted by $\Sigma_v$, to be the equivalence class $[\epsilon(v),w(e_1),\cdots,w(e_n)]$, where $e_1$, $\cdots$, $e_n$ are the edges of $v$. If an edge $e$ has vertices $p$ and $q$, we say that $(p,q)$ is $w(e)$-edge.
\end{Definition}

\begin{Definition} \label{d32}
Let a torus $T^k$ act on a compact oriented manifold $M$ with isolated fixed points. We say that a (signed labeled $k$-)multigraph \textbf{describes} $M$ if the following hold.
\begin{enumerate}
\item The vertex set of the multigraph is equal to the fixed point set $M^{T^k}$.
\item The fixed point data of a vertex $p$ is equal to the fixed point data of the corresponding fixed point $p$.
\item For each edge $e$, the two vertices of $e$ lie in the same connected component of the isotropy submanifold $M^{\ker w(e)}$.
\end{enumerate}
Let $\Gamma$ be a signed labeled $k$-multigraph describing $M$. We say that $\Gamma$ satisfies \textbf{Property A} if for every edge $e$, its vertices $p_1$ and $p_2$ satisfy $\epsilon_F(p_1)=-\epsilon_F(p_2)$, where $F$ is a connected component of $M^{\ker w(e)}$ containing $p_1$ and $p_2$, $F$ is chosen any orientation, and all weights in $T_pF$ are chosen to be positive multiples of $w(e)$, for all $p \in F \cap M^{T^k}$.
\end{Definition}

Property A looks technical at first glance, but is natural and seems to be an appropriate property to impose, as we shall see hereafter, because multigraphs with Property A behave well with various operations on manifolds, such as connected sum, self connected sum, and blow up.

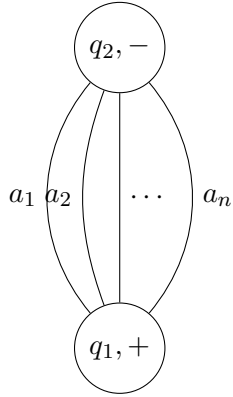
\begin{figure}
\centering
\begin{tikzpicture}[state/.style ={circle, draw}]
\node[state] (a) at (0,0) {$q_1,+$};
\node[state] (b) at (0,4) {$q_2,-$};
\path (a) [bend left =40]edge node[left] {$a_1$} (b);
\path (a) [bend left =20]edge node[left] {$a_2$} (b);
\path (a) edge node[right] {$\cdots$} (b);
\path (b) [bend left =40]edge node[right] {$a_n$} (a);
\end{tikzpicture}
\caption{Multigraph describing $S^{2n}$ in Example \ref{e1}} \label{f2}
\end{figure}

For instance, Figure \ref{f2} describes the action on $S^{2n}$ in Example \ref{e1}. For an oriented $S^1$-manifold with isolated fixed points, the author proved that there is a multigraph describing it \cite{J1}. Unfortunately, its proof holds if isotropy submanifolds are orientable.

To prove our main results, we need such a multigraph to satisfy Property A. It was not explicit but the multigraph the author gave in \cite{J1} satisfied Property A. Because this is important, we shall correct and review its proof.

\begin{pro} \label{p33} \cite{J1}
Let the circle act on a compact oriented manifold $M$ with isolated fixed points whose isotropy submanifolds are orientable. Then there exists a signed labeled multigraph describing $M$ that satisfies Property A. In particular, the multigraph has no self-loops.
\end{pro}

\begin{proof}[\textbf{Review of the proof in \cite{J1}}] For each fixed point $p$, choose an orientation of $L_{p,i}$ so that each weight $w_{p,i}$ at $p$ is positive. To each fixed point $p$, we assign a vertex, also denoted by $p$, and assign a sign $\epsilon(p)$ to the vertex $p$. Let $w$ be a positive integer. Consider the set $M^{\ker w}$ of points in $M$ that are fixed by the $\ker w$-action. Let $F$ be a component of $M^{\ker w}$ which contains an $S^1$-fixed point, that is, $F \cap M^{S^1} \neq \emptyset$. Choose an orientation of $F$ as it is orientable by the assumption. The $S^1$-action on $M$ restricts to act on $F$, and the smallest positive weight of the $S^1$-action on $F$ is $w$. Applying Lemma \ref{l28} to the $S^1$-action on $F$,
\begin{center}
$\displaystyle \sum_{p \in F^{S^1},\epsilon_F(p)=+1} N_p(w)=\sum_{p \in F^{S^1},\epsilon_F(p)=-1} N_p(w)$,
\end{center}
where $N_p(w)=|\{i  :  w_{p,i}=w\}|$ is the number of times weight $w$ occurs at $p$. Therefore, if a fixed point $p_1 \in F^{S^1}$ with $\epsilon_F(p_1)=+1$ has weight $w$ and a fixed point $p_2 \in F^{S^1}$ with $\epsilon_F(p_2)=-1$ has weight $w$, then we can draw an edge between the corresponding vertices $p_1$ and $p_2$, giving a label $w$ to the edge, to exhaust all weights $w$ that appear at fixed points in $F^{S^1}$. In particular, this does not cause any self-loops.

Repeating the above argument for each positive integer $w$ and each component of $M^{\mathbb{Z}_w}$, the proposition follows. \end{proof}

We extend Proposition \ref{p33} to torus actions.

\begin{pro} \label{p34}
Let a torus $T^k$ act on a compact oriented manifold $M$ with isolated fixed points whose isotropy submanifolds are orientable. Then there exists a (signed labeled $k$-)multigraph describing $M$ that satisfies Property A. In particular, the multigraph has no self-loops.
\end{pro}

\begin{proof}
Let $p$ be a fixed point and let $w$ be a weight at $p$. Let $F$ be a connected component of $M^{\ker w}$ which contains $p$. Choose an orientation of $F$ as it is orientable by the assumption.

The $T^k$-action on $M$ restricts to act on $F$, and the fixed point set $F^{T^k}$ of this action on $F$ is equal to $F \cap M^{T^k}$. Let $q \in F^{T^k}$. By definition of $\ker w$, every weight in $T_qF$ is a multiple of $w$. Replacing a weight with its negative (reversing of the orientation of $L_{q,i}$) if necessary, assume that every weight in $T_qF$ is a positive multiple of $w$. By Lemma \ref{l210},
\begin{center}
$\displaystyle \sum_{q \in F \cap M^{T^k},\epsilon_F(q)=+1} N_q(w)=\sum_{q \in F \cap M^{T^k},\epsilon_F(q)=-1} N_q(w)$,
\end{center}
where $N_q(w)=|\{i : w_{p,i}=w\}|$ is the multiplicity of $w$ in $T_qF$.
Therefore, we can exhaust all weights $w$ that appear at fixed points in $F^{T^k}$ as follows: if a fixed point $p_1 \in F^{T^k}$ with $\epsilon_F(p_1)=+1$ has weight $w$ and a fixed point $p_2 \in F^{T^k}$ with $\epsilon_F(p_2)=-1$ has weight $w$, then we draw an edge between the corresponding vertices $p_1$ and $p_2$, giving a label $w$ to the edge. In particular, this does not cause any self-loops.

Repeating the above argument for all weights $w$ over all of the fixed points $p$, the proposition follows. In particular, the multigraph obtained satisfies Property A, because for any edge $e$ with label $w(e)$ its two vertices $p_1$ and $p_2$ satisfy $\epsilon_F(p_1)=-\epsilon_F(p_2)$ where $F$ is a connected component of $M^{\ker w(e)}$ containing $p_1$ and $p_2$, $F$ is given any orientation, and all weights in $F$ are chosen to be positive multiples of $w(e)$. \end{proof}

Given a multigraph describing a compact oriented $T^k$-manifold with isolated fixed points, if we replace the label and the signs of an edge by their negatives, the resulting multigraph still describes the manifold. Accordingly, we need a corresponding notion for a multigraph.

\begin{Definition}[\textbf{Reversing edge}] \label{d35}
Let $\Gamma$ be a signed labeled multigraph. We define an operation, called \textbf{reversing an edge $e$}, as follows.
\begin{enumerate}
\item Replace the label $w(e)$ of the edge $e$ with its negative $-w(e)$.
\item For each vertex $p$ of the edge $e$, replace the sign $\epsilon(p)$ of $p$ with $-\epsilon(p)$.
\end{enumerate}
\end{Definition}

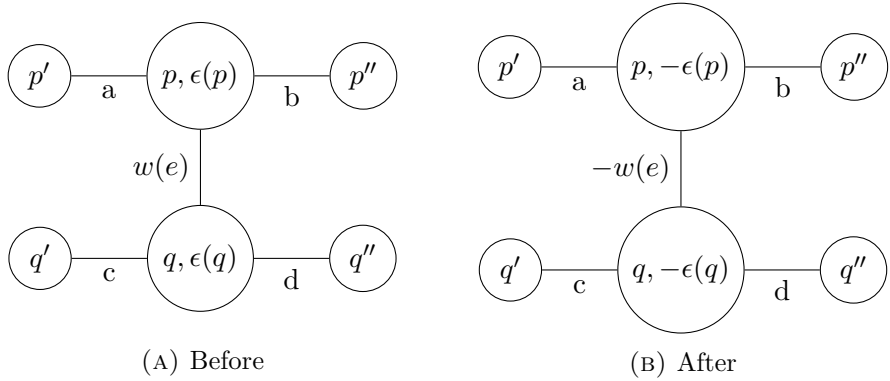
\begin{figure}
\centering
\begin{subfigure}[b][5cm][s]{.49\textwidth}
\centering
\vfill
\begin{tikzpicture}[state/.style ={circle, draw}]
\node[state] (a) {$p,\epsilon(p)$};
\node[state] (b) [left=of a] {$p'$};
\node[state] (c) [right=of a] {$p''$};
\node[state] (d) [below=of a]{$q,\epsilon(q)$};
\node[state] (e) [left=of d] {$q'$};
\node[state] (f) [right=of d] {$q''$};
\path (a) edge node[below] {a} (b);
\path (a) edge node [below] {b} (c);
\path (a) edge node[left] {$w(e)$} (d);
\path (d) edge node [below] {c} (e);
\path (d) edge node [below] {d} (f);
\end{tikzpicture}
\vfill
\caption{Before}\label{r1}
\end{subfigure}
\begin{subfigure}[b][5cm][s]{.50\textwidth}
\centering
\vfill
\begin{tikzpicture}[state/.style ={circle, draw}]
\node[state] (a) {$p,-\epsilon(p)$};
\node[state] (b) [left=of a] {$p'$};
\node[state] (c) [right=of a] {$p''$};
\node[state] (d) [below=of a]{$q,-\epsilon(q)$};
\node[state] (e) [left=of d] {$q'$};
\node[state] (f) [right=of d] {$q''$};
\path (a) edge node[below] {a} (b);
\path (a) edge node [below] {b} (c);
\path (a) edge node[left] {$-w(e)$} (d);
\path (d) edge node [below] {c} (e);
\path (d) edge node [below] {d} (f);
\end{tikzpicture}
\vfill
\caption{After}\label{r2}
\vspace{\baselineskip}
\end{subfigure}\qquad
\caption{Reversing an edge $e$ between $p$ and $q$: signs of other vertices and other edges at $p',p'',q',q''$ remain the same}\label{r}
\end{figure}

\begin{lem} \label{l36}
Let a torus $T^k$ act on a compact oriented manifold $M$ with isolated fixed points, whose isotropy submanifolds are orientable. Suppose that a multigraph $\Gamma$ describes $M$. Let $e$ be an edge of $\Gamma$. Let $\Gamma'$ be a multigraph obtained from $\Gamma$ by reversing the edge $e$. Then $\Gamma'$ also describes $M$. If $\Gamma$ satisfies Property A, so does $\Gamma'$.
\end{lem}

\begin{proof}
Let $p_1$ and $p_2$ be the vertices (fixed points) of $e$. Let $e_2$, $\cdots$, $e_n$ be the other edges at $p_1$. Because $\Gamma$ describes $M$, the fixed point data of $p_1$ is $[\epsilon(p_1),w(e),w(e_2),\cdots,w(e_n)]$. On the other hand, $\Gamma'$ gives $p_1$ the fixed point data $[-\epsilon(p_1),-w(e),w(e_2),\cdots,w(e_n)]$, which is equivalent to $[\epsilon(p_1),w(e),w(e_2),\cdots,w(e_n)]$. Similarly, both $\Gamma$ and $\Gamma'$ give $p_2$ the same fixed point data. The signs of other vertices and the labels of others edges of $\Gamma$ remain the same in $\Gamma'$. Because $\Gamma$ describes $M$, $p_1$ and $p_2$ are in the same component $F$ of $M^{\ker w(e)}$. Since in $\Gamma'$ the vertices $p_1$ and $p_2$ are still connected by an edge $e'$ (which is the edge $e$ of $\Gamma$), it follows that $\Gamma'$ also describes $M$. 

Suppose that $\Gamma$ satisfies Property A. Choose an orientation of $F$ as is orientable by Lemma \ref{l28}. Choose all weights in $F$ to be positive multiples of $w(e)$. Let $[\epsilon_F(p_i),w_{p_i,1},\cdots,w_{p_i,m}]$ denote the fixed point data of $p_i$ for the $T^k$-action on $F$, where $\dim F=2m$ and $w_{p_i,1},\cdots,w_{p_i,m}$ are the weights in $T_{p_i}F$ and all positive multiples of $w(e)$. Then $\epsilon_F(p_1)=-\epsilon_F(p_2)$ because $\Gamma$ satisfies Property A.

Now the label of $e'$ in $\Gamma'$ is $-w(e)$. To check that $\Gamma'$ satisfies Property A, we choose all weights in $F$ to be positive multiples of $-w(e)$. Then the fixed point data of $p_i$ is $[(-1)^m \epsilon_F(p_i),-w_{p_i,1},\cdots,-w_{p_i,m}]$. Therefore, with this choice of weights (all weights in $T_pF$ are positive multiples of $-w(e)$), the new signs $(-1)^m \epsilon_F(p_1)$ and $(-1)^m \epsilon_F(p_2)$ of $p_1$ and $p_2$ still satisfy $(-1)^m \epsilon_F(p_1)=-(-1)^m \epsilon_F(p_2)$, because $\epsilon_F(p_1)=-\epsilon_F(p_2)$. This also holds if we reverse the orientation of $F$. Thus, $\Gamma'$ also satisfies Property A. \end{proof}

In fact, we may want to regard $\Gamma$ and $\Gamma'$ as the same multigraph, as they both describe $M$.

\begin{Definition}[\textbf{Exchanging edges}]
Let $\Gamma$ be a signed labeled multigraph. Suppose that an edge $e_1$ has label $a$ and vertices $p_1$ and $p_2$ and an edge $e_2$ has label $a$ and vertices $q_1$ and $q_2$.
We define an operation, called \textbf{exchanging edges}, as follows.
\begin{enumerate}
\item Remove the edges $e_1$ and $e_2$.
\item Draw an edge of label $a$ between $p_1$ and $q_2$, and draw an edge of label $a$ between $q_1$ and $p_2$.
\end{enumerate}
\end{Definition}

Figure \ref{f4} illustrates exchanging edges.

\begin{figure}
\centering
\begin{subfigure}[b][4cm][s]{.4\textwidth}
\centering
\vfill
\begin{tikzpicture}[state/.style ={circle, draw}]
\node[state] (a) {$p_1$};
\node[state] (b) [above=of a] {$p_2$};
\node[state] (d) [right=of a]{$q_1$};
\node[state] (e) [above=of d] {$q_2$};
\path (a) edge node[left] {a} (b);
\path (d) edge node [right] {a} (e);
\end{tikzpicture}
\vfill
\caption{Before}\label{f4-1}
\end{subfigure}
\begin{subfigure}[b][4cm][s]{.4\textwidth}
\centering
\vfill
\begin{tikzpicture}[state/.style ={circle, draw}]
\node[state] (a) {$p_1$};
\node[state] (b) [above=of a] {$p_2$};
\node[state] (d) [right=of a]{$q_1$};
\node[state] (e) [above=of d] {$q_2$};
\path (a) edge node[pos=.8, right] {a} (e);
\path (b) edge node[pos=.2, left] {a} (d);
\end{tikzpicture}
\vfill
\caption{After}\label{f4-2}
\vspace{\baselineskip}
\end{subfigure}\qquad
\caption{Exchanging edges between $p_1$ and $p_2$ and $q_1$ and $q_2$: signs of vertices and other edges of $p_1,p_2,q_1,q_2$ are omitted}\label{f4}
\end{figure}
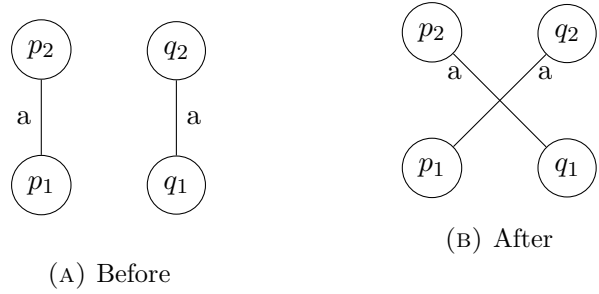

\begin{lem} \label{l38}
Let a torus $T^k$ act on a compact oriented manifold $M$ with isolated fixed points, whose isotropy submanifolds are orientable. Let $\Gamma$ be a multigraph describing $M$. Suppose that an edge $e_1$ has label $a$ and vertices $p_1$ and $p_2$, an edge $e_2$ has label $a$ and vertices $q_1$ and $q_2$, and $p_1$ and $q_1$ are in the same connected component $F$ of $M^{\ker a}$. Let $\Gamma'$ be a multigraph obtained from exchanging the edges $e_1$ and $e_2$. Then $\Gamma'$ also describes $M$.

Suppose that $\Gamma$ satisfies Property A and $\epsilon_F(p_1)=\epsilon_F(q_1)$, where $F$ is chosen any orientation and all weights in $F$ are chosen to be positive multiples of $a$. If we exchange the edges so that $(p_1,q_2)$ is $a$-edge and $(q_1,p_2)$ is $a$-edge in the resulting multigraph $\Gamma'$, then $\Gamma'$ also satisfies Property A.
\end{lem}

\begin{proof}
The multigraph $\Gamma'$ gives each vertices (fixed points) the same fixed point data as $\Gamma$. Because $(p_1,p_2)$ is $a$-edge, both $p_1$ and $p_2$ lie in $F$ and similarly $q_1$ and $q_2$ lie in $F$. By assumption, both $p_1$ and $q_1$ lie in $F$. Thus, the connected component $F$ of $M^{\ker a}$ contains all of $p_1,p_2,q_1,q_2$. Hence, $\Gamma'$ also describes $M$.

Since $\Gamma$ satisfies Property A and $(p_1,p_2)$ ($(q_1,q_2)$) is $a$-edge, $\epsilon_F(p_1)=-\epsilon_F(p_2)$ ($\epsilon_F(q_1)=-\epsilon_F(q_2)$) where $F$ is chosen any orientation and all weights in $T_pF$ are chosen to be all positive multiples of $a$ for all $p \in F \cap M^{T^k}$. 
With the same choice of weights and with any orientation on $F$, in $\Gamma'$ we have that $(p_1,q_2)$ ($(q_1,p_2)$) is $a$-edge and $\epsilon_F(p_1)=-\epsilon_F(q_2)$ ($\epsilon_F(q_1)=-\epsilon_F(p_2)$). Other vertices and edges satisfy Property A. Thus, $\Gamma'$ also satisfies Property A.
\end{proof}

While we introduce the notion of exchanging edges, we do not need to use in our proofs, because multigraphs satisfying Property A are somehow correct ones to study torus actions on oriented manifolds with isolated fixed points, in a sense that we do not need to exchange edges.

Kosniowski proved that if the circle acts on a compact oriented manifold with exactly two fixed points, they have the opposite fixed point data \cite{Ko}; also see \cite{J1}. We extend this to torus actions.

\begin{lem} \label{l39}
Let a torus $T^k$ act on a $2n$-dimensional compact oriented manifold $M$ with isolated fixed points, where $n>0$. Suppose that weights at each fixed point are $\{\pm a_1,\cdots, \pm a_n\}$ for some $a_1,\cdots,a_n \in \mathbb{Z}^k \setminus \{0\}$. Then the number of fixed points with fixed point data $[+,a_1,\cdots,a_n]$ and that with fixed point data $[-,a_1,\cdots,a_n]$ are equal. 
\end{lem}

\begin{proof}
Replacing $-a_i$ with $a_i$ (reversing the orientation of $L_{p,i}$) if necessary, we may assume that weights at each fixed point are $\{a_1,\cdots,a_n\}$. By a dimensional reason that $\int_M$ is a map from $H_{T^k}^i(M;\mathbb{Z})$ to $H^{i - 2n}(BT^k;\mathbb{Z})$, the image of $1$ under the map $\int_M$ is zero;
\begin{center}
$\displaystyle \int_M 1=0$.
\end{center}
Applying Theorem \ref{t24} to the equivariant cohomology class 1,
\begin{center}
$\displaystyle \int_M 1=\sum_{p \in M^{T^k}} \int_p \frac{1}{e_{T^k}(T_pM)}=\sum_{p \in M^{T^k},\epsilon(p)=+1} \frac{1}{a_1 \cdots a_n} + \sum_{p \in M^{T^k},\epsilon(p)=-1} (-1) \frac{1}{a_1 \cdots a_n}$.
\end{center}
These prove the lemma. \end{proof}

\begin{theo} \label{t310}
Let a torus $T^k$ act on a compact oriented manifold with exactly two fixed points $p$ and $q$, whose isotropy submanifolds are orientable. Then $p$ and $q$ have the opposite fixed point data, that is, $\Sigma_p=-\Sigma_q$.
\end{theo}

\begin{proof}
By Proposition \ref{p34} there is a signed labeled $k$-multigraph describing $M$ that satisfies Property A. In particular, it has no self-loops. Thus, $p$ and $q$ have the same multiset of weights (up to sign). By Lemma \ref{l39}, the theorem follows.
\end{proof}

\section{Connected sum and blow up of manifold and multigraph}\label{s4}

We describe an equivariant connected sum of two oriented manifolds $M$ and $N$ equipped with torus actions at fixed points and a corresponding operation on multigraphs describing them. Unlike a tradition we will take equivariant connected sums at several fixed points $p_1$, $\cdots$, $p_k$ of $M$ and $q_1$, $\cdots$, $q_k$ of $N$, where for each $i$ the two fixed points $p_i$ and $q_i$ have the opposite fixed point data that $\Sigma_{p_i}=-\Sigma_{q_i}$.

Let a torus $T^k$ act on two $2n$-dimensional connected oriented manifolds $M$ and $N$ with isolated fixed points. Suppose that for $i=1,\cdots,k$, a fixed point $p_i \in M^{T^k}$ of $M$ and a fixed point $q_i \in N^{T^k}$ of $N$ have the opposite fixed point data, that is, $\Sigma_{p_i}=-\Sigma_{q_i}$. Let $[\epsilon_{M}(p_i),w_{p_i,1},\cdots,w_{p_i,n}]$ be the fixed point data of $p_i$. Then $q_i$ has the fixed point data $[-\epsilon_{M}(p_i),w_{p_i,1},\cdots,w_{p_i,n}]$.

For each $i$, there is an equivariant diffeomorphism $f_i$ ($g_i$) from a unit disk $D_{2n}$ in $\mathbb{C}^n$ to a neighborhood of $p_i$ ($q_i$), where the torus $T^k$ acts on $\mathbb{C}^n$ by
\begin{center}
$g \cdot (z_1,\cdots,z_n)=(g^{w_{p,1}}z_1,\cdots,g^{w_{p,n}}z_n)$
\end{center}
for all $g \in T^k \subset \mathbb{C}^k$ and for all $(z_1,\cdots,z_n) \in \mathbb{C}^n$.

\begin{Definition} \label{d41}
The \textbf{equivariant connected sum} of $M$ and $N$ (at $p_1$, $\cdots$, $p_k \in M$ and at $q_1$, $\cdots$, $q_k \in N$) is the quotient 
\begin{center}
$\displaystyle \{(M \setminus \cup_{i=1}^k f_i(0)\} \sqcup \{N \setminus \cup_{i=1}^k g_i(0)\}/\sim$,
\end{center}
where we identify $f_i(tu)$ with $g_i((1-t)u)$ for each $u \in \partial D_{2n}$ and each $0 < t < 1$, for each $i$.
\end{Definition}

This definition also holds if $M=N$, that is, a self connected sum. Let a torus $T^k$ act on a $2n$-dimensional connected oriented manifolds $M$ with isolated fixed points. Suppose that for $i=1,\cdots,k$, two fixed points $p_i$ and $q_i$ have the opposite fixed point data, that is, $\Sigma_{p_i}=-\Sigma_{q_i}$. Let $f_i$ and $g_i$ be the above equivariant diffeomorphisms.

\begin{Definition} \label{d42}
The \textbf{self equivariant connected sum} of $M$ (at $p_1$, $\cdots$, $p_k \in M$ and at $q_1$, $\cdots$, $q_k \in M$) is the quotient 
\begin{center}
$\displaystyle \left\{M \setminus \{(\cup_{i=1}^k f_i(0)) \cup (\cup_{i=1}^k g_i(0))\}\right\}/\sim$,
\end{center}
where we identify $f_i(tu)$ with $g_i((1-t)u)$ for each $u \in \partial D_{2n}$ and each $0 < t < 1$, for each $i$.
\end{Definition}

If we take an equivariant connected sum at $p_i$ of $M$ and $q_i$ of $M$ for all $i$, because $\epsilon_M(p_i)=-\epsilon_N(q_i)$ for $1 \leq i \leq k$ and each gluing map reverses orientation, we get another connected oriented $T^k$-manifold $P$ with fixed points $(M^{S^1} \setminus \{p_1,\cdots,p_k\}) \sqcup (N^{S^1} \setminus \{q_1,\cdots,q_k\})$. Moreover, the fixed point data of $P$ is $(\Sigma_M \setminus \cup_{i=1}^k \Sigma_{p_i}) \sqcup (\Sigma_N \setminus \cup_{i=1}^k \Sigma_{q_i})$.

\begin{lem} \label{l43}
Let $M$ and $N$ be two $2n$-dimensional compact connected oriented $T^k$-manifolds with isolated fixed points. Suppose that for $i=1,\cdots,k$, $p_i \in M^{S^1}$ and $q_i \in N^{S^1}$ satisfy $\Sigma_{p_i}=-\Sigma_{q_i}$. The equivariant connected sum of $M$ and $N$ at $p_1$, $\cdots$, $p_k$ and $q_1$, $\cdots$, $q_k$ is a $2n$-dimensional compact connected oriented $T^k$-manifold $P$ with isolated fixed points, whose fixed point set is $(M^{T^k} \setminus \{p_1,\cdots,p_k\}) \sqcup (N^{T^k} \setminus \{q_1,\cdots,q_k\})$ and fixed point data is $(\Sigma_M \setminus \cup_{i=1}^k \Sigma_{p_i}) \sqcup (\Sigma_N \setminus \cup_{i=1}^k \Sigma_{q_i})$. \end{lem}

We describe what an equivariant connected sum does to multigraphs describing oriented $T^k$-manifolds. 

\begin{Definition}[\textbf{Connected sum of two multigraphs}] \label{d44}
Let $\Gamma_1$ and $\Gamma_2$ be two $n$-valent signed labeled $k$-multigraphs without any self-loops. Suppose that $p \in \Gamma_1$ and $q \in \Gamma_2$ satisfy $\Sigma_p=-\Sigma_q$. 
Let $e_1,\cdots,e_n$ be the edges of $p$ and $e_1',\cdots,e_n'$ the edges of $q$. 
Suppose that for each $i$, the label of $e_i$ and that of $e_i'$ are equal; $w(e_i)=w(e_i')$. For each $i$, let $p_i$ and $q_i$ be the other vertex of $e_i$ and $e_i'$, respectively. 
The \textbf{connected sum} of $\Gamma_1$ and $\Gamma_2$ at $p$ and $q$ is another signed labeled multigraph obtained as follows.
\begin{enumerate}
\item Remove $p$ and $q$ and the edges of $p$ and $q$.
\item For each $i$, draw an edge between $p_i$ and $q_i$, giving label $w(e_i)$. 
\end{enumerate}
\end{Definition}

\begin{figure}
\centering
\begin{subfigure}[b][5cm][s]{.6\textwidth}
\centering
\vfill
\begin{tikzpicture}[state/.style ={circle, draw}]
\node[state] (a) {$p,+$};
\node[state] (b) [above left=of a] {$p_1$};
\node[state] (c) [left=of a] {$p_2$};
\node[state] (d) [below left=of a] {$p_3$};
\node[state] (e) [right=of a]{$q,-$};
\node[state] (f) [above right=of e] {$q_1$};
\node[state] (g) [right=of e] {$q_2$};
\node[state] (h) [below right=of e] {$q_3$};
\path (a) edge node[right] {$a$} (b);
\path (a) edge node [below] {$b$} (c);
\path (a) edge node [right] {$c$} (d);
\path (e) edge node[left] {$a$} (f);
\path (e) edge node [below] {$b$} (g);
\path (e) edge node [left] {$c$} (h);
\end{tikzpicture}
\vfill
\caption{Before}\label{f5-1}
\end{subfigure}
\begin{subfigure}[b][5cm][s]{.3\textwidth}
\centering
\vfill
\begin{tikzpicture}[state/.style ={circle, draw}]
\node[state] (a) {$p_1$};
\node[state] (b) [below=of a] {$p_2$};
\node[state] (c) [below=of b] {$p_3$};
\node[state] (d) [right=of a] {$q_1$};
\node[state] (e) [below=of d] {$q_2$};
\node[state] (f) [below=of e] {$q_3$};
\path (a) edge node[below] {$a$} (d);
\path (b) edge node [below] {$b$} (e);
\path (c) edge node [below] {$c$} (f);
\end{tikzpicture}
\vfill
\caption{After}\label{f5-2}
\vspace{\baselineskip}
\end{subfigure}\qquad
\caption{Connected sum of $\Gamma_1$ ($M$) and $\Gamma_2$ ($N$) at $p$ and $q$. This also illustrates a self connected sum at $p$ and $q$.}\label{f5}
\end{figure}
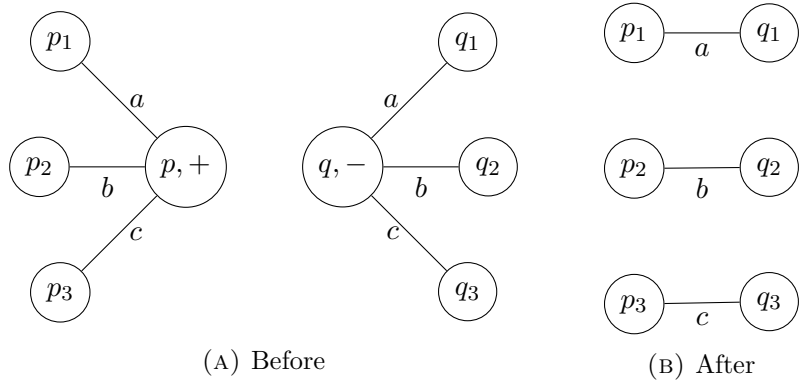

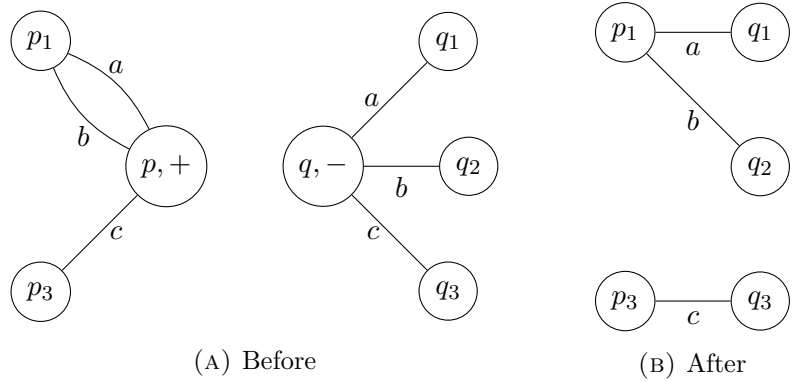
\begin{figure}
\centering
\begin{subfigure}[b][5cm][s]{.6\textwidth}
\centering
\vfill
\begin{tikzpicture}[state/.style ={circle, draw}]
\node[state] (a) {$p,+$};
\node[state] (b) [above left=of a] {$p_1$};
\node[state] (d) [below left=of a] {$p_3$};
\node[state] (e) [right=of a]{$q,-$};
\node[state] (f) [above right=of e] {$q_1$};
\node[state] (g) [right=of e] {$q_2$};
\node[state] (h) [below right=of e] {$q_3$};
\path (a) [bend left =20]edge node[below] {$b$} (b);
\path (a) [bend right =20]edge node[above] {$a$} (b);
\path (a) edge node [right] {$c$} (d);
\path (e) edge node[left] {$a$} (f);
\path (e) edge node [below] {$b$} (g);
\path (e) edge node [left] {$c$} (h);
\end{tikzpicture}
\vfill
\caption{Before}\label{f6-1}
\end{subfigure}
\begin{subfigure}[b][5cm][s]{.3\textwidth}
\centering
\vfill
\begin{tikzpicture}[state/.style ={circle, draw}]
\node[state] (a) {$p_1$};
\node[state] (d) [right=of a] {$q_1$};
\node[state] (e) [below=of d] {$q_2$};
\node[state] (f) [below=of e] {$q_3$};
\node[state] (c) [left=of f] {$p_3$};
\path (a) edge node[below] {$a$} (d);
\path (a) edge node[below] {$b$} (e);
\path (c) edge node [below] {$c$} (f);
\end{tikzpicture}
\vfill
\caption{After}\label{f6-2}
\vspace{\baselineskip}
\end{subfigure}\qquad
\caption{Connected sum of $\Gamma_1$ ($M$) and $\Gamma_2$ ($N$) at $p$ and $q$ when there are multiple edges. This also illustrates a self connected sum at $p$ and $q$}\label{f6}
\end{figure}

Figures \ref{f5} and \ref{f6} illustrate connected sums of two multigraphs.

\begin{Definition}[\textbf{(Self) connected sum of a multigraph}] \label{d45}
Let $\Gamma$ be an $n$-valent signed labeled $k$-multigraph without any self-loops. Suppose that $p \in \Gamma$ and $q \in \Gamma$ satisfy $\Sigma_p=-\Sigma_q$. 
Let $e_1,\cdots,e_n$ be the edges of $p$ and $e_1',\cdots,e_n'$ the edges of $q$. 
Suppose that for each $i$, $w(e_i)=w(e_i')=w_{p,i}$.
Suppose that there is no vertex $r$ such that there are an edge between $p$ and $r$ and an edge between $q$ and $r$ that have the same label.
The \textbf{(self) connected sum} of $\Gamma$ at $p$ and $q$ is another signed labeled multigraph obtained as follows.
\begin{enumerate}[(1)]
\item If $\Gamma$ has only two vertices $p$ and $q$, remove $p$ and $q$ and the edges of $p$ and $q$ to convert $\Gamma$ into the empty multigraph.
\item Suppose that $\Gamma$ has more than 2 vertices. Then
\begin{enumerate}[(a)]
\item Remove $p$ and $q$ and the edges of $p$ and $q$.
\item For each $i$, draw an edge between the other vertex of $e_i$ and the other vertex of $e_i'$, giving label $w(e_i)$.
\end{enumerate}
\end{enumerate}
\end{Definition}

\begin{Definition}
By a \textbf{connected sum} of multigraphs we mean any of Definitions \ref{d44} and \ref{d45}. 
\end{Definition}

\begin{figure}
\centering
\begin{subfigure}[b][5cm][s]{.49\textwidth}
\centering
\vfill
\begin{tikzpicture}[state/.style ={circle, draw}]
\node[state] (a) at (0,0) {$q,-$};
\node[state] (b) at (0,3) {$p,+$};
\path (a) [bend left =20]edge node[left] {$a$} (b);
\path (a) edge node[left] {$b$} (b);
\path (b) [bend left =20]edge node[right] {$c$} (a);
\end{tikzpicture}
\vfill
\caption{Before}\label{f7-1}
\end{subfigure}
\begin{subfigure}[b][5cm][s]{.49\textwidth}
\centering
\vfill
\begin{tikzpicture}[state/.style ={circle, draw}]
\end{tikzpicture}
\vfill
\caption{After}\label{f7-2}
\vspace{\baselineskip}
\end{subfigure}\qquad
\caption{Self connected sum at $p$ and $q$: When there are three edges between $p$ and $q$}\label{f7}
\end{figure}
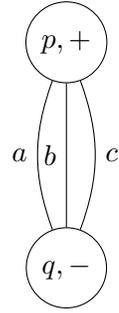

\begin{figure}
\centering
\begin{subfigure}[b][4cm][s]{.49\textwidth}
\centering
\vfill
\begin{tikzpicture}[state/.style ={circle, draw}]
\node[state] (a) {$p,+$};
\node[state] (b) [below=of a] {$q,-$};
\node[state] (c) [left=of a] {$p'$};
\node[state] (d) [left=of b] {$q'$};
\path (a) [bend left =20]edge node[right] {$a$} (b);
\path (b) [bend left =20]edge node[left] {$b$} (a);
\path (a) edge node[below] {$c$} (c);
\path (b) edge node[below] {$c$} (d);
\end{tikzpicture}
\vfill
\caption{Before}\label{f8-1}
\end{subfigure}
\begin{subfigure}[b][4cm][s]{.49\textwidth}
\centering
\vfill
\begin{tikzpicture}[state/.style ={circle, draw}]
\node[state] (a) {$p'$};
\node[state] (b) [below=of a] {$q'$};
\path (a) edge node[left] {$c$} (b);
\end{tikzpicture}
\vfill
\caption{After}\label{f8-2}
\vspace{\baselineskip}
\end{subfigure}\qquad
\caption{Self connected sum at $p$ and $q$: When there are two multiple edges between $p$ and $q$. Other edges at $p',q'$ are omitted}\label{f8}
\end{figure}
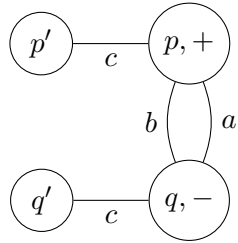
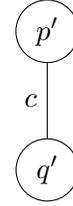

\begin{figure}
\centering
\begin{subfigure}[b][4cm][s]{.6\textwidth}
\centering
\vfill
\begin{tikzpicture}[state/.style ={circle, draw}]
\node[state] (a) {$p,+$};
\node[state] (b) [left=of a] {$p'$};
\node[state] (c) [right=of a] {$p''$};
\node[state] (d) [below=of a]{$q,-$};
\node[state] (e) [left=of d] {$q'$};
\node[state] (f) [right=of d] {$q''$};
\path (a) edge node[below] {$b$} (b);
\path (a) edge node [below] {$c$} (c);
\path (a) edge node[left] {$a$} (d);
\path (d) edge node [below] {$b$} (e);
\path (d) edge node [below] {$c$} (f);
\end{tikzpicture}
\vfill
\caption{Before}\label{f9-1}
\end{subfigure}
\begin{subfigure}[b][4cm][s]{.39\textwidth}
\centering
\vfill
\begin{tikzpicture}[state/.style ={circle, draw}]
\node[state] (a) {$p'$};
\node[state] (b) [right=of a] {$p''$};
\node[state] (c) [below=of a] {$q'$};
\node[state] (d) [right=of c] {$q''$};
\path (a) edge node[left] {$b$} (c);
\path (b) edge node[right] {$c$} (d);
\end{tikzpicture}
\vfill
\caption{After}\label{f9-2}
\vspace{\baselineskip}
\end{subfigure}\qquad
\caption{Self connected sum at $p$ and $q$ - When there is one edge between $p$ and $q$, and $p$ and $q$ have no multiple edges.}\label{f9}
\end{figure}
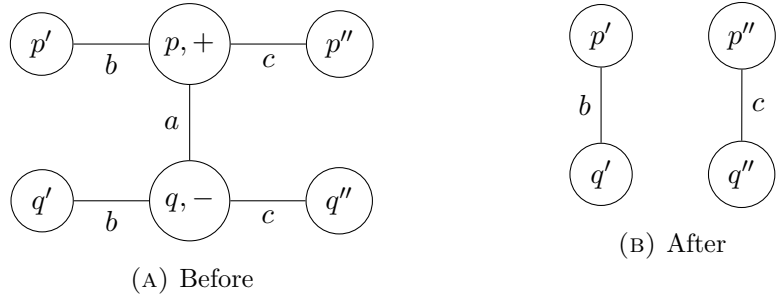

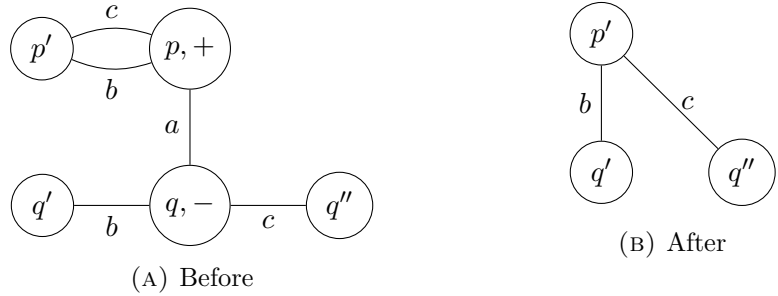
\begin{figure}
\centering
\begin{subfigure}[b][4cm][s]{.6\textwidth}
\centering
\vfill
\begin{tikzpicture}[state/.style ={circle, draw}]
\node[state] (a) {$p,+$};
\node[state] (b) [left=of a] {$p'$};
\node[state] (d) [below=of a]{$q,-$};
\node[state] (e) [left=of d] {$q'$};
\node[state] (f) [right=of d] {$q''$};
\path (a) [bend left =20]edge node[below] {$b$} (b);
\path (a) [bend right =20]edge node[above] {$c$} (b);
\path (a) edge node[left] {$a$} (d);
\path (d) edge node [below] {$b$} (e);
\path (d) edge node [below] {$c$} (f);
\end{tikzpicture}
\vfill
\caption{Before}\label{f10-1}
\end{subfigure}
\begin{subfigure}[b][4cm][s]{.39\textwidth}
\centering
\vfill
\begin{tikzpicture}[state/.style ={circle, draw}]
\node[state] (a) {$p'$};
\node[state] (c) [below=of a] {$q'$};
\node[state] (d) [right=of c] {$q''$};
\path (a) edge node[left] {$b$} (c);
\path (a) edge node[right] {$c$} (d);
\end{tikzpicture}
\vfill
\caption{After}\label{f10-2}
\vspace{\baselineskip}
\end{subfigure}\qquad
\caption{Self connected sum at $p$ and $q$ - When there is one edge between $p$ and $q$, and only $p$ has multiple edges.}\label{f10}
\end{figure}

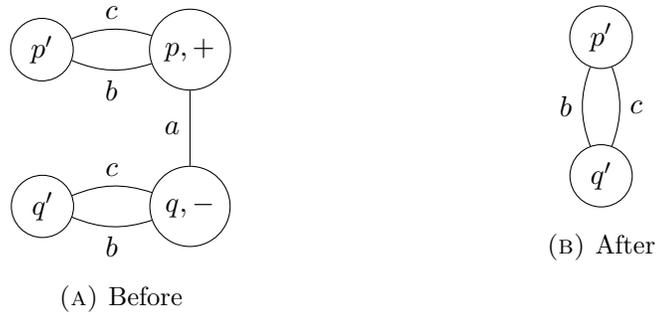
\begin{figure}
\centering
\begin{subfigure}[b][4cm][s]{.6\textwidth}
\centering
\vfill
\begin{tikzpicture}[state/.style ={circle, draw}]
\node[state] (a) {$p,+$};
\node[state] (b) [left=of a] {$p'$};
\node[state] (d) [below=of a]{$q,-$};
\node[state] (e) [left=of d] {$q'$};
\path (a) [bend left =20]edge node[below] {$b$} (b);
\path (a) [bend right =20]edge node[above] {$c$} (b);
\path (a) edge node[left] {$a$} (d);
\path (d) [bend left =20]edge node[below]  {$b$} (e);
\path (d) [bend right =20]edge node[above]  {$c$} (e);
\end{tikzpicture}
\vfill
\caption{Before}\label{f11-1}
\end{subfigure}
\begin{subfigure}[b][4cm][s]{.39\textwidth}
\centering
\vfill
\begin{tikzpicture}[state/.style ={circle, draw}]
\node[state] (a) {$p'$};
\node[state] (c) [below=of a] {$q'$};
\path (c) [bend left =20]edge node[left]  {$b$} (a);
\path (c) [bend right =20]edge node[right]  {$c$} (a);
\end{tikzpicture}
\vfill
\caption{After}\label{f11-2}
\vspace{\baselineskip}
\end{subfigure}\qquad
\caption{Self connected sum at $p$ and $q$ - When there is one edge between $p$ and $q$, and both $p$ and $q$ have multiple edges.}\label{f11}
\end{figure}

\begin{figure}
\centering
\begin{subfigure}[b][4.2cm][s]{.6\textwidth}
\centering
\vfill
\begin{tikzpicture}[state/.style ={circle, draw}]
\node[state] (a) at (0,3) {$p,+$};
\node[state] (b) at (0,0) {$q,-$};
\node[state] (c) at (-1.5,1.5) {$r$};
\node[state] (d) at (1.5,3) {$p'$};
\node[state] (e) at (1.5,0) {$q'$};
\path (a) edge node[left] {$a$} (b);
\path (a) edge node[above] {$b$} (c);
\path (b) edge node[below] {$c$} (c);
\path (a) edge node[below] {$c$} (d);
\path (b) edge node[above] {$b$} (e);
\end{tikzpicture}
\vfill
\caption{Before}\label{f12-1}
\end{subfigure}
\begin{subfigure}[b][4.2cm][s]{.39\textwidth}
\centering
\vfill
\begin{tikzpicture}[state/.style ={circle, draw}]
\node[state] (c) at (-1.5,1.5) {$r$};
\node[state] (d) at (1.5,3) {$p'$};
\node[state] (e) at (1.5,0) {$q'$};
\path (c) edge node[above] {$b$} (d);
\path (c) edge node[below] {$c$} (e);
\end{tikzpicture}
\vfill
\caption{After}\label{f12-2}
\vspace{\baselineskip}
\end{subfigure}\qquad
\caption{Self connected sum at $p$ and $q$ - When there is a vertex $r$ with one edge to $p$ and one edge to $q$}\label{f12}
\end{figure}

\begin{figure}
\centering
\begin{subfigure}[b][4.2cm][s]{.6\textwidth}
\centering
\vfill
\begin{tikzpicture}[state/.style ={circle, draw}]
\node[state] (a) at (0,3) {$p,+$};
\node[state] (b) at (0,0) {$q,-$};
\node[state] (c) at (-1.5,1.5) {$r$};
\node[state] (d) at (1.5,1.5) {$s$};
\path (a) edge node[left] {$a$} (b);
\path (a) edge node[above] {$b$} (c);
\path (b) edge node[below] {$c$} (c);
\path (a) edge node[below] {$c$} (d);
\path (b) edge node[above] {$b$} (d);
\end{tikzpicture}
\vfill
\caption{Before}\label{f13-1}
\end{subfigure}
\begin{subfigure}[b][4.2cm][s]{.39\textwidth}
\centering
\vfill
\begin{tikzpicture}[state/.style ={circle, draw}]
\node[state] (c) at (-1.5,1.5) {$r$};
\node[state] (d) at (1.5,1.5) {$s$};
\path (c) [bend left =20] edge node [above] {$b$} (d);
\path (c) [bend right =20] edge node [below] {$c$} (d);
\end{tikzpicture}
\vfill
\caption{After}\label{f13-2}
\vspace{\baselineskip}
\end{subfigure}\qquad
\caption{Self connected sum at $p$ and $q$ - When there are two vertices $r,s$ with one edge to $p$ and one edge to $q$}\label{f13}
\end{figure}

Figures \ref{f7} through \ref{f13} illustrate self connected sums of 3-valent signed labeled multigraphs in various cases.

\begin{pro}[\textbf{Connected sum of two manifolds}] \label{p47}
For $i=1,2$, let a torus $T^k$ act on a $2n$-dimensional compact connected oriented manifold $M_i$ with isolated fixed points, whose isotropy submanifolds are orientable. Let $\Gamma_i$ be a multigraph describing $M_i$ with Property A. Suppose that $p \in M_1^{T^k}$ and $q \in M_2^{T^k}$ satisfy $\Sigma_{p}=-\Sigma_{q}$, and the labels of edges $e_i$ of $p$ in $\Gamma_1$ and the labels of edges $e_i'$ of $q$ in $\Gamma_2$ agree up to order. Let $M$ ($\Gamma$) be a connected sum of $M_1$ and $M_2$ ($\Gamma_1$ and $\Gamma_2$) at $p$ and $q$. Then $\Gamma$ describes $M$ with Property A. \end{pro}

\begin{proof}
Relabeling the indices if necessary, we may assume that the label of $e_i$ and the label of $e_i'$ are the same for all $i$. Let $a_i$ be the label of $e_i$, $1 \leq i \leq n$. Then $\Sigma_{p}=[\epsilon(p),a_1,\cdots,a_n]$ and $\Sigma_{q}=[-\epsilon(p),a_1,\cdots,a_n]$. 

Let $p_i$ be the other vertex of the edge $e_i$ at $p$ and let $q_i$ be the other vertex of the edge $e_i'$ at $q$. By definition of a multigraph describing $M_1$, the fixed points $p$ and $p_i$ are in the same component $F_i$ of $M_1^{\ker a_i}$. Similarly, $q$ and $q_i$ are in the same component $G_i$ of $M_2^{\ker a_i}$. An equivariant connected sum of $M_1$ and $M_2$ at $p$ and $q$ removes neighborhoods of $p$ and $q$ and glue together, and thus removes a neighborhood of $p$ in $F_i$ and a neighborhood of $q$ in $G_i$ and connects $F_i$ and $G_i$. Hence, after taking the equivariant connected sum, in $M$ the fixed points $p_i$ and $q_i$ are in the same component $H_i$ fixed by the $\ker a_i$-action and both have weight $a_i$. Therefore, if we remove the two vertices (fixed points) $p$ and $q$ and the edges $e_i$ and $e_i'$ of $p$ and $q$ and then draw an edge between $p_i$ and $q_i$ with label $a_i$ for each $i$, the resulting multigraph $\Gamma$ describes the connected sum of $M_1$ and $M_2$ at $p$ and $q$.

Choose all weights in $T_{p'}F_i$ and $T_{q'}G_i$ to be positive multiples of $a_i$ for all $p' \in F_i \cap M_1^{T^k}$ and for all $q' \in G_i \cap M_2^{T^k}$. By the assumption, $F_i$ is orientable. Choose an orientation of $F_i$ so that $\epsilon_{F_i}(p)=+1$. Similarly, choose an orientation of $G_i$ so that $\epsilon_{G_i}(q)=-1$. Since $\Gamma_i$ satisfies Property A and $(p,p_i)$ is $a_i$-edge, we have $\epsilon_{F_i}(p_i)=-\epsilon_{F_i}(p)$. Similarly, $\epsilon_{G_i}(q_i)=-\epsilon_{G_i}(q)$. Because $\epsilon_{F_i}(p)=-\epsilon_{G_i}(q)$, it follows that $\epsilon_{F_i}(p_i)=-\epsilon_{G_i}(q_i)$. Then in $\Gamma$ there is an edge between $p_i$ and $q_i$ with label $a_i$, $\epsilon_{H_i}(p_i)=-\epsilon_{H_i}(q_i)$, and $p_i,q_i$ are in the same component $H_i$ of $M^{\ker a_i}$, where all weights in $H_i$ are positive multiples of $a_i$. Therefore, $\Gamma$ satisfies Property A. \end{proof}

\begin{pro}[\textbf{Self connected sum}] \label{p48}
Let a torus $T^k$ act on a compact connected oriented manifold $M$ with isolated fixed points, whose isotropy submanifolds are orientable. Let $\Gamma$ be a multigraph describing $M$ with Property A. Suppose that two fixed points $p$ and $q$ satisfy $\Sigma_p=-\Sigma_q$, and the labels of edges of $p$ and the labels of edges of $q$ agree up to order.
Let $M'$ ($\Gamma'$) be a self connected sum of $M$ ($\Gamma$) at $p$ and $q$. Then $\Gamma'$ describes $M'$ with Property A. 
\end{pro}

\begin{proof}
The proof is analogous to the proof of Proposition \ref{p47} and we highlight differences.

We use the same notations as in Proposition \ref{p47}. Let $(p,p_i)$ be $a_i$-edge and let $(q,q_i)$ be $a_i$-edge. In this case, the component $F_i$ of $M^{\ker a_i}$ containing $p$ and $p_i$ and the component $G_i$ of $M^{\ker a_i}$ containing $q$ and $q_i$ may be equal. Moreover, we may have $p_i=q$ and thus $q_i=p$, meaning that there is an edge between $p$ and $q$ with label $a_i$. Nevertheless, the proof of Proposition \ref{p47} applies with a little modification.

Since $\Gamma$ satisfies Property A, if $p_i \neq q$, by Lemma \ref{l49} below $p_i$ and $q_i$ are necessarily distinct. Therefore, a self connected sum does not produce any self-loops. The rest of the proof is similar to that of Proposition \ref{p47}. 
\end{proof}

The below lemma shows that our multigraph does not have a vertex with two edges of same label, to vertices of opposite fixed point datum.

\begin{lem} \label{l49}
Let a torus $T^k$ act effectively on a compact oriented manifold $M$ with isolated fixed points, whose isotropy submanifolds are orientable. Let $\Gamma$ be a multigraph describing $M$ with Property A. Suppose that there are three vertices $p_1$, $p_2$, and $p_3$ such that $\Sigma_{p_1}=-\Sigma_{p_2}$, $(p_1,p_3)$ is $a$-edge, and $(p_2,p_3)$ is $b$-edge. Then $a \neq b$. \end{lem}

\begin{proof}
Assume on the contrary that $a=b$. Let $[\epsilon(p_1),a,w_2,\cdots,w_n]$ be the fixed point data of $p_1$. Then $p_2$ has fixed point data $[-\epsilon(p_1),a,w_2,\cdots,w_n]$.

Suppose first that $k=1$ and $a=\pm 1$. Then $M^{\ker a}=M$. Since $\Gamma$ satisfies Property A, vertices of an edge of label $\pm 1$ must have different signs. Because $(p_1,p_3)$ is $a$-edge and $(p_2,p_3)$ is $a$-edge, this means that the signs $\epsilon(p_1)$ and $\epsilon(p_2)$ are the opposite of $\epsilon(p_3)$, but this leads to a contradiction since $\epsilon(p_1)=-\epsilon(p_2)$ by the assumption. Thus, $a \neq \pm 1$ if $k=1$.

From now on we suppose that $a \neq \pm 1$ if $k=1$. Consider a connected component $F$ of $M^{\ker a}$ containing $p$. Then $\dim F<\dim M$. Since the action is effective, $p_3$ has two weights $\{a,a\}$ divisible by $a$, and $a \neq \pm 1$ if $k=1$, it follows that $\dim F:=2m \geq 4$; $p_3$ may have more than two weights that are integer multiples of $a$. Choose an orientation of $F$ as $F$ is orientable by the assumption, and choose all weights in $F$ to be positive multiples of $a$. Since $\Gamma$ satisfies Property A, the signs $\epsilon_F(p_1)$ and $\epsilon_F(p_2)$ for the $T^k$-action on $F$ are the opposite of $\epsilon_F(p_3)$ because $(p_1,p_3)$ is $a$-edge and $(p_2,p_3)$ is $a$-edge. Therefore, $\epsilon_F(p_1)=\epsilon_F(p_2)$.

We have $\Sigma_{p_1}=[\epsilon(p_1),a,w_2,\cdots,w_n]=-\Sigma_{p_2}=[-\epsilon(p_1),a,w_2,\cdots,w_n]$. Thus, we can take a self equivariant connected sum of $M$ at $p$ and $q$ to construct another compact oriented $T^k$-manifold $M'$. The self connected sum takes a self equivariant connected sum of the induced $T^k$-action on $F$ at $p_1$ and $p_2$, to construct another $2m$-dimensional compact oriented manifold $F'$, which is a connected submanifold of $M'$. However, because $\epsilon_F(p_1)=\epsilon_F(p_2)$, the connected sum of $M$ yields $F'$ non-orientable, which is a contradiction. Therefore, $a \neq b$ and the lemma follows.

We can also see that $F'$ is not orientable as follows. Assume on the contrary that $F'$ is orientable; choose an orientation of $F'$ so that if $p$ is in both $F^{T^k}$ and ${F'}^{T^k}$, then $\epsilon_F(p)=\epsilon_{F'}(p)$. For instance $p_3$ is in both of them.

Applying Theorem \ref{t24} to the induced $T^k$-action on $F$ with taking $\alpha=1$,
\begin{center}
$\displaystyle \int_F 1=\sum_{p \in F^{T^k}} \int_p \frac{1}{e_{T^k}(T_pM)}=\sum_{p \in F^{T^k}} \epsilon_F(p) \cdot \frac{1}{w_{p,1}\cdots w_{p,m}}$,
\end{center}
which is zero by a dimensional reason $\int_F:H_{T^k}^i(F;\mathbb{Z}) \to H^{i- \dim F}(\mathbb{CP}^\infty;\mathbb{Z})$. Here $w_{p,1},\cdots,w_{p,m}$ are weights in $N_pF$.

On the other hand, because $p_1$ and $p_2$ have the same multiset of weights for the $T^k$-action on $M$ and $\epsilon_F(p_1)=\epsilon_F(p_2)$, the fixed point data of $p_1$ and $p_2$ for the induced $T^k$-action on $F$ are equal; it is $[\epsilon_F(p_1),a,x_2,\cdots,x_m]$ for some $x_2,\cdots,x_m$.

Because $F'$ has two less fixed points than $F$, the fixed point data of $F'$ is
\begin{center}
$\displaystyle \Sigma_{F'}=\Sigma_F \setminus \{[\epsilon_F(p_1),a,x_2,\cdots,x_m],[\epsilon_F(p_2),a,x_2,\cdots,x_m]\}$.
\end{center}
Applying Theorem \ref{t24} to the induced $T^k$-action on $F'$ with taking $\alpha=1$,
\begin{center}
$\displaystyle 0=\int_{F'} 1=\sum_{p \in {F'}^{T^k}} \int_p \frac{1}{e_{T^k}(T_pM)}=\sum_{p \in {F}^{T^k}} \epsilon_F(p) \cdot \frac{1}{w_{p,1}\cdots w_{p,m}}-\epsilon_F(p_1) \cdot \frac{1}{w_{p_1,1}\cdots w_{p_1,m}}-\epsilon_F(p_2) \cdot \frac{1}{w_{p_2,1}\cdots w_{p_2,m}}=0-2 \epsilon_F(p_1) \cdot \frac{1}{ax_2\cdots x_m}$,
\end{center}
which is a contradiction. Thus $F'$ is not orientable. \end{proof}

\begin{figure}
\centering
\begin{tikzpicture}[state/.style ={circle, draw}]
\node[state] (a) at (0,3) {$p_1,+$};
\node[state] (b) at (0,0) {$p_2,-$};
\node[state] (c) at (-2,1.5) {$p_3$};
\path (a) edge node[above] {$a$} (c);
\path (b) edge node[below] {$a$} (c);
\end{tikzpicture}
\caption{Any multigraph describing $M$ with Property A does not contain a sub-multigraph of this form with $\Sigma_{p_1}=-\Sigma_{p_2}$.}\label{f14}
\end{figure}
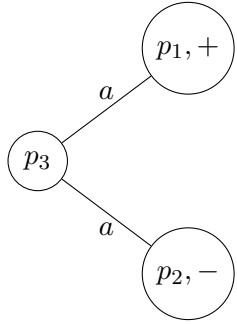

We discuss equivariant blow up at a fixed point, of a torus action on an oriented manifold with isolated fixed points and a corresponding operation on a multigraph describing it. For a complex manifold, blow up is an operation that replaces a point by all complex straight lines through it. Blowing up a point is equivalent to taking a connected sum with $\overline{\mathbb{CP}^n}$. Thus, we shall discuss a multigraph describing a linear $T^k$-action on $\mathbb{CP}^n$. For a torus action on an almost complex manifold with isolated fixed points, the sign of every weight at a fixed point is well-defined and there is a better notion of a multigraph which encodes the information on weights at fixed points \cite{J6}; also see \cite{GS} and \cite{JT} on multigraphs for circle actions. Hence, we also discuss multigraphs for (almost) complex manifolds and illustrate multigraphs of standard examples as almost complex manifolds, and then convert such a multigraph for a complex manifold into a multigraph for an oriented manifold.

\begin{Definition} \label{d410} \cite{J6}
A \textbf{directed labeled $k$-multigraph} $\Gamma$ is a set $V$ of vertices, a set $E$ of
edges, maps $i : E \to V$ and $t : E \to V$ giving the initial and terminal vertices of
each edge, and a map $w:E \to \mathbb{Z}^k$ giving the label of each edge.
\end{Definition}

\begin{Definition} \label{d411} \cite{J6}
Let a torus $T^k$ act on a compact almost complex manifold $M$ with isolated fixed points. We say that a (directed labeled $k$-)multigraph $\Gamma=(V,E)$ \textbf{describes} $M$ if the following hold:
\begin{enumerate}[(i)]
\item The vertex set $V$ is equal to the fixed point set $M^{T^k}$.
\item The multiset of the weights at $p$ is $\{w(e) \, | \, i(e)=p\} \cup \{-w(e) \, | \, t(e)=p\}$ for all $p \in M^{T^k}$.
\item For each edge $e$, the two endpoints $i(e)$ and $t(e)$ are in the same component of the isotropy submanifold $M^{\ker w(e)}$.
\end{enumerate}
\end{Definition}

\begin{exa}[The complex projective space $\mathbb{CP}^n$] \label{e412}
Fix a positive integer $k$. Let $a_1,\cdots,a_n$ be non-zero elements of $\mathbb{Z}^k$ such that $a_i \neq a_j$ for $i \neq j$. Let a torus $T^k$ act on the complex projective space $\mathbb{CP}^n$ by
\begin{center}
$g \cdot [z_0:z_1:\cdots:z_n]=[z_0:g^{a_1}z_1:\cdots:g^{a_n}z_n]$
\end{center}
for all $g \in T^k \subset \mathbb{C}^n$ and for all $[z_0:\cdots:z_n] \in \mathbb{CP}^n$. The action has $n+1$ fixed points $q_i=[0:\cdots:0:1:0:\cdots:0]$ where the $i$-th entry is 1, $0 \leq i \leq n$. The complex weights at $q_i$ are $\{a_j-a_i\}_{j \neq i}$ if we let $a_0=(0,\cdots,0) \in \mathbb{Z}^k$.

For $i<j$, the fixed points $q_i$ and $q_j$ are in the connected component $F_{ij}:=[0:\cdots:0:z_i:0:\cdots:0:z_j:0:\cdots:0]$ of $M^{\ker (a_j-a_i)}$, which is the 2-sphere, on which the $T^k$-action on $\mathbb{CP}^n$ restricts to act by
\begin{center}
$g \cdot [0:\cdots 0:z_i:0:\cdots:0:z_j:0:\cdots:0]$

$=[0:\cdots:0:g^{a_i} z_i:0:\cdots:0:g^{a_j}z_j:0:\cdots:0]$,
\end{center}
giving $q_i$ and $q_j$ weight $a_j-a_i$ and $a_i-a_j$ for this action, respectively. Thus, with the orientation on $F$ induced from the orientation on  $\mathbb{CP}^n$, the fixed point data for the $T^k$-action on $F_{ij}$ at $q_i$ and $q_j$ are $[+,a_j-a_i]$ and $[-,a_j-a_i]$. In particular, $\epsilon_{F_{ij}}(q_i)=+1$ and $\epsilon_{F_{ij}}(q_j)=-1$. To each fixed point $q_i$, we assign a vertex, also denoted $q_i$.

\textbf{Complex multigraph: directed labeled multigraph}

For $i<j$ we draw a directed edge from $q_i$ to $q_j$ and give the edge label $a_j-a_i$. Let $\Gamma_{c}$ be a directed labeled multigraph obtained. Then $\Gamma_{c}$ describes the linear $T^k$-action on $\mathbb{CP}^n$ as a directed labeled multigraph describing a complex manifold. Figure \ref{f15-1} is such an example of $\Gamma_c$ for $n=3$.

\textbf{Oriented multigraph: signed labeled multigraph}

For $i<j$ we draw a non-directed edge between $q_i$ and $q_j$ and give the edge label $a_j-a_i$. With this choice of weights, the sign of $q_i$ is $(-1)^i$, that is, the fixed point data of $q_i$ as an oriented manifold is 
\begin{center}
$\Sigma_{q_i}=[(-1)^i,a_i-a_0,\cdots,a_i-a_{i-1},a_{i+1} - a_i,\cdots,a_n-a_i]$
\end{center}
if $i \neq 0$ and
\begin{center}
$\Sigma_{q_0}=[+,a_1,\cdots,a_n]$.
\end{center}
Thus, to each vertex $q_i$ we assign the sign $(-1)^i$. Let $\Gamma_{\mathbb{CP}^n}$ be a signed labeled multigraph obtained. Because $\epsilon_{F_{ij}}(q_i)=+1$ and $\epsilon_{F_{ij}}(q_j)=-1$ for $i<j$, $\Gamma_{\mathbb{CP}^n}$ describes $\mathbb{CP}^n$ with Property A. Figure \ref{f15-2} is such an example of $\Gamma$ for $n=3$.

Let $\overline{\Gamma}_{\mathbb{CP}^n}$ be another signed labeled multigraph obtained from $\Gamma_{\mathbb{CP}^n}$ by replacing the sign $\epsilon(v)$ of every vertex $v$ with its negative $-\epsilon(v)$. Then $\overline{\Gamma}_{\mathbb{CP}^n}$ describes $\overline{\mathbb{CP}^n}$ with Property A, where the $T^k$-action on $\overline{\mathbb{CP}^n}$ is the same as the action on $\mathbb{CP}^n$ in this example.
\end{exa}

\begin{figure} 
\centering
\begin{subfigure}[b][8.1cm][s]{.4\textwidth}
\centering
\vfill
\begin{tikzpicture}[state/.style ={circle, draw}]
\node[state] (a) {$q_0$};
\node[state] (b) [above right=of a] {$q_1$};
\node[state] (c) [above=of b] {$q_2$};
\node[state] (d) [above left=of c] {$q_3$};
\path (a) [->] edge node[right] {$a_1$} (b);
\path (a) [->] edge node[left] {$a_2$} (c);
\path (a) [->] edge node[left] {$a_3$} (d);
\path (b) [->] edge node [right] {$a_2-a_1$} (c);
\path (b) [->] edge node [left] {$a_3-a_1$} (d);
\path (c) [->] edge node [right] {$a_3-a_2$} (d);
\end{tikzpicture}
\caption{Directed labeled multigraph describing the complex manifold $\mathbb{CP}^3$} \label{f15-1}
\vfill
\end{subfigure}
\begin{subfigure}[b][8.1cm][s]{.4\textwidth}
\centering
\vfill
\begin{tikzpicture}[state/.style ={circle, draw}]
\node[state] (a) {$q_0,+$};
\node[state] (b) [above right=of a] {$q_1,-$};
\node[state] (c) [above=of b] {$q_2,+$};
\node[state] (d) [above left=of c] {$q_3,-$};
\path (a) edge node[right] {$a_1$} (b);
\path (a) edge node[left] {$a_2$} (c);
\path (a) edge node[left] {$a_3$} (d);
\path (b) edge node [right] {$a_2-a_1$} (c);
\path (b) edge node [left] {$a_3-a_1$} (d);
\path (c) edge node [right] {$a_3-a_2$} (d);
\end{tikzpicture}
\caption{Signed labeled multigraph describing the oriented manifold $\mathbb{CP}^3$} \label{f15-2}
\vfill
\end{subfigure}
\caption{Multigraphs describing a linear $T^k$-action on $\mathbb{CP}^3$}\label{f15}
\end{figure}
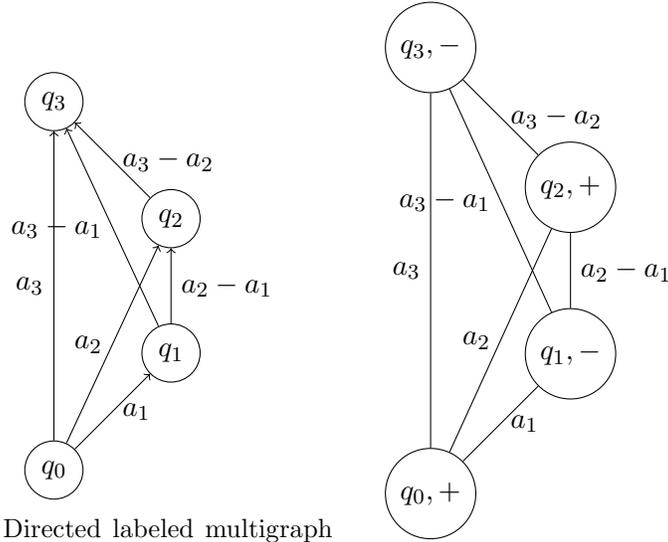

We define blow up of a torus action on an oriented manifold at an isolated fixed point.

\begin{Definition}[\textbf{Blow up of a manifold}] \label{d413}
Let a torus $T^k$ act on a $2n$-dimensional oriented manifold $M$ with isolated fixed points. Suppose that a fixed point $p$ has fixed point data $[\epsilon(p),a_1,\cdots,a_n]$, where $a_i \neq a_j$ for $i \neq j$. By \textbf{blow up of $p$} we mean an equivariant connected sum of $M$ and $\overline{\mathbb{CP}^n}$ ($\mathbb{CP}^n$) at $p$ and $q_0$ if $\epsilon(p)=+1$ (if $\epsilon(p)=-1$), where the torus action on $\overline{\mathbb{CP}^n}$ ($\mathbb{CP}^n$) is as in Example \ref{e412}.
\end{Definition}

The blow up of Definition \ref{d413} depends on the choice of the weights $a_i$ (the orientations of $L_{p,i}$). For instance, if we use weights $\{-a_1,-a_2,a_3,\cdots,a_n\}$ where $-a_1,-a_2,a_3,\cdots,a_n$ are mutually distinct, this multiset of weights still gives $p$ the same fixed point data $[\epsilon(p),a_1,\cdots,a_n]$, but we get a different manifold when we blow up $p$ with these weights $-a_1,-a_2,a_3,\cdots,a_n$.

Accordingly, we define blow up of a multigraph.

\begin{Definition}[\textbf{Blow up of a multigraph}] \label{d414}
Let $\Gamma$ be an $n$-regular signed labeled $k$-multigraph without any self-loops. Suppose a vertex $p$ has edges with labels $a_1$, $\cdots$, $a_n$. By \textbf{blow up of $p$} we mean a connected sum at $p$ and $q_0$ of $\Gamma$ and $\overline{\Gamma}_{\mathbb{CP}^n}$ if $\epsilon(p)=+1$ ($\Gamma_{\mathbb{CP}^n}$ if $\epsilon(p)=-1$), where $\Gamma_{\mathbb{CP}^n}$ and $\overline{\Gamma}_{\mathbb{CP}^n}$ are as in Examle \ref{e412}.
\end{Definition}

Blow up of a multigraph describes blow up of a manifold.

\begin{pro}
Let a torus $T^k$ act on a compact oriented manifold $M$ with isolated fixed points. Let $\Gamma$ be a multigraph describing $M$ with Property A. Suppose that a fixed point (vertex) $p$ has edges with labels $a_1,\cdots,a_n$ and we choose $a_1,\cdots,a_n$ to be the weights at $p$. Let $M'$ be blow up of $M$ at $p$ and let $\Gamma'$ be blow up of $\Gamma$ at $p$. Then $\Gamma'$ describes $M'$ with Property A.
\end{pro}

\begin{proof}
By definition, blow up of $M$ at $p$ is an equivariant connected sum at $p$ and $q_0$ of $M$ and $\overline{\mathbb{CP}^n}$ ($\mathbb{CP}^n$ if $\epsilon(p)=+1$), where the $T^k$-action on $\overline{\mathbb{CP}^n}$ ($\mathbb{CP}^n$) is as in Example \ref{e412}. As in Example \ref{e412}, $\Gamma_{\mathbb{CP}^n}$ ($\overline{\Gamma}_{\mathbb{CP}^n}$) describes the action on $\mathbb{CP}^n$ ($\overline{\mathbb{CP}^n}$) with Property A. 
By definition, blow up of $\Gamma$ at $p$ is a connected sum at $p$ and $q_0$ of $\Gamma$ and $\overline{\Gamma}_{\mathbb{CP}^n}$ ($\Gamma_{\mathbb{CP}^n}$ if $\epsilon(p)=+1$). By Proposition \ref{p47}, the blown up multigraph $\Gamma'$ describes $M'$ with Property A.
\end{proof}

\begin{figure} 
\centering
\begin{subfigure}[b][5.3cm][s]{.49\textwidth}
\centering
\vfill
\begin{tikzpicture}[state/.style ={circle, draw}]
\node[state] (e) at (-1.3, 0) {$p,\pm$};
\node[state] (f) at (-2.6, -1.7) {};
\node[state] (g) at (-2.6, 0) {};
\node[state] (h) at (-2.6, 1.7) {};
\path (f) edge node [right] {$a_1$} (e);
\path (g) edge node [below] {$a_2$} (e);
\path (h) edge node [right] {$a_3$} (e);
\end{tikzpicture}
\caption{$\Gamma$} \label{f16-1}
\vfill
\end{subfigure}
\begin{subfigure}[b][5.3cm][s]{.49\textwidth}
\centering
\vfill
\begin{tikzpicture}[state/.style ={circle, draw}]
\node[state] (a) at (0,0) {$q_0,\mp$};
\node[state] (b) at (1.3, -1.7) {$q_1,\pm$};
\node[state] (c) at (2.6, 0) {$q_2,\mp$};
\node[state] (d) at (1.3, 1.7) {$q_3,\pm$};
\node[state] (e) at (-1.3, 0) {$p,\pm$};
\node[state] (f) at (-2.6, -1.7) {};
\node[state] (g) at (-2.6, 0) {};
\node[state] (h) at (-2.6, 1.7) {};
\path (a) edge node[left] {$a_1$} (b);
\path (a) edge node[pos=.2, below] {$a_2$} (c);
\path (a) edge node[left] {$a_3$} (d);
\path (b) edge node [right] {$a_2-a_1$} (c);
\path (b) edge node [above] {$a_3-a_1$} (d);
\path (c) edge node [right] {$a_3-a_2$} (d);
\path (f) edge node [right] {$a_1$} (e);
\path (g) edge node [below] {$a_2$} (e);
\path (h) edge node [right] {$a_3$} (e);
\end{tikzpicture}
\caption{$\Gamma \sqcup \overline{\Gamma}_{\mathbb{CP}^n}$ (or $\Gamma \sqcup \Gamma_{\mathbb{CP}^n}$)} \label{f16-2}
\vfill
\end{subfigure}
\begin{subfigure}[b][5cm][s]{.49\textwidth}
\centering
\vfill
\begin{tikzpicture}[state/.style ={circle, draw}]
\node[state] (b) at (1, -1.7) {$q_1,\pm$};
\node[state] (c) at (2.6, 0) {$q_2,\mp$};
\node[state] (d) at (1, 1.7) {$q_3,\pm$};
\node[state] (f) at (-1.6, -1.7) {};
\node[state] (g) at (-1.6, 0) {};
\node[state] (h) at (-1.6, 1.7) {};
\path (b) edge node [right] {$a_2-a_1$} (c);
\path (b) edge node [pos=.8, left] {$a_3-a_1$} (d);
\path (c) edge node [right] {$a_3-a_2$} (d);
\path (f) edge node [below] {$a_1$} (b);
\path (g) edge node [below] {$a_2$} (c);
\path (h) edge node [below] {$a_3$} (d);
\end{tikzpicture}
\caption{Blow up of $\Gamma$ at $p$} \label{f16-3}
\vfill
\end{subfigure}
\caption{Blow up of a multigraph}\label{f16}
\end{figure}
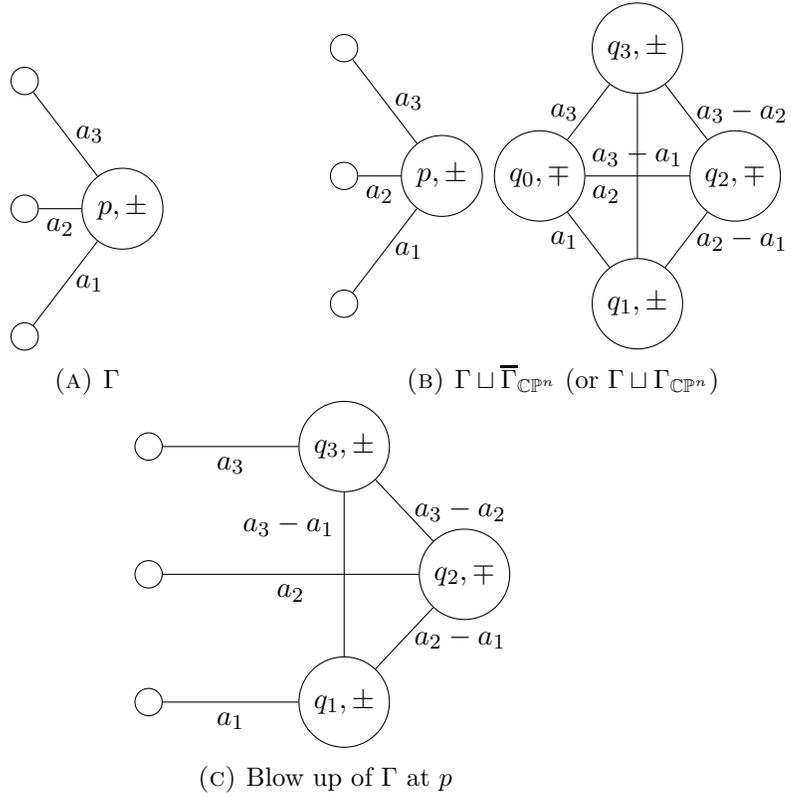

\section{Dimension 4: classification of multigraph and fixed point data}\label{s5}

In this section, we classify multigraphs describing torus actions on 4-dimensional compact oriented manifolds with isolated fixed points, and discuss its applications.

\begin{figure}
\centering
\begin{tikzpicture}[state/.style ={circle, draw}]
\node[state] (a) at (0,0) {$q_1,+$};
\node[state] (b) at (0,2) {$q_2,-$};
\path (a) [bend left =20]edge node[left] {$A$} (b);
\path (b) [bend left =20]edge node[right] {$B$} (a);
\end{tikzpicture}
\caption{Minimal model in dimension 4, Operation (1)} \label{f17}
\end{figure}
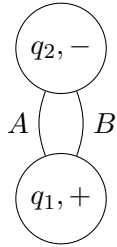

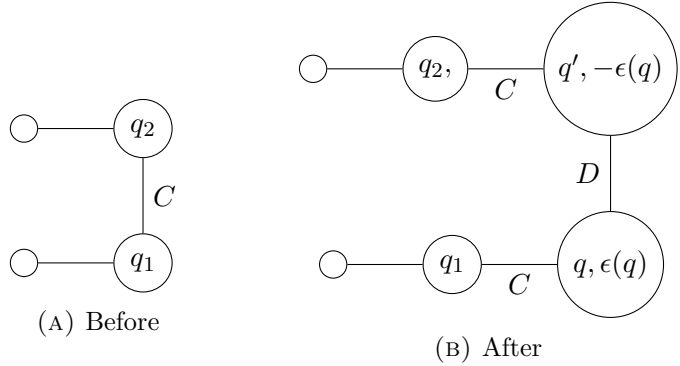
\begin{figure} 
\centering
\begin{subfigure}[b][4.5cm][s]{.4\textwidth}
\centering
\vfill
\begin{tikzpicture}[state/.style ={circle, draw}]
\node[state] (a) {$q_1$};
\node[state] (b) [above=of a] {$q_2$};
\node[state] (c) [left=of a] {};
\node[state] (d) [left=of b] {};
\path (a) edge node[right] {$C$} (b);
\path (a) edge node[below] {} (c);
\path (b) edge node[below] {} (d);
\end{tikzpicture}
\caption{Before} \label{f18-1}
\vfill
\end{subfigure}
\begin{subfigure}[b][4.5cm][s]{.4\textwidth}
\centering
\vfill
\begin{tikzpicture}[state/.style ={circle, draw}]
\node[state] (a) {$q_1$};
\node[state] (b) [right=of a] {$q,\epsilon(q)$};
\node[state] (c) [above=of b] {$q',-\epsilon(q)$};
\node[state] (d) [left=of c] {$q_2,$};
\node[state] (e) [left=of a] {};
\node[state] (f) [left=of d] {};
\path (a) edge node[below] {$C$} (b);
\path (b) edge node[left] {$D$} (c);
\path (c) edge node[below] {$C$} (d);
\path (a) edge node[below] {} (e);
\path (d) edge node[below] {} (f);
\end{tikzpicture}
\caption{After} \label{f18-2}
\vfill
\end{subfigure}
\caption{Operation (2)}\label{f18}
\end{figure}

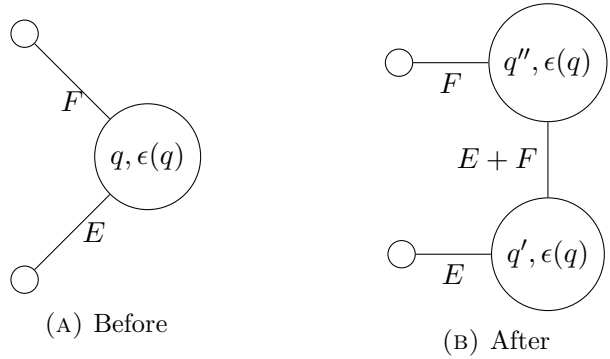
\begin{figure} 
\centering
\begin{subfigure}[b][4.5cm][s]{.4\textwidth}
\centering
\vfill
\begin{tikzpicture}[state/.style ={circle, draw}]
\node[state] (a) {};
\node[state] (b) [above right=of a] {$q,\epsilon(q)$};
\node[state] (c) [above left=of b] {};
\path (a) edge node[right] {$E$} (b);
\path (b) edge node[below] {$F$} (c);
\end{tikzpicture}
\caption{Before} \label{f19-1}
\vfill
\end{subfigure}
\begin{subfigure}[b][4.5cm][s]{.4\textwidth}
\centering
\vfill
\begin{tikzpicture}[state/.style ={circle, draw}]
\node[state] (a) {};
\node[state] (b) [right=of a] {$q',\epsilon(q)$};
\node[state] (c) [above=of b] {$q'',\epsilon(q)$};
\node[state] (d) [left=of c] {};
\path (a) edge node[below] {$E$} (b);
\path (b) edge node[left] {$E+F$} (c);
\path (c) edge node[below] {$F$} (d);
\end{tikzpicture}
\caption{After} \label{f19-2}
\vfill
\end{subfigure}
\caption{Operation (3)}\label{f19}
\end{figure}

\begin{theo} \label{t51} Let a torus $T^k$ act effectively on a 4-dimensional compact oriented manifold $M$ with isolated fixed points, whose isotropy submanifolds are orientable. Then a multigraph describing $M$ is obtained by applying a combination of the following operations, beginning with the empty graph.
\begin{enumerate}[(1)]
\item Add Figure \ref{f17} for some non-zero $A,B \in \mathbb{Z}^k$ that span $\mathbb{Z}^k$.
\item If $(q_1,q_2)$ is $C$-edge, replace the edge $(q_1,q_2)$ with three edges $(q_1,q)$ of label $C$, $(q,q')$ of label $D$, $(q',q_2)$ of label $C$ with $\epsilon(q)=-\epsilon(q')$, where $C,D$ span $\mathbb{Z}^k$ (Figure \ref{f18}).
\item If a vertex $q$ has edges with labels $E$ and $F$ where $E,F$ span $\mathbb{Z}^k$, replace the vertex $q$ with $(E+F)$-edge $(q',q'')$ with $\epsilon(q)=\epsilon(q')=\epsilon(q'')$ (Figure \ref{f19}).
\item Reverse an edge.
\end{enumerate}
Conversely, there exists a multigraph $\Gamma$ describing $M$ with Property A, and we can convert $\Gamma$ to the empty graph by self connected sums, blow ups, and reversing edges. In other words, we can convert $\Gamma$ to the empty multigraph by performing a combination of the reverses of Operations (1-4) above.
\end{theo}

Along the way of proving Theorem \ref{t51}, we also prove the following.

\begin{theo} \label{t52}
Let a torus $T^k$ act effectively on a 4-dimensional compact oriented manifold $M$ with isolated fixed points, whose isotropy submanifolds are orientable. Then we can successively blow up fixed points and take self connected sums to construct another 4-dimensional compact oriented manifold, which is equipped with a fixed point free $T^k$-action.
\end{theo}

\begin{proof}[\textbf{Proof of Theorem \ref{t51} and Theorem \ref{t52}}]

\textbf{$\Gamma_0$ is a multigraph describing $M$, and $l$ is a weight of the biggest magnitude}

By Proposition \ref{p34}, there is a multigraph $\Gamma_0$ describing $M$ that satisfies Property A. Let $l$ be a weight of the biggest magnitude; $|l|=\max\{|w_{p,i}|:p \in M^{T^k}, 1 \leq i \leq 2\}$.

\textbf{We may reverse edges to get another multigraph $\Gamma$ describing $M$, Operation (4)}

Let $e$ be an edge of $\Gamma_0$. If its label $w(e)$ satisfies $\langle w(e), l \rangle < 0$ ($w(e)<0$ if $k=1$), then we reverse the edge $e$ (Definition \ref{d35}), so that its new label, still denoted $w(e)$, satisfies $\langle w(e), l \rangle \geq 0$ ($w(e)>0$ if $k=1$). Accordingly, if $e$ has vertices $p_1$ and $p_2$, we replace $w_{p,i}=w(e)$ by its negative $-w_{p,i}$ (reverse the orientation of $L_{p,i}$) so that the weights $w_{p,1},w_{p,2}$ at $p$ satisfy $\langle w_{p,i}, l \rangle \geq 0$ ($w_{p,i}>0$ if $k=1$). By Lemma \ref{l36}, the resulting multigraph $\Gamma$ also describes $M$ with Property A.

\textbf{Case (0): $k=1$ and $l=1$ - We use self connected sum, whose reverse is Operation (1) or (2)}

Suppose that $k=1$ and $l=1$. By Lemma \ref{l28}, the number of fixed points with fixed point data $[+,1,1]$ is equal to the number of fixed points with fixed point data $[-,1,1]$. Let $e$ be an edge of $\Gamma$ and let $p_1$ and $p_2$ the vertices of $e$. Because $\Gamma$ satisfies Property A, $p_1$ and $p_2$ satisfy $\epsilon_F(p_1)=-\epsilon_F(p_2)$, where $F$ is a component of $M^{\ker 1}$. Since $M^{\ker 1}=M$, this means $\epsilon_M(p_1)=\epsilon_F(p_1)$ and similarly $\epsilon_M(p_2)=\epsilon_F(p_2)$. Therefore, $\epsilon(p_1)=-\epsilon(p_2)$ and hence $\Sigma_{p_1}=[+,1,1]=-[-,1,1]=-\Sigma_{p_2}$.

Assume that there are two edges between $p_1$ and $p_2$; Figure \ref{f20-1} with $l=x=1$ describes $M$. Let $M'$ ($\Gamma'$) be a self connected sum of $M$ ($\Gamma$) at $p_1$ and $p_2$. By Proposition \ref{p48}, $\Gamma'$, which is Figure \ref{f20-2}, describes $M'$ with Property A. The reverse of this self connected sum is Operation (1).

Next, assume that $e$ is the only edge between $p_1$ and $p_2$. Let $e_1$ ($e_2$) be the other edge of $p_1$ ($p_2$) and let $p_1'$ ($p_2'$) be the other vertex of $e_1$ ($e_2$). Because $\epsilon(p_1)=-\epsilon(p_2)$ and the labels of $e_1$ and $e_2$ are both 1, by Lemma \ref{l49} $p_1'$ and $p_2'$ are distinct. Figure \ref{f21-1} describes $M$ with Property A. Let $M'$ ($\Gamma'$) be a self connected sum of $M$ ($\Gamma$) at $p_1$ and $p_2$. By Proposition \ref{p48}, $\Gamma'$, which is Figure \ref{f21-2}, describes $M'$ with Property A. The reverse of this self connected sum is Operation (2).

\textbf{From now on assume $l>1$ if $k=1$ - there are two cases}

From now on, we assume that $l>1$ if $k=1$. Let $e$ be an edge with label $l$. Let $p_1$ and $p_2$ be the vertices of $e$ and let $[\epsilon(p_1),l,x]$ be the fixed point data of $p$. Because $\Gamma$ describes $M$, by Lemma \ref{l53} below, the fixed points $p_1$ and $p_2$ are in the same 2-sphere of $M^{\ker l}$, and one of the following holds for the fixed point data of $p_2$:
\begin{enumerate}[(a)]
\item $\Sigma_{p_2}=[-\epsilon(p_1),l,x]$.
\item $\Sigma_{p_2}=[\epsilon(p_1),l,l-x]$.
\end{enumerate}

\textbf{Case (a) - We use self connected sum, whose reverse is Operation (1) or (2)}

Suppose that Case (a) holds that $\Sigma_{p_2}=[-\epsilon(p_1),l,x]$.

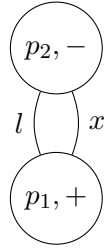
\begin{figure} 
\centering
\begin{subfigure}[b][3.4cm][s]{.49\textwidth}
\centering
\vfill
\begin{tikzpicture}[state/.style ={circle, draw}]
\node[state] (a) at (0,0) {$p_1,+$};
\node[state] (b) at (0,2) {$p_2,-$};
\path (a) [bend left =20]edge node[left] {$l$} (b);
\path (b) [bend left =20]edge node[right] {$x$} (a);
\end{tikzpicture}
\caption{Before} \label{f20-1}
\vfill
\end{subfigure}
\begin{subfigure}[b][3.4cm][s]{.49\textwidth}
\centering
\vfill
\begin{tikzpicture}[state/.style ={circle, draw}]
\end{tikzpicture}
\caption{After} \label{f20-2}
\vfill
\end{subfigure}
\caption{Case (a) - there are two edges between $p_1$ and $p_2$}\label{f20}
\end{figure}

Assume that there are two edges between $p_1$ and $p_2$. Figure \ref{f20-1} describes $M$. If we let $M'$ ($\Gamma'$, Figure \ref{f20-2}) be a self connected sum of $M$ ($\Gamma$) at $p_1$ and $p_2$, by Proposition \ref{p48} $\Gamma'$ describes $M'$ with Property A. The reverse of this self connected sum is Operation (1).

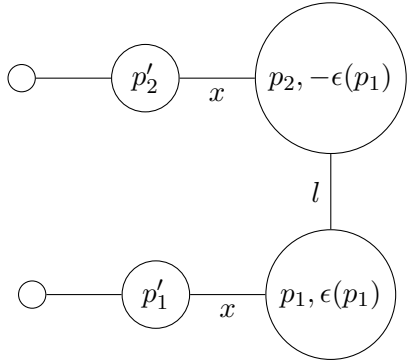
\begin{figure} 
\centering
\begin{subfigure}[b][4.9cm][s]{.49\textwidth}
\centering
\vfill
\begin{tikzpicture}[state/.style ={circle, draw}]
\node[state] (a) {$p_1'$};
\node[state] (b) [right=of a] {$p_1,\epsilon(p_1)$};
\node[state] (c) [above=of b] {$p_2,-\epsilon(p_1)$};
\node[state] (d) [left=of c] {$p_2'$};
\node[state] (e) [left=of a] {};
\node[state] (f) [left=of d] {};
\path (a) edge node[below] {$x$} (b);
\path (b) edge node[left] {$l$} (c);
\path (c) edge node[below] {$x$} (d);
\path (a) edge node[below] {} (e);
\path (d) edge node[below] {} (f);
\end{tikzpicture}
\caption{Before} \label{f21-1}
\vfill
\end{subfigure}
\begin{subfigure}[b][4.9cm][s]{.49\textwidth}
\centering
\vfill
\begin{tikzpicture}[state/.style ={circle, draw}]
\node[state] (a) {$p_1'$};
\node[state] (b) [above=of a] {$p_2'$};
\node[state] (c) [left=of a] {};
\node[state] (d) [left=of b] {};
\path (a) edge node[right] {$x$} (b);
\path (a) edge node[below] {} (c);
\path (b) edge node[below] {} (d);
\end{tikzpicture}
\caption{After} \label{f21-2}
\vfill
\end{subfigure}
\caption{Case (a) - there is only one edge between $p_1$ and $p_2$}\label{f21}
\end{figure}

Assume that $e$ is the only edge between $p_1$ and $p_2$. Let $e_1$ ($e_2$) be the other edge of $p_1$ ($p_2$) and let $p_1'$ ($p_2'$) be the other vertex of $e_1$ ($e_2$). Since $\epsilon(p_1)=-\epsilon(p_2)$ and the labels of $e'$ and $e''$ are both $x$, by Lemma \ref{l49} $p_1'$ and $p_2'$ are distinct. Figure \ref{f21-1} describes $M$ with Property A. Let $M'$ ($\Gamma'$) be a self connected sum of $M$ ($\Gamma$) at $p_1$ and $p_2$. By Proposition \ref{p48}, $\Gamma'$, which is Figure \ref{f21-2}, describes $M'$ with Property A. The reverse of this self connected sum is Operation (2).

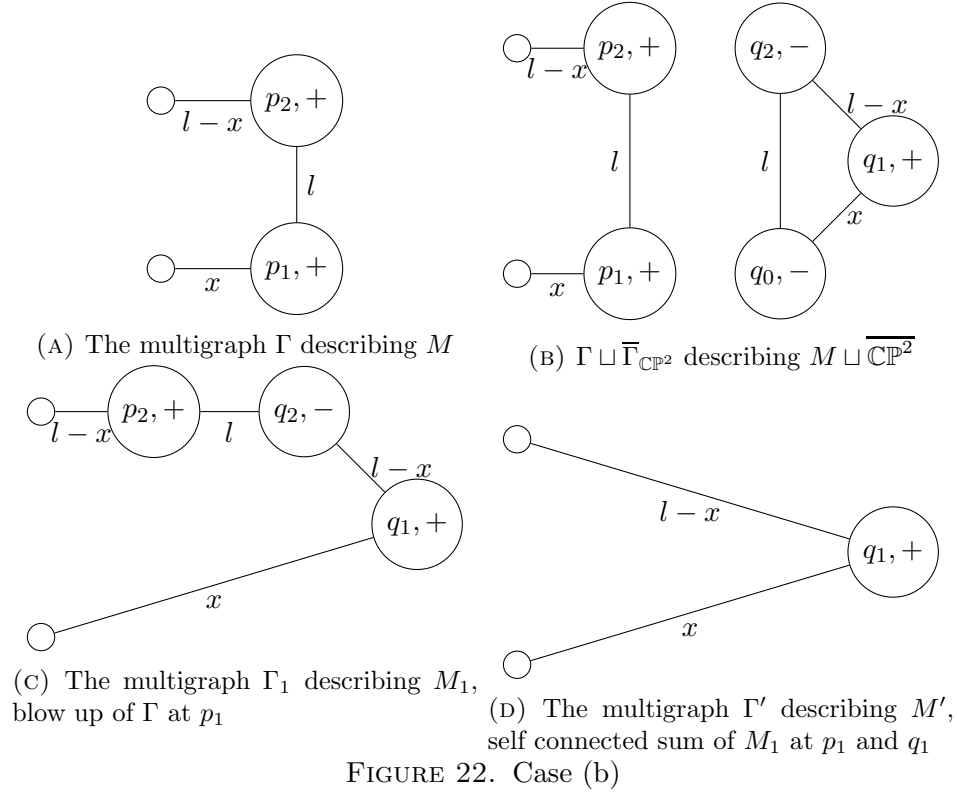
\textbf{Case (b) - We use blow up and self connected sum, whose reverse is Operation (3)}

Assume Case (b) that $\Sigma_{p_2}=[\epsilon(p_1),l,l-x]$. Suppose that $\epsilon(p_1)=+1$; the other case will be similar.
We blow up $p_1$, which is a connected sum of $M$ and $\overline{\mathbb{CP}^2}$ of Example \ref{e412} at $p_1$ and $q_0$.
Take the $T^k$-action on $\overline{\mathbb{CP}^2}$ in Example \ref{e412} with taking $a_1=x$ and $a_2=l$ and reverse its orientation. Then the right figure of Figure \ref{f22-2} describes the action on $\overline{\mathbb{CP}^2}$ with Property A. Because $\Sigma_{p_1}=-\Sigma_{q_0}$, we can take an equivariant connected sum at $p_1$ and $q_0$ of $M$ and $\overline{\mathbb{CP}^2}$ ($\Gamma$ and $\overline{\Gamma}_{\mathbb{CP}^2}$); let $M_1$ ($\Gamma_1$, which is Figure \ref{f22-3}) denote the connected sum. By Proposition \ref{p47}, $\Gamma_1$ (Figure \ref{f22-3}) describes $M_1$ with Property A. In $\Gamma_1$, $(p_2,q_2)$ is $l$-edge. Also, in $M_1$, the fixed points $p_2$ and $q_2$ have the opposite fixed point data; $\Sigma_{p_2}=[+,l,l-x]=-[-,l,l-x]=-\Sigma_{q_2}$. Next, take a self connected sum of $M_1$ ($\Gamma_1$) at $p_2$ and $q_2$ and let $M'$ ($\Gamma'$, which is Figure \ref{f22-4}) denote the manifold (multigraph) constructed. By Proposition \ref{p48}, Figure \ref{f22-4} describes $M'$ with Property A. On the level of multigraph, the reverse of this whole procedure is Operation (3).

\begin{figure} 
\centering
\begin{subfigure}[b][4.8cm][s]{.49\textwidth}
\centering
\vfill
\begin{tikzpicture}[state/.style ={circle, draw}]
\node[state] (a) {$p_1,+$};
\node[state] (b) [above=of a] {$p_2,+$};
\node[state] (c) [left=of a] {};
\node[state] (d) [left=of b] {};
\path (a) edge node[right] {$l$} (b);
\path (a) edge node[below] {$x$} (c);
\path (b) edge node[below] {$l-x$} (d);
\end{tikzpicture}
\caption{The multigraph $\Gamma$ describing $M$} \label{f22-1}
\vfill
\end{subfigure}
\begin{subfigure}[b][4.8cm][s]{.49\textwidth}
\centering
\vfill
\begin{tikzpicture}[state/.style ={circle, draw}]
\node[state] (a) at (0,0) {};
\node[state] (b) at (1.5,0) {$p_1,+$};
\node[state] (c) at (1.5,3) {$p_2,+$};
\node[state] (d) at (0,3) {};
\node[state] (e) at (3.5,0) {$q_0,-$};
\node[state] (f) at (5,1.5) {$q_1,+$};
\node[state] (g) at (3.5,3) {$q_2,-$};
\path (a) edge node[below] {$x$} (b);
\path (b) edge node[left] {$l$} (c);
\path (c) edge node[below] {$l-x$} (d);
\path (e) edge node[right] {$x$} (f);
\path (e) edge node[left] {$l$} (g);
\path (f) edge node [right] {$l-x$} (g);
\end{tikzpicture}
\caption{$\Gamma \sqcup \overline{\Gamma}_{\mathbb{CP}^2}$ describing $M \sqcup \overline{\mathbb{CP}^2}$} \label{f22-2}
\vfill
\end{subfigure}
\begin{subfigure}[b][4.8cm][s]{.49\textwidth}
\centering
\vfill
\begin{tikzpicture}[state/.style ={circle, draw}]
\node[state] (a) at (0,0) {};
\node[state] (c) at (1.5,3) {$p_2,+$};
\node[state] (d) at (0,3) {};
\node[state] (f) at (5,1.5) {$q_1,+$};
\node[state] (g) at (3.5,3) {$q_2,-$};
\path (a) edge node[below] {$x$} (f);
\path (c) edge node[below] {$l$} (g);
\path (d) edge node[below] {$l-x$} (c);
\path (g) edge node[right] {$l-x$} (f);
\end{tikzpicture}
\caption{The multigraph $\Gamma_1$ describing $M_1$, blow up of $\Gamma$ at $p_1$} \label{f22-3}
\vfill
\end{subfigure}
\begin{subfigure}[b][4cm][s]{.49\textwidth}
\centering
\vfill
\begin{tikzpicture}[state/.style ={circle, draw}]
\node[state] (a) at (0,0) {};
\node[state] (d) at (0,3) {};
\node[state] (f) at (5,1.5) {$q_1,+$};
\path (a) edge node[below] {$x$} (f);
\path (d) edge node[below] {$l-x$} (f);
\end{tikzpicture}
\caption{The multigraph $\Gamma'$ describing $M'$, self connected sum of $M_1$ at $p_1$ and $q_1$} \label{f22-4}
\vfill
\end{subfigure}
\caption{Case (b)}\label{f22}
\end{figure}

\textbf{Conclusion - We can successively reduce the number of vertices (fixed points) by self connected sums, blow ups, and reversing edges, which are the reverses of Operations (1-4)}

We have shown that by the steps above (by the reverses of Operations (1-4)), we can remove fixed points having a weight of the biggest magnitude (an edge having a label of the biggest magnitude) to construct another 4-dimensional compact connected oriented $T^k$-manifold $M'$ with fewer fixed points (a multigraph $\Gamma'$ with fewer vertices). Take the manifold $M'$ and the multigraph $\Gamma'$ describing $M'$ with Property A, and repeat the above. Continuing this, we get a 4-dimensional compact oriented $T^k$-manifold with no fixed points (the empty multigraph). This implies that $\Gamma$ is obtained by Operarions (1-4), beginning with the empty multigraph. \end{proof}

In the proof of Theorem \ref{t51} we used the following lemma.

\begin{lemma} \label{l53}
Let a torus $T^k$ act effectively on a 4-dimensional compact oriented manifold $M$ with isolated fixed points, whose isotropy submanifolds are orientable. Suppose that a fixed point $q$ has fixed point data $[\epsilon(q),l,a]$, where $\langle l, a \rangle \geq 0$. If $k=1$, we assume further that $|l|>1$. Then there exists another (unique) fixed point $q'$ in the same component of $M^{\ker l}$, which is the 2-sphere, so that one of the following holds for the fixed point data of $q'$:
\begin{enumerate}
\item $\Sigma_{q'}=[-\epsilon(q),l,a+m_1l]$ for some integer $m_1$.
\item $\Sigma_{q'}=[\epsilon(q),l,-a+m_2l]$ for some integer $m_2$.
\end{enumerate}
If furthermore $l$ is a weight of the biggest magnitude, that is, $|l|=\max\{|w_{p,i}|:p \in M^{T^k}, 1 \leq i \leq 2\}$, then $m_1=0$ and $m_2=1$.
\end{lemma}

\begin{proof}
First, we shall do the following: for every weight $w_{p,i}$ at a fixed point $p$, if the dot product of $w_{p,i}$ and $l$ is negative, replace $w_{p,i}$ with $-w_{p,i}$. If $k=1$ this can be also done by choosing all weights to be positive. In this way we have that $\langle w_{p,i},l \rangle \geq 0$ for all weights $w_{p,i}$.

Let $F$ be a connected component of $M^{\ker l}$ containing $q$, which is a compact submanifold of $M$. Choose an orientation of $F$; also choose an orientation of $NF$ so that the induced orientation on $TF \oplus NF$ agrees with the orientation of $M$.

Because every weight in $T_qF$ is a multiple of $l$ and $q$ has only one weight disivible by $l$ (the action is effective), $\dim F=2$. The $T^k$-action on $M$ restricts to act on $F$, and this action on $F$ has a fixed point $q$. By Lemma \ref{l211}, $F$ is the 2-sphere and has another (unique) fixed point $q'$. Applying Theorem \ref{t310} to the $T^k$-action on $F$, $\epsilon_F(q)=-\epsilon_F(q')$. Let $a$ be the weight of $N_qF$ and let $c$ be the weight of $N_{q'}F$.

First, suppose that $\epsilon(q)=\epsilon(q')$. Since $\epsilon_F(q)=-\epsilon_F(q')$, with $\epsilon(q)=\epsilon_F(q) \cdot \epsilon_N(q)$ and $\epsilon(q')=\epsilon_F(q') \cdot \epsilon_N(q')$ this implies that $\epsilon_N(q)=-\epsilon_N(q')$. Therefore, there is an orientation reversing isomorphism from $N_qF$ to $N_{q'}F$ as $\ker l$-representations. This implies that $a \equiv -c \mod l$, that is, $a=-c+m_1l$ for some integer $m_1$. This is Case (2) of this lemma.

Suppose that $|l|=\max\{|w_{p,i}|:p \in M^{T^k}, 1 \leq i \leq 2\}$. Because $l$, $a$, and $c$ are in the same half space $\{z \in \mathbb{R}^n \, | \, \langle z, l \rangle \geq 0\}$ of $\mathbb{R}^n$ and $l$ has the biggest magnitude, this implies that $m_1=1$, that is, $a=l-c$. Then $q'$ has fixed point data $[-\epsilon(q),l,l-a]$. 

Second, suppose that $\epsilon(q)=-\epsilon(q')$. Since $\epsilon_F(q)=-\epsilon_F(q')$, it follows that $\epsilon_N(q)=\epsilon_N(q')$. Hence, there is an orientation preserving isomorphism from $N_qF$ to $N_{q'}F$ as $\ker l$-representations. This implies that $a \equiv c \mod l$ and so $a=c+m_2l$ for some integer $m_2$. 

Suppose that $|l|=\max\{|w_{p,i}|:p \in M^{T^k}, 1 \leq i \leq 2\}$. Because $l$, $a$, and $c$ are in the same half space $\{z \in \mathbb{R}^n \, | \, \langle z, l \rangle \geq 0\}$ and $l$ has the biggest magnitude, this implies that $m_2=0$, that is, $a=c$. Then $q'$ has fixed point data $[-\epsilon(q),l,a]$. This is Case (1) of this lemma. \end{proof}

Theorem \ref{t51} implies the following theorem, which was obtained in \cite{J1} for circle actions.

\begin{theo}
Let a torus $T^k$ act effectively on a 4-dimensional compact oriented manifold $M$ with isolated fixed points, whose isotropy submanifolds are orientable. Then the fixed point data of $M$ can be achieved in the following way: begin with the empty collection, and apply a combination of the following steps.
\begin{enumerate}
\item Add $[+,a,b]$ and $[-,a,b]$, where $a$ and $b$ span $\mathbb{Z}^k$.
\item Replace $[+,c,d]$ with $[+,c,c+d]$ and $[+,d,c+d]$.
\item Replace $[-,e,f]$ with $[-,e,e+f]$ and $[-,f,e+f]$.
\end{enumerate} \end{theo}

\begin{figure} 
\centering
\begin{tikzpicture}[state/.style ={circle, draw}]
\node[state] (a) {$q_1,\pm$};
\node[state] (b) [above right=of a] {$q_2,\mp$};
\node[state] (c) [above left=of b] {$q_3,\pm$};
\path (a) edge node[right] {$a$} (b);
\path (a) edge node[left] {$a+b$} (c);
\path (b) edge node [right] {$b$} (c);
\end{tikzpicture}
\caption{3 fixed points}\label{f23}
\end{figure}
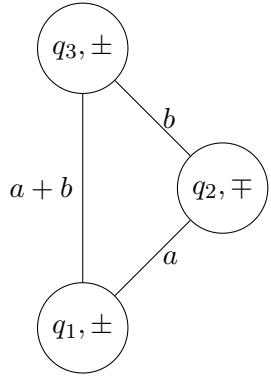

As an immediate consequence of Theorem \ref{t51}, a multigraph describing a 4-dimensional manifold with 3 fixed points is obtained.

\begin{theo}
Let a torus $T^k$ act effectively on a 4-dimensional compact oriented manifold $M$ with 3 fixed points. Then Figure \ref{f23} describes $M$ with Property A, for some non-zero $a,b$ that span $\mathbb{Z}^k$.
\end{theo}

A GKM manifold, named after Goresky, Kottwitz, and MacPherson, is a manifold equipped with a torus action with isolated fixed points such that the action is equivariantly formal, and weights at each fixed point are pairwise linearly independent \cite{GKM}. For a given GKM manifold, a labeled graph, called a GKM graph, is associated, and such a graph encodes the equivariant cohomology of the manifold. The GKM theory has been extensively studied since the appearance of \cite{GKM}. If a given GKM manifold is compact and oriented, our signed labeled multigraph describing it is a GKM graph if we forget signs of vertices.

Let a torus $T^k$ act on a 4-dimensional compact oriented manifold $M$ with isolated fixed points. Suppose that weights at each fixed point are linearl independent (for example, a GKM manifold) and we are only given a labeled multigraph describing $M$, without knowing signs of vertices (for example, a GKM graph). In this case we can also classify such a labeled multigraph as follows.

\begin{theo}[\textbf{Classification of graphs of 4-dimensional oriented GKM manifolds}] \label{t56}
Let a torus $T^k$ act effectively on a 4-dimensional compact oriented manifold $M$ with isolated fixed points, whose isotropy submanifolds are orientable. Suppose that weights at each fixed point are linearly independent (for example, $M$ is a GKM manifold). 
Suppose that we are only given a labeled multigraph $\Delta$ describing $M$ (for example, a GKM graph).
Let $e'$ be an edge whose label $l$ has the biggest magnitude, and let $p_1$ and $p_2$ be the vertices of $e'$. Reversing the label $w(e)$ of an edge $e$ by its negative $-w(e)$ if necessary, one of the below cases holds, and we can proceed as below.
\begin{enumerate}[(1)]
\item Labels of edges of $p_1$ and $p_2$ are both $\{l,x\}$. In this case, take a self connected sum of $\Delta$ (Definition \ref{d45}, ignoring signs). This step is going from Figure \ref{f20-1} to Figure \ref{f20-2} if there are two edges between $p_1$ and $p_2$, and from Figure \ref{f21-1} to Figure \ref{f21-2} if $e'$ is the only edge between $p_1$ and $p_2$.
\item Labels of edges of $p_1$ and $p_2$ are $\{l,x\}$ and $\{l,l-x\}$. In this case, blow up $p_1$ to convert $\Delta$ (Figure \ref{f22-1}) into Figure \ref{f22-3}, and then take a self connected sum of Figure \ref{f22-3} at $p_2$ and $q_3$. This step is going from Figure \ref{f22-1} to Figure \ref{f22-4}, ignoring signs of vertices.
\end{enumerate}
Repeating the above steps, we can convert $\Delta$ into the empty multigraph.

Conversely, $\Delta$ is obtained from the empty graph by successively applying Operations (1-4) of Theorem \ref{t51}, where $A,B$ are linearly independent, $C,D$ are linearly independent, and $E,F$ are linearly independent, ignoring the signs of vertices.
\end{theo}

\begin{proof}
Choose the sign of $p_1$. We may as well choose it to be $+1$. This determines the orientation of $M$. We show that we can determine the sign of $p_2$.

Consider a component $F$ of $M^{\ker l}$ containing $p_1$ and $p_2$. Since the weights at $p_1$ are linearly independent, $\dim F=2$; note that $l>2$ if $k=1$. By Theorem \ref{t310}, $\epsilon_F(p_1)=-\epsilon_F(p_2)$ for any orientation of $F$ if we choose weights in $F$ to be $l$.

Let $\Sigma_{p_1}=[+,l,x]$. By Lemma \ref{l53}, one of the following holds for the fixed point data $\Sigma_{p_2}$ of $p_2$.
\begin{enumerate}[(a)]
\item $\Sigma_{p_2}=[-,l,x]$.
\item $\Sigma_{p_2}=[+,l,l-x]$.
\end{enumerate}

Suppose that $x=l-x$. Then $p_1$ has weights $\{2x,x\}$ and so these weights are not linearly independent. Thus, $x \neq l-x$. Therefore, either $p_2$ has weights $\{l,x\}$ and $\epsilon(p_2)$ is -1, or $p_2$ has weights $\{l,l-x\}$ and $\epsilon(p_2)$ is +1. Then the rest of the proof is similar to the proof of Theorem \ref{t51}.

If $\Sigma_{p_2}=[-,l,x]$ then we take a connected sum of $M$ ($\Delta$) at $p_1$ and $p_2$ to construct another manifold $M'$ ($\Delta'$). Then $\Delta'$ describes $M'$ though is not signed. If there are two edges between $p_1$ and $p_2$, then $\Delta$ is Figure \ref{f20-1} and $\Delta'$ is Figure \ref{f20-2}. If $e'$ is the only edge between $p_1$ and $p_2$, $\Delta$ is Figure \ref{f21-1} and $\Delta'$ is Figure \ref{f21-2}.

If $\Sigma_{p_2}=[+,l,l-x]$ then we blow up $p$ of $M$ ($\Delta$) to construct another manifold $M_1$ ($\Delta_1$, which is Figure \ref{f22-3}). Then $\Delta_1$ describes $M_1$. Next take a self connected sum of $M_1$ ($\Delta_1$) to get another manifold $M'$ ($\Delta'$, which is Figure \ref{f22-4}), where $\Delta'$ describes $M'$ by Proposition \ref{p48}.

On the manifold $M'$, weights at each fixed point are still linearly independent (the fixed point $q_2$ in the second case is the only new fixed point and has weights $\{x,l-x\}$, which are linearly independent), and a non-signed labeled multigraph $\Delta'$ describes $M'$. Thus, on $M'$ and $\Delta'$ take an edge whose label has the biggest magnitude, and repeat the above procedure. Repeating the above arguments, we can convert $\Delta$ into the empty graph. \end{proof}

\section{Dimension 6: statements of the main results}\label{s6}

In the introduction we gave brief statements of our main results. With the terminologies understood, we give precise statements of Theorems \ref{t11} and \ref{t14}.

\begin{theo} \label{t61}
Let the circle group $S^1$ act effectively on a 6-dimensional compact oriented manifold $M$ with isolated fixed points, whose isotropy submanifolds are orientable.
Then there exists a signed labeled multigraph $\Gamma$ describing $M$ that satisfies Property A, and we can successively apply a combination of reversing edges and connected sums with itself, or with Figures \ref{f15-2}, \ref{f25-2}, \ref{f26-2}, \ref{f28} (multigraphs for $\mathbb{CP}^3$, $Z_1$, $Z_2$, $Z_2 \sharp \overline{Z_2}$) and these with all signs of vertices reversed, to convert $\Gamma$ to the empty multigraph. There is a definite procedure that this ends in a finite number of steps.
\end{theo}

A classification of the fixed point data of a circle action on a 6-dimensional compact oriented manifold with isolated fixed points is as follows.

\begin{theorem} \label{t62}
Let the circle group $S^1$ act on a 6-dimensional compact oriented manifold $M$ with isolated fixed points, whose isotropy submanifolds are orientable. To the fixed point data $\Sigma_M$ of $M$, we can apply a combination of the following operations to convert $\Sigma_M$ to the empty collection.
\begin{enumerate}
\item[(1)] Remove $[+,A,B,C]$ and $[-,A,B,C]$ together.
\item[(2)] Remove $[\pm,A,B,C]$ and $[\mp,C-A,C-B,C]$, and add $[\pm,A,B-A,C-A ]$ and $[\mp, B,B-A,C-B]$, where $0<A<B<C$.
\item[(3)] Remove $[\pm,A,B,C]$ and $[\pm,A,C-B,C]$, and add $[\pm, C-B,C-A,A]$, $[\pm, C-B,B,A]$, $[\pm, C-B,A-B,A]$, and $[\mp, C-A,A-B,A]$, where $0<A,B<C$ and $A \neq B$.
\item[(4)] Remove $[\pm,A,A,C]$ and $[\pm,A,C-A,C]$, and add $[\pm,C-A,C-2A,A]$, $[\pm,C-A,A,A]$, $[\pm,C-A,A,A]$, $[\mp,C-2A,A,A]$, where $0<A<C$.
\item[(5)] Remove $[\pm, C, A, A]$ and $[\mp, C, C-A, C-A]$, and add $[\pm,C-A,C-2A,A]$, $[\pm,C-A, A, A]$, $[\pm,C-A, A, A]$, $[\mp,C-2A, A, A]$, $[\pm, A, C-2A, C-A]$, $[\mp,A,C-A,C-A]$, $[\mp,A,C-A,C-A]$, $[\mp, C-2A, C-A, C-A]$, where $0<A<C$.
\end{enumerate}
There is a definite procedure that this ends in a finite number of steps.
\end{theorem}

The sign convention in Theorem \ref{t12} means the following: suppose that the fixed point data of $M$ contains $[-,A,B,C]$ and $[+,C-A,C-B,C]$. Then we can perform Operation (2) to remove these, and add $[-,A,B-A,C-A]$ and $[+,B,B-A,C-B]$. If the fixed point data of $M$ contains instead $[+,A,B,C]$ and $[-,C-A,C-B,C]$ (same fixed point data as above but with opposite signs), then we can perform Operation (2) to remove these and add $[+,A,B-A,C-A]$ and $[-,B,B-A,C-B]$.

The operations in Theorem \ref{t12} work in a way that they remove the weight $C$ with the biggest magnitude; every weight in what we add has magnitude strictly smaller than $|C|$. Therefore, the process of Theorem \ref{t12} stops in a finite number of steps. Operation (1) corresponds to a self connected sum, Operation (2) corresponds to a connected sum with $\mathbb{CP}^3$ (or $\overline{\mathbb{CP}^3}$), Operations (3) corresponds to a connected sum with the manifold $Z_1$ (or $\overline{Z_1}$), Operations (4) corresponds to a connected sum with the manifold $Z_2$ (or $\overline{Z_2}$), and Operation (5) corresponds to a connected sum with the manifold $Z_2 \sharp \overline{Z_2}$ in Example \ref{e74}, which is a connected sum at fixed points of two copies of $Z_2$, one with reversed orientation.

If a 6-dimensional compact oriented $S^1$-manifold has exactly 2 fixed points $p$ and $q$, the two fixed points have the opposite fixed point data, that is, $\Sigma_p=-\Sigma_q$ (see Theorem \ref{t310}); by performing Operation (1) of Theorem \ref{t62} once to the fixed point data of the manifold, we reach the empty collection.

\section{Dimension 6: standard models and their multigraphs} \label{s7}

In this section, we illustrate multigraphs describing torus actions on standard 6-dimensional oriented manifolds. These manifolds form a set of minimal models for our main results. All examples are in fact complex manifolds. We already discussed a torus action on the complex projective space $\mathbb{CP}^n$ and multigraphs describing it.

\begin{exa}[The manifold $Z_n$, 6-dimensional analogue of the Hirzebruch surfaces]\label{e71}

Fix an integer $n$. By the 6-dimensional analogue $Z_n$ of Hirzebruch surfaces we mean a compact complex manifold
\begin{center}
$Z_n=\{([z_0:z_1:z_2:z_3],[w_2:w_3]) \in \mathbb{CP}^3 \times \mathbb{CP}^1 \, : \, z_2 w_3^n=z_3 w_2^n\}$.
\end{center}
Fix $k \in \{1,2,3\}$. Let $a$, $b$, and $c$ be non-zero elements of $\mathbb{Z}^k$ such that $b-a \neq 0$, $nc-a \neq 0$, and $nc-b \neq 0$. Let a torus $T^k$ act on $Z_n$ by
\begin{center}
$g \cdot ([z_0:z_1:z_2:z_3],[w_2:w_3])=([g^a z_0: g^b z_1: z_2: g^{nc} z_3],[w_2:g^c w_3])$
\end{center}
for all $g \in T^k \subset \mathbb{C}^k$ and for all $([z_0:z_1:z_2:z_3],[w_2:w_3]) \in Z_n$. We denote by $Z_n(a,b,c)$ the manifold with this action.
The action has 6 fixed points; at each fixed point, we exhibit local coordinates and the weights at the fixed point as complex $T^k$-representations.
\begin{enumerate}
\item $q_1=([1:0:0:0],[1:0])$: local coordinates $(z_1/z_0, z_2/z_0, w_3/w_2)$, weights $\{b-a, -a, c\}$
\item $q_2=([1:0:0:0],[0:1])$: local coordinates $(z_1/z_0, z_3/z_0, w_2/w_3)$, weights $\{b-a, nc-a, -c\}$
\item $q_3=([0:1:0:0],[1:0])$: local coordinates $(z_0/z_1, z_2/z_1, w_3/w_2)$, weights $\{a-b, -b, c\}$
\item $q_4=([0:1:0:0],[0:1])$: local coordinates $(z_0/z_1, z_3/z_1, w_2/w_3)$, weights $\{a-b, nc-b, -c\}$
\item $q_5=([0:0:1:0],[1:0])$: local coordinates $(z_0/z_2, z_1/z_2, w_3/w_2)$, weights $\{a, b, c\}$
\item $q_6=([0:0:0:1],[0:1])$: local coordinates $(z_0/z_3, z_1/z_3, w_2/w_3)$, weights $\{a-nc, b-nc, -c\}$
\end{enumerate}
For instance, local coordinates of $q_2$ are $(z_1/z_0, z_3/z_0, w_2/w_3)$, and the torus $T^k$ acts near $q_2$ by
\begin{center}
$\displaystyle g \cdot \left( \frac{z_1}{z_0}, \frac{z_3}{z_0}, \frac{w_2}{w_3}\right)=\left(  \frac{g^b z_1}{g^a z_0}, \frac{g^{nc} z_3}{g^a z_0}, \frac{w_2}{g^c w_3}\right)=\left(g^{b-a} \frac{z_1}{z_0}, g^{nc-a} \frac{z_3}{z_0}, g^{-c} \frac{w_2}{w_3}\right)$.
\end{center}
Thus the complex $T^k$-weights at $q_2$ are $\{b-a,nc-a,-c\}$. With the orientation on $Z_n$ induced from the complex structure, the fixed point data of $q_2$ is $[-,b-a,nc-a,c]$. The fixed point data of $Z_n$ as an oriented manifold is
\begin{center}
$[+,a-b,a,c], [-,a-b,a-nc,c], [-,a-b,b,c], [-,a-b,nc-b,c], [+,a,b,c], [+,a-nc,nc-b,c]$.
\end{center}
Figure \ref{f24-1} describes $Z_n$ as a directed labeled multigraph, and Figure \ref{f24-2} describes $Z_n$ as a signed labeled multigraph.

As Figure \ref{f15-2} satisfies Property A, Figure \ref{f24-2} also satisfies Property A. For example, $q_1$ has weight $c$, $q_2$ has weight $-c$, $q_1$ and $q_2$ are in the same 2-sphere $F:=\{([1:0:0:0],[w_2:w_3])\}$ of $M^{\ker c}$. With the induced orientation on $F$, the fixed point data  of $q_1$ and $q_2$ for the action on $F$ are $[+,c]$ and $[-,c]$, respectively. Therefore,  $\epsilon_F(q_1)=+1=-\epsilon_F(q_2)$. \end{exa}

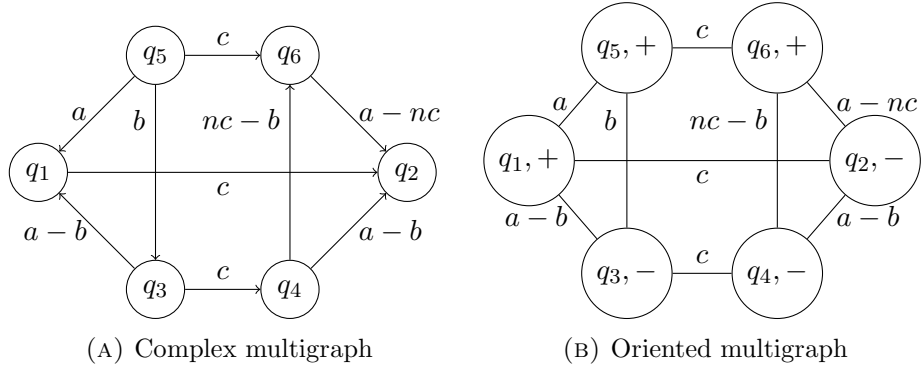
\begin{figure} 
\centering
\begin{subfigure}[b][5cm][s]{.49\textwidth}
\centering
\vfill
\begin{tikzpicture}[state/.style ={circle, draw}]
\node[state] (a) {$q_1$};
\node[state] (c) [below right=of a] {$q_3$};
\node[state] (d) [right=of c] {$q_4$};
\node[state] (b) [above right=of d] {$q_2$};
\node[state] (e) [above right=of a] {$q_5$};
\node[state] (f) [right=of e] {$q_6$};
\path (a) [->] edge node[below] {$c$} (b);
\path (c) [->] edge node[left] {$a-b$} (a);
\path (e) [->] edge node[left] {$a$} (a);
\path (d) [->] edge node [right] {$a-b$} (b);
\path (f) [->] edge node [right] {$a-nc$} (b);
\path (c) [->] edge node [above] {$c$} (d);
\path (d) [->] edge node [pos=.8, left] {$nc-b$} (f);
\path (e) [->] edge node [pos=.2, left] {$b$} (c);
\path (e) [->] edge node [above] {$c$} (f);
\end{tikzpicture}
\caption{Complex multigraph} \label{f24-1}
\vfill
\end{subfigure}
\begin{subfigure}[b][5cm][s]{.49\textwidth}
\centering
\vfill
\begin{tikzpicture}[state/.style ={circle, draw}]
\node[state] (a) at (0,0) {$q_1,+$};
\node[state] (c) at (1.3, -1.5) {$q_3,-$};
\node[state] (d) at (3.3, -1.5) {$q_4,-$};
\node[state] (b) at (4.6, 0) {$q_2,-$};
\node[state] (e) at (1.3, 1.5) {$q_5,+$};
\node[state] (f) at (3.3, 1.5) {$q_6,+$};
\path (a) edge node[below] {$c$} (b);
\path (a) edge node[left] {$a-b$} (c);
\path (e) edge node[left] {$a$} (a);
\path (b) edge node [right] {$a-b$} (d);
\path (b) edge node [right] {$a-nc$} (f);
\path (c) edge node [above] {$c$} (d);
\path (d) edge node [pos=.8, left] {$nc-b$} (f);
\path (e) edge node [pos=.2, left] {$b$} (c);
\path (e) edge node [above] {$c$} (f);
\end{tikzpicture}
\caption{Oriented multigraph} \label{f24-2}
\vfill
\end{subfigure}
\caption{Multigraph describing a $T^k$-action on $Z_n$}\label{f24}
\end{figure}

To simplify the proof of Theorems \ref{t11} and \ref{t12}, in Example \ref{e71} we will take specific values of $a$, $b$, $c$, and $n$ and record them as separate examples.

\begin{exa}[The manifold $Z_1(a,b,c)$ with $a \neq b$, $a \neq c$, $b \neq c$]\label{e72}
Suppose that $a,b,c$ are mutually distinct. Take the manifold $Z_1(a,b,c)$ in Example \ref{e71}. The weights at the fixed points as complex $T^k$-representations are
\begin{center}
$\{b-a, -a, c\}$, $\{b-a, c-a, -c\}$, $\{a-b, -b, c\}$, $\{a-b, c-b, -c\}$, $\{a, b, c\}$, $\{a-c, b-c, -c\}$,
\end{center} 
respectively. Therefore, as an oriented real $T^k$-representations, the fixed point data of $Z_1(a,b,c)$ is
\begin{center}
$[+, a-b, a, c]$, $[-, a-b, a-c, c]$, $[-, a-b, b, c]$, $[-, a-b, c-b, c]$, $[+, a, b, c]$, $[+, a-c, c-b, c]$.
\end{center}
Figure \ref{f25-1} describes $Z_1(a,b,c)$ as a directed labeled multigraph and Figure \ref{f25-2} describes $Z_1(a,b,c)$ as a signed labeled multigraph with Property A.
\end{exa}

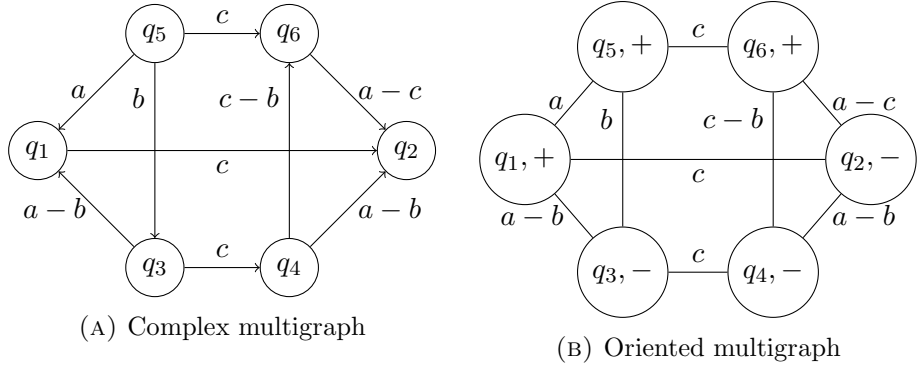
\begin{figure} 
\centering
\begin{subfigure}[b][4.5cm][s]{.49\textwidth}
\centering
\vfill
\begin{tikzpicture}[state/.style ={circle, draw}]
\node[state] (a) {$q_1$};
\node[state] (c) [below right=of a] {$q_3$};
\node[state] (d) [right=of c] {$q_4$};
\node[state] (b) [above right=of d] {$q_2$};
\node[state] (e) [above right=of a] {$q_5$};
\node[state] (f) [right=of e] {$q_6$};
\path (a) [->] edge node[below] {$c$} (b);
\path (c) [->] edge node[left] {$a-b$} (a);
\path (e) [->] edge node[left] {$a$} (a);
\path (d) [->] edge node [right] {$a-b$} (b);
\path (f) [->] edge node [right] {$a-c$} (b);
\path (c) [->] edge node [above] {$c$} (d);
\path (d) [->] edge node [pos=.8, left] {$c-b$} (f);
\path (e) [->] edge node [pos=.2, left] {$b$} (c);
\path (e) [->] edge node [above] {$c$} (f);
\end{tikzpicture}
\caption{Complex multigraph} \label{f25-1}
\vfill
\end{subfigure}
\begin{subfigure}[b][4.5cm][s]{.49\textwidth}
\centering
\vfill
\begin{tikzpicture}[state/.style ={circle, draw}]
\node[state] (a) at (0,0) {$q_1,+$};
\node[state] (c) at (1.3,-1.5) {$q_3,-$};
\node[state] (d) at (3.3, -1.5) {$q_4,-$};
\node[state] (b) at (4.6, 0) {$q_2,-$};
\node[state] (e) at (1.3, 1.5) {$q_5,+$};
\node[state] (f) at (3.3, 1.5) {$q_6,+$};
\path (a) edge node[below] {$c$} (b);
\path (a) edge node[left] {$a-b$} (c);
\path (e) edge node[left] {$a$} (a);
\path (b) edge node [right] {$a-b$} (d);
\path (b) edge node [right] {$a-c$} (f);
\path (c) edge node [above] {$c$} (d);
\path (d) edge node [pos=.8, left] {$c-b$} (f);
\path (e) edge node [pos=.2, left] {$b$} (c);
\path (e) edge node [above] {$c$} (f);
\end{tikzpicture}
\caption{Oriented multigraph} \label{f25-2}
\vfill
\end{subfigure}
\caption{Multigraph describing a $T^k$-action on $Z_1$}\label{f25}
\end{figure}

\begin{exa}[The manifold $Z_2(a,d,d)$ with $2d \neq a$]\label{e73}
Let $a$ and $d$ be non-zero elements of $\mathbb{Z}^k$ such that $2d \neq a$. In Example \ref{e71}, take $n=2$ and $b=c=d$. Complex $T^k$-weights at the fixed points are
\begin{center}
$\{d-a, -a, d\}$, $\{d-a, 2d-a, -d\}$, $\{a-d, -d, d\}$, $\{a-d, d, -d\}$, $\{a, d, d\}$, $\{a-2d, -d, -d\}$.
\end{center} 
As real $T^k$-representations, the fixed point data of $Z_2(a,d,d)$ is
\begin{center}
$[+, a-d, a, d]$, $[-,a-d,a-2d,d]$, $[-,a-d, d, d]$, $[-,a-d, d, d]$, $[+, a, d, d]$, $[+,a-2d, d, d]$.
\end{center}
Figure \ref{f26-1} describes $Z_2(a,d,d)$ as a directed labeled multigraph and Figure \ref{f26-2} describes $Z_2(a,d,d)$ as a signed labeled multigraph with Property A.
\end{exa}

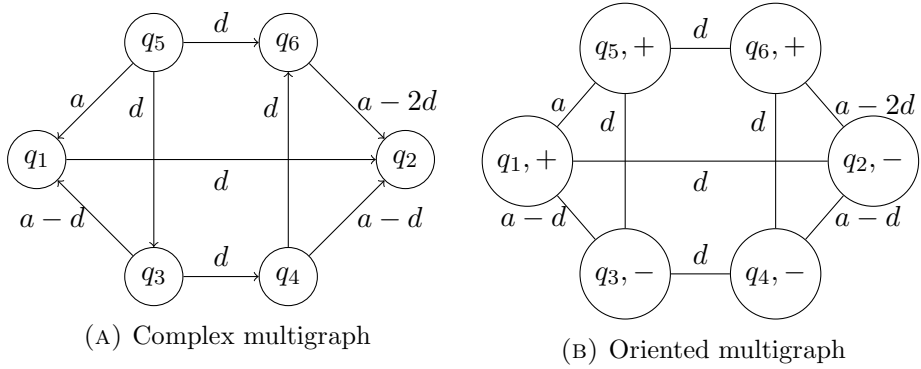
\begin{figure} 
\centering
\begin{subfigure}[b][4.5cm][s]{.49\textwidth}
\centering
\vfill
\begin{tikzpicture}[state/.style ={circle, draw}]
\node[state] (a) {$q_1$};
\node[state] (c) [below right=of a] {$q_3$};
\node[state] (d) [right=of c] {$q_4$};
\node[state] (b) [above right=of d] {$q_2$};
\node[state] (e) [above right=of a] {$q_5$};
\node[state] (f) [right=of e] {$q_6$};
\path (a) [->] edge node[below] {$d$} (b);
\path (c) [->] edge node[left] {$a-d$} (a);
\path (e) [->] edge node[left] {$a$} (a);
\path (d) [->] edge node [right] {$a-d$} (b);
\path (f) [->] edge node [right] {$a-2d$} (b);
\path (c) [->] edge node [above] {$d$} (d);
\path (d) [->] edge node [pos=.8, left] {$d$} (f);
\path (e) [->] edge node [pos=.2, left] {$d$} (c);
\path (e) [->] edge node [above] {$d$} (f);
\end{tikzpicture}
\caption{Complex multigraph} \label{f26-1}
\vfill
\end{subfigure}
\begin{subfigure}[b][4.5cm][s]{.49\textwidth}
\centering
\vfill
\begin{tikzpicture}[state/.style ={circle, draw}]
\node[state] (a) at (0,0) {$q_1,+$};
\node[state] (c) at (1.3,-1.5) {$q_3,-$};
\node[state] (d) at (3.3, -1.5) {$q_4,-$};
\node[state] (b) at (4.6, 0) {$q_2,-$};
\node[state] (e) at (1.3, 1.5) {$q_5,+$};
\node[state] (f) at (3.3, 1.5) {$q_6,+$};
\path (a) edge node[below] {$d$} (b);
\path (a) edge node[left] {$a-d$} (c);
\path (e) edge node[left] {$a$} (a);
\path (b) edge node [right] {$a-d$} (d);
\path (b) edge node [right] {$a-2d$} (f);
\path (c) edge node [above] {$d$} (d);
\path (d) edge node [pos=.8, left] {$d$} (f);
\path (e) edge node [pos=.2, left] {$d$} (c);
\path (e) edge node [above] {$d$} (f);
\end{tikzpicture}
\caption{Oriented multigraph} \label{f26-2}
\vfill
\end{subfigure}
\caption{Multigraph describing a $T^k$-action on $Z_2$}\label{f26}
\end{figure}

\begin{exa}[The manifold $Z_2(a,e,e) \sharp \overline{Z_2}(a,a-e,a-e)$]\label{e74}
Let $a$ and $e$ be non-zero elements of $\mathbb{Z}^k$ such that $2e \neq a$. Take $Z_2(a,e,e)$ of Example \ref{e73} that has fixed point data
\begin{center}
$[+, a-e, a, e]$, $[-,a-e,a-2e,e]$, $[-,a-e, e, e]$, $[-,a-e, e, e]$, $[+, a, e, e]$, $[+,a-2e, e, e]$.
\end{center}
Denote the fixed points by $q_1',\cdots,q_6'$, respectively. We also take $Z_2(a,a-e,a-e)$ (take $d=a-e$) of Example \ref{e73} that has fixed point data
\begin{center}
$[+, e, a, a-e]$, $[+, e, a-2e, a-e]$, $[-, e, a-e, a-e]$, $[-, e, a-e, a-e]$, $[+, a, a-e, a-e]$, $[-, a-2e, a-e, a-e]$.
\end{center}
We reverse the orientation of $Z_2(a,a-e,a-e)$ to get a manifold $\overline{Z_2}(a,a-e,a-e)$ that has fixed point data
\begin{center}
$[-, e, a, a-e]$, $[-, e, a-2e, a-e]$, $[+,e,a-e,a-e]$, $[+,e,a-e,a-e]$, $[-, a, a-e, a-e]$, $[+, a-2e, a-e, a-e]$.
\end{center}
Denote the fixed points by $q_1'',\cdots,q_6''$, respectively.
Now, $q_1'$ and $q_1''$ have the same weights (as real $T^k$-representations) and satisfy $\epsilon(q_1')=-\epsilon(q_1'')$. Therefore, we can take an equivariant connected sum at $q_1'$ of $Z_2(a,e,e)$ and at $q_1''$ of $\overline{Z_2}(a,a-e,a-e)$ to construct another 6-dimensional compact connected oriented $T^k$-manifold $Z_2(a,e,e) \sharp \overline{Z_2}(a,a-e,a-e)$ with 10 fixed points $q_2'$, $\cdots$, $q_6'$, $q_2''$, $\cdots$, $q_6''$, that has fixed point data
\begin{center}
$[-,a-e,a-2e,e]$, $[-,a-e, e, e]$, $[-,a-e, e, e]$, $[+, a, e, e]$, $[+,a-2e, e, e]$, $[-, e, a-2e, a-e]$, $[+,e,a-e,a-e]$, $[+,e,a-e,a-e]$, $[-, a, a-e, a-e]$, $[+, a-2e, a-e, a-e]$.
\end{center}
Since Figure \ref{f26-2} describes $Z_2$ with Property A and $Z_2 \sharp \overline{Z_2}$ is constructed by taking a connected sum of two copies of $Z_2$ (one with reversed orientation), by Proposition \ref{p47} Figure \ref{f28} describes $Z_2 \sharp \overline{Z_2}$ with Property A.
\end{exa}

\begin{figure} 
\centering
\begin{tikzpicture}[state/.style ={circle, draw}]
\node[state] (a) at (0,0) {$q_1',+$};
\node[state] (c) at (-1.3,-1.5) {$q_3',-$};
\node[state] (d) at (-3.3,-1.5) {$q_4',-$};
\node[state] (b) at (-4.6,0) {$q_2',-$};
\node[state] (e) at (-1.3,1.5) {$q_5',+$};
\node[state] (f) at (-3.3,1.5) {$q_6',+$};
\node[state] (g) at (1.5,0) {$q_1'',-$};
\node[state] (i) at (2.8,-1.5) {$q_3'',+$};
\node[state] (j) at (4.8,-1.5) {$q_4'',+$};
\node[state] (h) at (6.1,0) {$q_2'',-$};
\node[state] (k) at (2.8,1.5) {$q_5'',-$};
\node[state] (l) at (4.8,1.5) {$q_6'',+$};
\path (a) edge node[below] {$e$} (b);
\path (a) edge node[right] {$a-e$} (c);
\path (e) edge node[above] {$a$} (a);
\path (b) edge node [left] {$a-e$} (d);
\path (b) edge node [left] {$a-2e$} (f);
\path (c) edge node [below] {$e$} (d);
\path (d) edge node [pos=.8, right] {$e$} (f);
\path (e) edge node [pos=.2, right] {$e$} (c);
\path (e) edge node [above] {$e$} (f);
\path (g) edge node[below] {$a-e$} (h);
\path (g) edge node[left] {$e$} (i);
\path (k) edge node[left] {$a$} (g);
\path (h) edge node [right] {$e$} (j);
\path (h) edge node [right] {$a-2e$} (l);
\path (i) edge node [below] {$a-e$} (j);
\path (j) edge node [pos=.8, left] {$a-e$} (l);
\path (k) edge node [pos=.2, left] {$a-e$} (i);
\path (k) edge node [above] {$a-e$} (l);
\end{tikzpicture}
\caption{Multigraph for $Z_2(a,e,e) \sqcup \overline{Z_2}(a,a-e,a-e)$}\label{f27}
\end{figure}
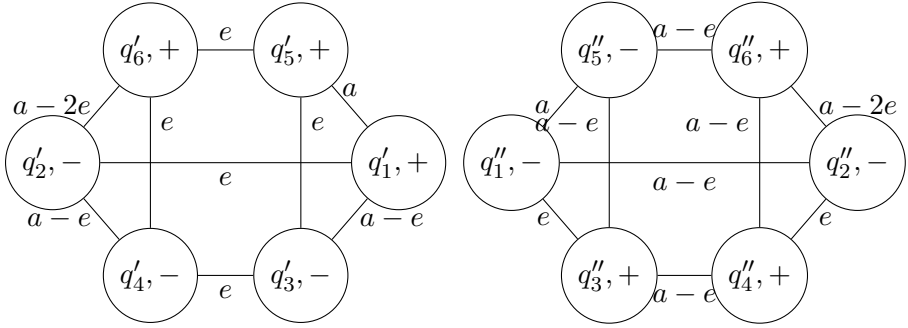

\begin{figure} 
\centering
\begin{tikzpicture}[state/.style ={circle, draw}]
\node[state] (c) at (-1.3,-1.5) {$q_3',-$};
\node[state] (d) at (-3.3,-1.5) {$q_4',-$};
\node[state] (b) at (-4.6,0) {$q_2',-$};
\node[state] (e) at (-1.3,1.5) {$q_5',+$};
\node[state] (f) at (-3.3,1.5) {$q_6',+$};
\node[state] (i) at (1.8,-1.5) {$q_3'',+$};
\node[state] (j) at (3.8,-1.5) {$q_4'',+$};
\node[state] (h) at (5.1,0) {$q_2'',-$};
\node[state] (k) at (1.8,1.5) {$q_5'',-$};
\node[state] (l) at (3.8,1.5) {$q_6'',+$};
\path (e) edge node[above] {$a$} (k);
\path (b) edge node[above] {$e$} (i);
\path (c) edge node[above] {$a-e$} (h);
\path (b) edge node [left] {$a-e$} (d);
\path (b) edge node [left] {$a-2e$} (f);
\path (c) edge node [below] {$e$} (d);
\path (d) edge node [pos=.8, right] {$e$} (f);
\path (e) edge node [pos=.2, right] {$e$} (c);
\path (e) edge node [above] {$e$} (f);
\path (h) edge node [right] {$e$} (j);
\path (h) edge node [right] {$a-2e$} (l);
\path (i) edge node [below] {$a-e$} (j);
\path (j) edge node [pos=.8, left] {$a-e$} (l);
\path (k) edge node [pos=.2, left] {$a-e$} (i);
\path (k) edge node [above] {$a-e$} (l);
\end{tikzpicture}
\caption{Multigraph for $Z_2(a,e,e) \sharp \overline{Z_2}(a,a-e,a-e)$}\label{f28}
\end{figure}

\section{Dimension 6: fixed point data in isotropy submanifold} \label{s8}

In this section, for a torus action on a 6-dimensional oriented manifold, we discuss a relation between fixed point datum of fixed points that are in the same component of a set fixed by some subgroup of the torus, an isotropy submanifold.

Let a torus $T^k$ act on a 6-dimensional oriented manifold $M$ with isolated fixed points, whose isotropy submanifolds are orientable. Let $l$ be a weight with the biggest magnitude. First, assume that $M^{\ker l}$ has a 2-dimensional component $F$ that contains a $T^k$-fixed point $q$. Then $F$ is the 2-sphere, and $F$ contains another fixed point $q'$. A relationship between the fixed point data of $q$ and that of $q'$ is as follows.

\begin{lemma} \label{l81}
Let a torus $T^k$ act effectively on a 6-dimensional compact oriented manifold $M$ with isolated fixed points, whose isotropy submanifolds are orientable. Suppose that a fixed point $q$ has fixed point data $[\epsilon(q),l,a,b]$, where $|l|=\max\{|w_{p,i}|:p \in M^{T^k}, 1 \leq i \leq 3\}$, $a$ and $b$ are not integer multiples of $l$, $\langle l,a \rangle \geq 0$, and $\langle l,b \rangle \geq 0$. Then there exists another (unique) fixed point $q'$ so that one of the following holds for the fixed point data of $q'$:
\begin{enumerate}
\item $\Sigma_{q'}=[-\epsilon(q),l,a,b]$.
\item $\Sigma_{q'}=[-\epsilon(q),l,l-a,l-b]$.
\item $\Sigma_{q'}=[\epsilon(q),l,a,l-b]$.
\item $\Sigma_{q'}=[\epsilon(q),l,l-a,b]$.
\end{enumerate}
The fixed points $q$ and $q'$ are in the same component of $M^{\ker l}$, which is the 2-sphere. 
\end{lemma}

\begin{proof}
First, we shall do the following: for every weight $w_{p,i}$ at a fixed point $p$, if the dot product of $w_{p,i}$ and $l$ is negative, replace $w_{p,i}$ with $-w_{p,i}$ (that is, reverse the orientation of $L_{p,i}$). If $k=1$ this can be also done by choosing all weights to be positive. In this way we have that $\langle w_{p,i},l \rangle \geq 0$ for all weights $w_{p,i}$.

Let $F$ be a connected component of $M^{\ker l}$ that contains $q$, which is a smaller dimensional compact submanifold of $M$. Choose an orientation of $F$; also choose an orientation of $NF$ so that the induced orientation on $TF \oplus NF$ agrees with the orientation of $M$. Because $q$ has only one weight disivible by $l$, $\dim F=2$. The $T^k$-action on $M$ restricts to act on $F$, and this action on $F$ has a fixed point $q$. By Lemma \ref{l211}, $F$ is the 2-sphere and has another fixed point $q'$. Applying Theorem \ref{t310} to the $T^k$-action on $F$, $\epsilon_F(q)=-\epsilon_F(q')$.

Let $N_qF=L_{q,2} \oplus L_{q,3}$, where the torus $T^k$ acts on $L_{q,2}$ with weight $a$ and on $L_{q,3}$ with weight $b$. Similarly, let $N_{q'}F=L_{q',2} \oplus L_{q',3}$, where $T^k$ acts on $L_{q',2}$ with weight $c$ and on $L_{q',3}$ with weight $d$ for some $c$ and $d$.

First, suppose that $\epsilon(q)=\epsilon(q')$. Since $\epsilon_F(q)=-\epsilon_F(q')$, with $\epsilon(q)=\epsilon_F(q) \cdot \epsilon_N(q)$ and $\epsilon(q')=\epsilon_F(q') \cdot \epsilon_N(q')$ this implies that $\epsilon_N(q)=-\epsilon_N(q')$. Therefore, there is an orientation reversing isomorphism $\phi$ from $L_{q,2} \oplus L_{q,3}$ to $L_{q',2} \oplus L_{q',3}$ as $\ker l$-representations. Without loss of generality, by permuting $L_{q',2}$ and $L_{q',3}$ if necessary, we may assume that this isomorphism takes $L_{q,2}$ to $L_{q',2}$ and $L_{q,3}$ to $L_{q',3}$. Then one of the following holds.

\begin{enumerate}[(a)]
\item The isomorphism $\phi$ is orientation preserving from $L_{q,2}$ to $L_{q',2}$, and orientation reversing from $L_{q,3}$ to $L_{q',3}$.
\item The isomorphism $\phi$ is orientation reversing from $L_{q,2}$ to $L_{q',2}$, and orientation preserving from $L_{q,3}$ to $L_{q',3}$.
\end{enumerate}

Assume that Case (a) holds. Since $\phi$ is an orientation preserving isomorphism from $L_{q,2}$ to $L_{q',2}$ as $\ker l$-representations, this implies that $a \equiv c \mod l$, that is, $a=c+m_1l$ for some integer $m_1$. Because $l$, $a$, and $c$ are in the same half space $\{z \in \mathbb{R}^n \, | \, \langle z, l \rangle \leq 0\}$ of $\mathbb{R}^n$ and $l$ has the biggest magnitude, this implies that $m_1=0$, that is, $a=c$.

Next, since $\phi$ is an orientation reserving isomorphism from $L_{q,3}$ to $L_{q',3}$ as $\ker l$-representations, this implies that $b \equiv -d \mod l$. Because $l$, $b$, and $d$ are in the same half space $\{z \in \mathbb{R}^n \, | \, \langle z, l \rangle \geq 0\}$ and $l$ has the biggest magnitude, this implies that $-d=b-l$, that is, $d=l-b$. This is Case (3) of this lemma. Similarly, if Case (b) holds, then $c=l-a$ and $b=d$; this is Case (4) of this lemma.

Second, suppose that $\epsilon(q)=-\epsilon(q')$. Since $\epsilon_F(q)=-\epsilon_F(q')$, it follows that $\epsilon_N(q)=\epsilon_N(q')$. Hence, there is an orientation preserving isomorphism $\phi$ from $L_{q,2} \oplus L_{q,3}$ to $L_{q',2} \oplus L_{q',3}$ as $\mathbb{Z}_l$-representations. We may assume that this isomorphism takes $L_{q,2}$ to $L_{q',2}$ and $L_{q,3}$ to $L_{q',3}$. There are two possibilities.
\begin{enumerate}[(i)]
\item The isomorphism $\phi$ is orientation preserving from $L_{q,i}$ to $L_{q',i}$ for both $i \in \{2,3\}$.
\item The isomorphism $\phi$ is orientation reversing from $L_{q,i}$ to $L_{q',i}$ for both $i \in \{2,3\}$.
\end{enumerate}
Case (i) means that $a=c$ and $b=d$; this is Case (1) of this lemma. Case (ii) means that $c=l-a$ and $d=l-b$; this is Case (2) of this lemma. 
\end{proof}

Suppose now that the dimension of $M$ is arbitrary but at least 6, and $M^{\ker l}$ has a  component of codimension 2, which contains a $T^k$-fixed point $q$. In other words, the multiplicity of $\pm l$ in $T_qM$ is $n-1$, where $\dim M=2n$. Then the component contains another fixed point $q'$ that has the opposite fixed point data of $q$; $\Sigma_q=\Sigma_{q'}$. 

\begin{lemma} \label{l82}
Let $n \geq 3$. Let a torus $T^k$ act effectively on a $2n$-dimensional compact oriented manifold $M$ with isolated fixed points, whose isotropy submanifolds are orientable. Let $l$ be a weight with the biggest magnitude; $|l|=\max\{|w_{p,i}| : 1 \leq i \leq n, p \in M^{T^k}\}$. If $k=1$, we assume further that $|l|>1$.
\begin{enumerate}[(1)]
\item Suppose that there is a fixed point $q$ that has fixed point data $[\pm,l,\cdots,l,a]$ for some $a \in \mathbb{Z}^k \setminus \{0\}$. Then there exists another fixed point $q'$ that has fixed point data $-\Sigma_q$. Moreover, $q$ and $q'$ are in the same component of $M^{\ker l}$.
\item For each $a \in \mathbb{Z}^k \setminus \{0\}$, the number of fixed points with fixed point data $[+,l,\cdots,l,a]$ is equal to the number of fixed points with fixed point data $[-,l,\cdots,l,a]$.
\end{enumerate}
\end{lemma}

\begin{proof}
As in the proof of Lemma \ref{l81}, for every weight $w_{p,i}$ at a fixed point $p$, if $\langle w_{p,i},l \rangle <0$, replace $w_{p,i}$ with $-w_{p,i}$.

Suppose that a fixed point $p_1:=q$ has fixed point data $[\pm,l,\cdots,l,a]$ for some $a$. Let $F$ be a component of $M^{\ker l}$ that contains $p_1$; $F$ is a $(2n-2)$-dimensional compact orientable submanifold of $M$. Choose an orientation of $F$, so that $\epsilon_F(p_1)=+1$ for simplicity of the proof. Also choose an orientation of $NF$ so that the induced orientation on $TF \oplus NF$ agrees with the orientation of $M$. The $T^k$-action on $M$ restricts to act on $F$, and the fixed point set $F^{T^k}$ of this action on $F$ is equal to $F \cap M^{T^k}$ so it is non-empty and finite. For any $p \in F^{T^k}$, the weights in $T_pF$ are all integer multiples of $l$, and hence they are all equal to $l$ as $l$ is the biggest weight. Applying Lemma \ref{l39} to the induced $T^k$-action on $F$ by taking $a_1=\cdots=a_{n-1}=l$, the number of fixed points $p \in F^{T^k}$ with $\epsilon_F(p)=+1$ and the number of fixed points $p \in F^{T^k}$ with $\epsilon_F(p)=-1$ are equal. Let $p_1,\cdots,p_k \in F^{T^k}$ have $\epsilon_F(p_i)=+1$ and let $q_1,\cdots,q_k \in F^{T^k}$ have $\epsilon_F(q_i)=-1$. 

By permuting $p_i$'s if necessary, let $\epsilon(p_1)=\cdots=\epsilon(p_s)$ and $\epsilon(p_{s+1})=\cdots=\epsilon(p_k)=-\epsilon(p_1)$. Similarly, by permuting $q_j$'s if necessary, let $\epsilon(q_1)=\cdots=\epsilon(q_t)=-\epsilon(p_1)$ and $\epsilon(q_{t+1})=\cdots=\epsilon(q_k)=\epsilon(p_1)$.

Let $\{l,\cdots,l,a_i\}$ be the weights at $p_i$ and let $\{l,\cdots,l,b_j\}$ be the weights at $q_j$. Since $NF$ is an oriented $\ker l$-bundle over $F$ and $F$ is connected, the $\ker l$-representations of $N_{p}F$ and $N_{p'}F$ are isomorphic for any two fixed points $p$ and $p'$ in $F^{T^k}$.

Consider $p_i$ for $1 \leq i \leq s$ ($s+1 \leq i \leq k$.) Since $\epsilon(p_i)=\epsilon(p_1)$ ($\epsilon(p_i)=-\epsilon(p_1)$) and $\epsilon_F(p_i)=\epsilon_F(p_1)$ ($\epsilon_F(p_i)=\epsilon_F(p_1)$), it follows that $\epsilon_N(p_i)=\epsilon_N(p_1)$ ($\epsilon_N(p_i)=-\epsilon_N(p_1)$). Thus, there is an orientation preserving (reversing) isomorphism from the representation $N_{p_1}F$ with weight $a$ to that $N_{p_i}F$ with weight $a_i$ as $\ker l$-representations. This implies that $a \equiv a_i \mod l$ ($a \equiv -a_i \mod l$). Because $l$, $a$, and $a_i$ are in the same half space $\{z \in \mathbb{R}^n \, | \, \langle z, l \rangle  \geq 0\}$ of $\mathbb{R}^n$ and $l$ has the biggest magnitude, it follows that $a=a_i$ ($a_i=l-a$, respectively).

Similarly, for $q_j$ with $1 \leq j \leq t$ ($t+1 \leq j \leq k$), $\epsilon(q_j)=-\epsilon(p_1)$ ($\epsilon(q_j)=\epsilon(p_1)$) and $\epsilon_F(q_j)=-\epsilon_F(p_1)$ imply that $\epsilon_N(q_j)=\epsilon_N(p_1)$ ($\epsilon_N(q_j)=-\epsilon_N(p_1)$) and hence there is an orientation preserving (reversing) isomorphism from $N_{p_1}F$ with weight $a$ to $N_{q_j}F$ with weight $b_j$ as $\mathbb{Z}_l$-representations, and thus $a \equiv b_j \mod l$ ($a \equiv -b_j \mod l$), that is, $b_j=a$ ($b_j=l-a$, respectively.)

Let $e_{T^k}(NF)$ denote the equivariant Euler class of the normal bundle to $F$ in $M$. The restriction of $e_{T^k}(NF)$ at $p_i$ is 
\begin{enumerate}
\item $e_{T^k}(NF)(p_i)=\epsilon_N(p_i) \cdot a=\epsilon_N(p_1) \cdot a$ if $1 \leq i \leq s$.
\item $e_{T^k}(NF)(p_i)=\epsilon_N(p_i) \cdot (l-a)=-\epsilon_N(p_1) \cdot (l-a)$ if $s+1 \leq i \leq k$.
\end{enumerate}
Similarly, the restriction of $e_{T^k}(NF)$ at $q_j$ is
\begin{enumerate}
\item $e_{T^k}(NF)(q_j)=\epsilon_N(q_j) \cdot a=\epsilon_N(p_1) \cdot a$ if $1 \leq j \leq t$.
\item $e_{T^k}(NF)(q_j)=\epsilon_N(q_j) \cdot (l-a)=-\epsilon_N(p_1) \cdot (l-a)$ if $t+1 \leq j \leq k$.
\end{enumerate}

For the induced action on $F$, $\int_F:=\pi_*$ is a map from $H_{T^k}^i(F;\mathbb{Z})$ to $H^{i- \dim F}(\mathbb{CP}^\infty;\mathbb{Z})$ for all $i \in \mathbb{Z}$. Since $\dim F=2n-2 \geq 4$ and $e_{T^k}(NF)$ has degree 2, the image under $\int_F$ of $e_{T^k}(NF)$ vanishes;
\begin{center}
$\displaystyle \int_F e_{T^k}(NF)=0$.
\end{center}
On the other hand, applying Theorem \ref{t24} to the induced action on $F$ with taking $\alpha=e_{T^k}(NF)$,
\begin{center}
$\displaystyle \int_F e_{T^k}(NF)=\sum_{p \in F^{T^k}}\int_p \frac{e_{T^k}(NF)|_p}{e_{T^k}(T_pF)}$\footnote{Note that the bundle in the denominator is the normal bundle of $p$ in $F$, which is therefore the tangent space $T_pF$ of $p$ in $F$.}

$\displaystyle =\sum_{i=1}^k \bigg\{ \epsilon_F(p_i)\frac{e_{T^k}(NF)(p_i)}{e_{T^k}(T_{p_i}F)}\bigg\}+ \sum_{j=1}^k \bigg\{ \epsilon_F(q_j)\frac{e_{T^k}(NF)(q_j)}{e_{T^k}(T_{q_j}F)}\bigg\}$

$\displaystyle =\sum_{i=1}^s \bigg\{ \epsilon_F(p_i)\frac{e_{T^k}(NF)(p_i)}{e_{T^k}(T_{p_i}F)} \bigg\} + \sum_{i=s+1}^k \bigg\{ \epsilon_F(p_i)\frac{e_{T^k}(NF)(p_i)}{e_{T^k}(T_{p_i}F)} \bigg\}$

$\displaystyle + \sum_{j=1}^t \bigg\{ \epsilon_F(q_j)\frac{e_{T^k}(NF)(q_j)}{e_{T^k}(T_{q_j}F)}\bigg\} + \sum_{j=t+1}^k \bigg\{ \epsilon_F(q_j)\frac{e_{T^k}(NF)(q_j)}{e_{T^k}(T_{q_j}F)}\bigg\}$

$\displaystyle =\sum_{i=1}^s \bigg\{ +\frac{+\epsilon_N(p_1) \cdot a}{l^{n-1}}\bigg\} +\sum_{i=s+1}^k \bigg\{ +\frac{-\epsilon_N(p_1) \cdot (l-a)}{l^{n-1}}\bigg\}$

$\displaystyle +\sum_{j=1}^t \bigg\{ -\frac{+\epsilon_N(p_1) \cdot a}{l^{n-1}} \bigg\} + \sum_{j=t+1}^k \bigg\{ -\frac{-\epsilon_N(p_1) \cdot (l-a)}{l^{n-1}}\bigg\}$

$\displaystyle =\frac{\epsilon_N(p_1)}{l^{n-1}} \{sa - (k-s)(l-a) - ta +(k-t)(l-a)\}$

$\displaystyle =\frac{\epsilon_N(p_1)}{l^{n-2}}(s-t).$
\end{center}
Therefore, $s=t$. Then $p_1(=q),\cdots,p_s$ have fixed point data $[\epsilon(q),l,\cdots,l,a]$ and $q_1,\cdots,q_s$ have fixed point data $[-\epsilon(q),l,\cdots,l,a]$. Also, $p_{s+1},\cdots,p_k$ have fixed point data $[-\epsilon(q),l,\cdots,l,l-a]$ and $q_{s+1},\cdots,q_k$ have fixed point data $[\epsilon(q),l,\cdots,l,l-a]$. Thus the lemma holds. \end{proof}

\section{Dimension 6: some base cases} \label{s9}

We prove our main results by separating into cases in terms of the biggest weight. First, we deal with the case that all weights are small.

\begin{lem} \label{l92}
Let the circle act effectively on a 6-dimensional compact connected oriented manifold $M$ with isolated fixed points, whose isotropy submanifolds are orientable. Suppose that the maximum of the absolute values of weights is at most 2, that is, $\max\{|w_{p,i}| : p \in M^{S^1}, 1 \leq i \leq 3\} \leq 2$. Then the number of fixed points with fixed point data $[+,1,1,1]$ ($[+,1,1,2]$ and $[+,1,2,2]$) and the number of fixed points with fixed point data $[-,1,1,1]$ ($[-,1,1,2]$ and $[-,1,2,2]$) are equal, respectively. Consequently, we can successively take self connected sums at the fixed points of $M$, to construct a fixed point free $S^1$-action on a 6-dimensional compact connected oriented manifold.
Moreover, there exists a multigraph describing $M$ that satisfies Property A, and we can convert it to the empty multigraph by reversing edges and by self connected sums.
\end{lem}

\begin{proof}
By Proposition \ref{p34}, there exists a multigraph describing $M$ that satisfies Property A. By reversing edges (the signs of weights), let every edge (fixed point) have label 1 or 2 (weights $+1$ or $+2$). By Lemma \ref{l36}, the resulting multigraph $\Gamma$ also describes $M$ and satisfies Property A.

Since the action is effective, possible multisets of weights at a fixed point are $\{1,1,1\}$, $\{1,1,2\}$, and $\{1,2,2\}$. Let $k_1,\cdots,k_6$ be the numbers of fixed points with fixed point data $[+,1,1,1]$, $[+,1,1,2]$, $[+,1,2,2]$, $[-,1,1,1]$, $[-,1,1,2]$, and $[-,1,2,2]$, respectively. By Proposition \ref{p26}, the number $k_1+k_2+k_3$ of fixed points with sign $+1$ is equal to the number $k_4+k_5+k_6$ of fixed points with sign $-1$, that is,  $k_1+k_2+k_3=k_4+k_5+k_6$. Moreover, by Lemma \ref{l82}, the number $k_3$ of fixed points with fixed point data $[+,1,2,2]$ and the number $k_6$ of fixed points with fixed point data $[-,1,2,2]$ are equal; $k_3=k_6$. Thus, $k_1+k_2=k_4+k_5$.

Applying Lemma \ref{l28} to $M$ (take $a=1$), it follows that $3k_1+2k_2+k_3=3k_4+2k_5+k_6$. Therefore, $3k_1+2k_2=3k_4+2k_5$. With $k_1+k_2=k_4+k_5$, this implies that $k_1=k_4$ and hence $k_2=k_5$. Therefore, we can divide the fixed points into pairs $(p_i,q_i)$ so that $\Sigma_{q_i}=-\Sigma_{p_i}$.

Let $M_1$ ($\Gamma_1$) be a connected sum of $M$ ($\Gamma$) at $p_1$ and $q_1$. By Proposition \ref{p48}, $\Gamma_1$ describes $M_1$ with Property A. Next, let $M_2$ ($\Gamma_2$) be a self connected sum of $M_1$ ($\Gamma_1$) at $p_2$ and $q_2$; then $\Gamma_2$ describes $M_2$ with Property A. Continuing this, the lemma follows. 
\end{proof} 

\section{Dimension 6: proof of the main results} \label{s10}

In this section, we prove our main results, Theorems \ref{t12}, \ref{t61}, and \ref{t62} altogether. We prove these theorems by showing that we can successively remove fixed points that have a weight of the biggest magnitude, by taking connected sums at fixed points with itself, or with $\mathbb{CP}^3$, $Z_1$, and $Z_2$ (and these with opposite orientations).

\begin{proof}[\textbf{Proof of Theorems \ref{t12}, \ref{t61}, and \ref{t62}}]

\textbf{Step 1 - Preparation: reverse labels of multigraph and weights to be positive}

By Proposition \ref{p34}, there exists a multigraph $\Gamma_0$ describing $M$ that satisfies Property A. For an edge $e$ of $\Gamma_0$, if its label $w(e)$ is negative, then we reverse the edge $e$; Definition \ref{d35}. By Lemma \ref{l36}, the resulting multigraph $\Gamma$ also describes $M$ and satisfies Property A. Then the label of every edge of $\Gamma$ is positive. Accordingly, for every weight $w_{p,i}$ at a fixed point $p$, if $w_{p,i}<0$, we replace $w_{p,i}$ by $-w_{p,i}$ (reverse the orientation of $L_{p,i}$), so that every weight is positive.

\textbf{Step 2 - Pick the biggest weight}

Let $l$ be the biggest weight that occurs at any fixed point, that is,
\begin{center}
$l=\max\{w_{p,i}:p \in M^{S^1}, 1 \leq i \leq 3\}$.
\end{center}

\textbf{Case (0-a): $l \leq 2$ - We use self connected sum (for both manifold and multigraph) and Operation (1) of Theorem \ref{t62}}

When $l \leq 2$, the theorem follows from Lemma \ref{l92}, in which we use self connected sums for both $M$ and $\Gamma$. Because a self connected sum removes fixed points $p_1$ and $p_2$ that have the opposite fixed point data ($\Sigma_{p_1}=-\Sigma_{p_2}$), this corresponds to Operation (1) of Theorem \ref{t62}. Figures \ref{f5} and \ref{f6} (when there are no edges between $p_1$ and $p_2$), and Figures \ref{f7} through \ref{f13} (when there are edges between $p_1$ and $p_2$) illustrate self connected sums of $\Gamma$ in various cases.

\textbf{From now on, we assume that $l >2$. A fixed point $p_1$ has weight $l$, and we may assume that $\epsilon(p_1)=+1$}

From now on we assume that $l>2$. Let $e$ be an edge with label $l$, and let $p_1$ and $p_2$ be the vertices of $e$. Then both of fixed points $p_1$ and $p_2$ have weight $l$. By reversing the orientation of $M$, we may assume that $\epsilon(p_1)=+1$; the other case is completely analogous. Since $\Gamma$ describes $M$, $p_1$ and $p_2$ are in the same component $F$ of $M^{\ker l}$. Because the action is effective, there are two possibilities.
\begin{enumerate}[(I)]
\item The multiplicity of $l$ in $T_{p_1}M$ is 2. 
\item The multiplicity of $l$ in $T_{p_1}M$ is 1. 
\end{enumerate}

\textbf{Case (0-b): $p_1$ has $l$ twice - Self connected sum, Operation (1)}

Suppose that the multiplicity of $l$ in $T_{p_1}M$ is 2. Let $[+,l,l,x]$ be the fixed point data of $p_1$ for some positive integer $x$. By Lemma \ref{l82}, there is a fixed points $q \in F \cap M^{S^1}$ that has fixed point data $[-,l,l,x]$. By Step 1, the labels of the edges of $p_1$ ($p_2$) are $l$, $l$, and $x$.

Let $M'$ ($\Gamma'$) be a self connected sum at $p_1$ and $q$ of $M$ ($\Gamma$). By Proposition \ref{p48}, $\Gamma'$ describes $M'$ and satisfies Property A. Because a self connected sum removes fixed points $p_1$ and $q$ that have the opposite fixed point data (that is, $\Sigma_{p_1}=-\Sigma_{p}$), this corresponds to Operation (1) of Theorem \ref{t62}.

\textbf{From now on $p_1$ has $l$ once - We divide into more cases}

Suppose that the multiplicity of $l$ in $T_{p_1}M$ is 1. Let $[+,l,x,y]$ be the fixed point data of $p_1$ for some positive integers $x$ and $y$. Then $x,y<l$. By Lemma \ref{l81}, $p_1$ and $p_2$ are in the same 2-sphere $F \subset M^{\ker l}$, and one of the following holds for the fixed point data of $p_2$:
\begin{enumerate}
\item $\Sigma_{p_2}=[-,l,x,y]$.
\item $\Sigma_{p_2}=[-,l,l-x,l-y]$.
\item $\Sigma_{p_2}=[+,l,x,l-y]$.
\item $\Sigma_{p_2}=[+,l,l-x,y]$.
\end{enumerate}
Permuting $x$ and $y$, Case (3) and Case (4) are equivalent; we only need to consider Cases (1-3).

\textbf{Case (1) - Self connected sum, Operation (1)}

Suppose that $\Sigma_{p_2}=[-,l,x,y]$. Then $\Sigma_{p_1}=-\Sigma_{p_2}$. Let $M'$ ($\Gamma'$) be a connected sum at $p_1$ and $p_2$ of $M$ ($\Gamma$). By Proposition \ref{p48}, $\Gamma'$ describes $M'$ and satisfies Property A. The self connected sum corresponds to Operation (1) of Theorem \ref{t62}. Figures \ref{f7} through \ref{f13} illustrate this connected sum in various situations.

\textbf{Case (2) - We divide into more cases}

Assume $\Sigma_{p_2}=[-,l,l-x,l-y]$. We consider the following three cases.
\begin{enumerate}[(a)]
\item $x \neq y$.
\item $x=y$ and $2x \neq l$.
\item $x=y$ and $2x=l$.
\end{enumerate}

\textbf{Case (2-a) - Connected sum with $\mathbb{CP}^3$ (and with itself), Operation (2)}

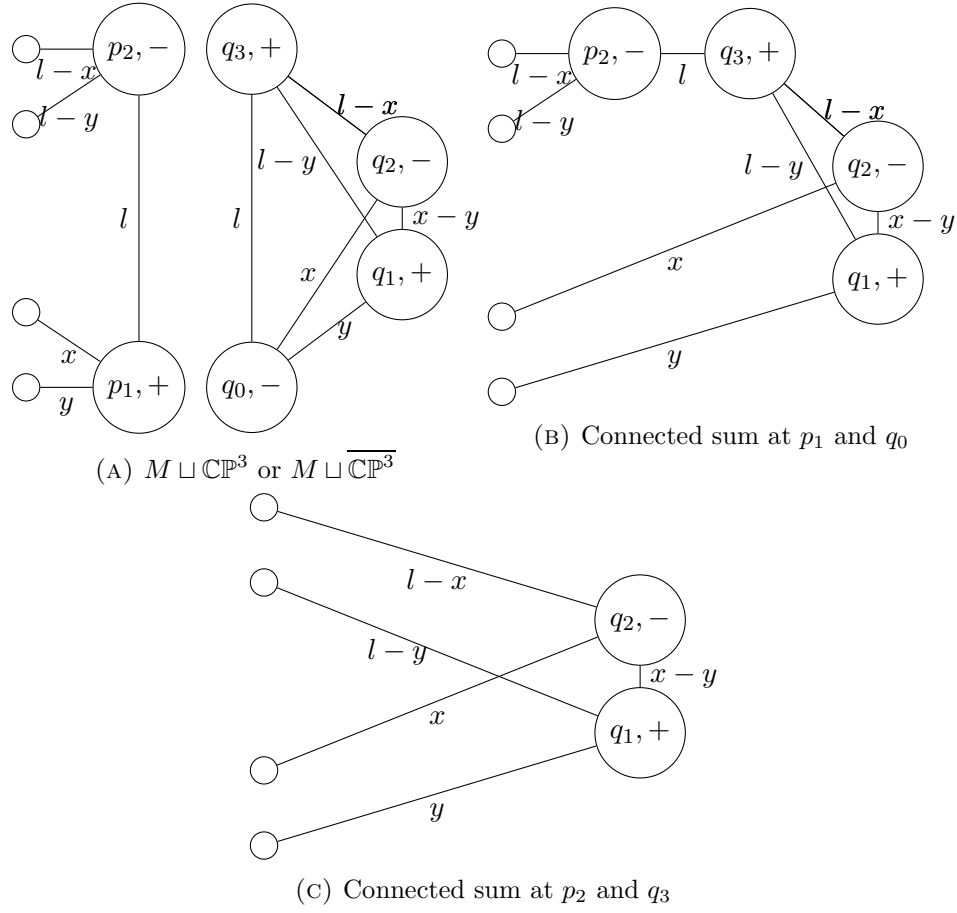
\begin{figure} 
\centering
\begin{subfigure}[b][6cm][s]{.49\textwidth}
\centering
\vfill
\begin{tikzpicture}[state/.style ={circle, draw}]
\node[state] (a) at (0,0) {$q_0,-$};
\node[state] (b) at (2,1.5){$q_1,+$};
\node[state] (c) at (2,3) {$q_2,-$};
\node[state] (d) at (0,4.5){$q_3,+$};
\node[state] (e) at (-1.5,0) {$p_1,+$};
\node[state] (f) at (-1.5,4.5) {$p_2,-$};
\node[state] (g) at (-3,0) {};
\node[state] (h) at (-3,4.5) {};
\node[state] (i) at (-3,1) {};
\node[state] (j) at (-3,3.5) {};
\path (a) edge node[right] {$y$} (b);
\path (a) edge node[left] {$x$} (c);
\path (a) edge node[left] {$l$} (d);
\path (b) edge node [right] {$x-y$} (c);
\path (b) edge node [left] {$l-y$} (d);
\path (c) edge node [right] {$l-x$} (d);
\path (c) edge node [right] {$l-x$} (d);
\path (e) edge node[left] {$l$} (f);
\path (e) edge node[below] {$y$} (g);
\path (f) edge node[below] {$l-x$} (h);
\path (e) edge node[below] {$x$} (i);
\path (f) edge node[below] {$l-y$} (j);
\end{tikzpicture}
\caption{$M \sqcup \mathbb{CP}^3$ or $M \sqcup \overline{\mathbb{CP}^3}$} \label{f29-1}
\vfill
\end{subfigure}
\begin{subfigure}[b][6cm][s]{.49\textwidth}
\centering
\vfill
\begin{tikzpicture}[state/.style ={circle, draw}]
\node[state] (b) at (2,1.5){$q_1,+$};
\node[state] (c) at (2,3) {$q_2,-$};
\node[state] (d) at (0.3,4.5){$q_3,+$};
\node[state] (f) at (-1.5,4.5) {$p_2,-$};
\node[state] (g) at (-3,0) {};
\node[state] (h) at (-3,4.5) {};
\node[state] (i) at (-3,1) {};
\node[state] (j) at (-3,3.5) {};
\path (g) edge node[below] {$y$} (b);
\path (i) edge node[below] {$x$} (c);
\path (b) edge node [right] {$x-y$} (c);
\path (b) edge node [left] {$l-y$} (d);
\path (c) edge node [right] {$l-x$} (d);
\path (c) edge node [right] {$l-x$} (d);
\path (f) edge node[below] {$l-x$} (h);
\path (f) edge node[below] {$l-y$} (j);
\path (f) edge node[below] {$l$} (d);
\end{tikzpicture}
\caption{Connected sum at $p_1$ and $q_0$} \label{f29-2}
\vfill
\end{subfigure}
\begin{subfigure}[b][6cm][s]{.49\textwidth}
\centering
\vfill
\begin{tikzpicture}[state/.style ={circle, draw}]
\node[state] (b) at (2,1.5){$q_1,+$};
\node[state] (c) at (2,3) {$q_2,-$};
\node[state] (g) at (-3,0) {};
\node[state] (h) at (-3,4.5) {};
\node[state] (i) at (-3,1) {};
\node[state] (j) at (-3,3.5) {};
\path (g) edge node[below] {$y$} (b);
\path (i) edge node[below] {$x$} (c);
\path (b) edge node [right] {$x-y$} (c);
\path (b) edge node [left] {$l-y$} (j);
\path (c) edge node [below] {$l-x$} (h);
\end{tikzpicture}
\caption{Connected sum at $p_2$ and $q_3$} \label{f29-3}
\vfill
\end{subfigure}
\caption{Case (2-a)}\label{f29}
\end{figure}

Suppose that Case (2-a) holds. The fixed points $p_1$ and $p_2$ have fixed point data $[+,l,x,y]$ and $[-,l,l-x,l-y]$ with $x \neq y$. We may assume that $x>y$. In Example \ref{e412} of the action on $\mathbb{CP}^3$, take $a_3=l$, $a_2=x$, and $a_1=y$, and reverse its orientation (if $\epsilon(p_1)=-1$ then we do not reverse the orientation of $\mathbb{CP}^3$); its fixed point data is
\begin{center}
$\Sigma_{q_0}=[-,y,x,l]$, $\Sigma_{q_1}=[+,y,x-y,l-y]$, $\Sigma_{q_2}=[-,x,x-y,l-x]$, $\Sigma_{q_3}=[+,l,l-y,l-x]$.
\end{center}
Because $\Sigma_{p_1}=-\Sigma_{q_0}$ and $\Sigma_{p_2}=-\Sigma_{q_3}$, we can take an equivariant connected sum at $p_1$ and $p_2$ of $M$ and $q_1$ and $q_4$ of $\overline{\mathbb{CP}^3}$ to construct another compact connected oriented $S^1$-manifold $M'$ with isolated fixed points that has fixed point data 
\begin{center}
$\displaystyle \Big\{ \Sigma_M \setminus ([+,l,x,y] \cup [-,l,l-y,l-x]) \Big\} \cup ([+,y,x-y,l-y] \cup [-,x,x-y,l-x])$. 
\end{center}
This corresponds to Operation (2) of Theorem \ref{t62} with $l=C$, $x=B$, and $y=A$. The connected sum removes $p_1$ and $p_2$ that have weights $\{l,x,y\}$ and $\{l,l-x,l-y\}$, and adds $q_1$ and $q_2$ that have weights $\{y,x-y,l-y\}$, $\{x,x-y,l-x\}$. In terms of weights, it removes two $l$'s and adds two $x-y$'s, where $0<x-y<l$\footnote{This only holds for $S^1$-actions. If instead $T^2$ or $T^3$ acts on $M$ then the weight $x-y$ may have magnitude bigger than $|l|$; then the arguments of this proof do not work for $T^2$ or $T^3$-actions.}. Therefore, the connected sum reduces the number of biggest weights.

For a multigraph, we do this in two steps. First, we take a connected sum of $M$ and $\overline{\mathbb{CP}^3}$ ($\Gamma$ and the right figure of Figure \ref{f29-1}) at $p_1$ and $q_0$, to construct another $S^1$-manifold $M_1$ (another multigraph $\Gamma_1$, which is Figure \ref{f29-2}). Next, we take a self connected sum of $M_1$ ($\Gamma_1$) at $p_2$ and $q_3$ to construct the manifold $M'$ (a multigraph $\Gamma'$, which is Figure \ref{f29-3}). By Propositions \ref{p47} and \ref{p48}, $\Gamma_1$ and $\Gamma'$ describe $M_1$ and $M'$ with Property A, respectively.

\textbf{Case (2-b) - Connected sum with $Z_2 \sharp \overline{Z_2}$, Operation (5)}

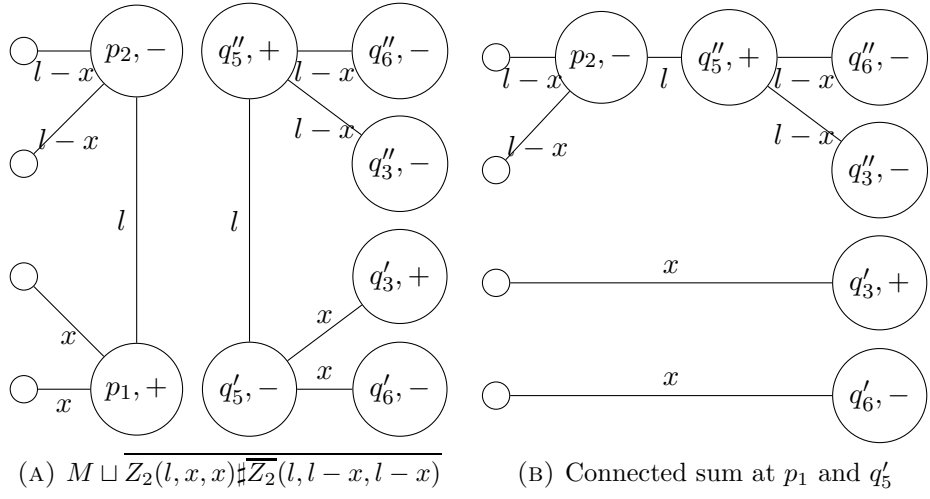
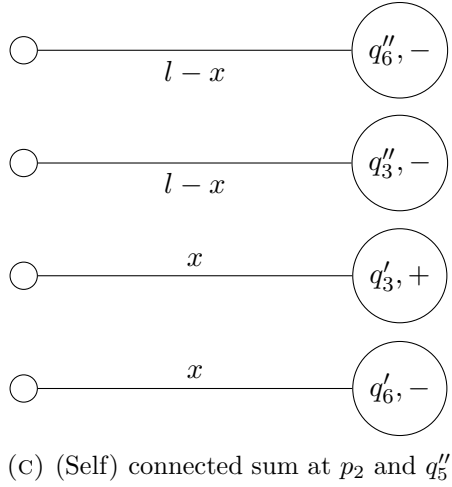
\begin{figure}
\centering
\begin{subfigure}[b][6.6cm][s]{.49\textwidth}
\centering
\vfill
\begin{tikzpicture}[state/.style ={circle, draw}]
\node[state] (a) at (0,0) {$q_5',-$};
\node[state] (b) at (0,4.5) {$q_5'',+$};
\node[state] (c) at (2,0) {$q_6',-$};
\node[state] (d) at (2,4.5) {$q_6'',-$};
\node[state] (e) at (2,1.5) {$q_3',+$};
\node[state] (l) at (2,3) {$q_3'',-$};
\path (a) edge node[left] {$l$} (b);
\path (a) edge node[above] {$x$} (c);
\path (a) edge node[above] {$x$} (e);
\path (b) edge node[below] {$l-x$} (d);
\path (b) edge node[below] {$l-x$} (l);
\node[state] (f) at (-1.5,0) {$p_1,+$};
\node[state] (g) at (-1.5,4.5) {$p_2,-$};
\node[state] (h) at (-3,0) {};
\node[state] (i) at (-3,4.5) {};
\node[state] (j) at (-3,1.5) {};
\node[state] (k) at (-3,3) {};
\path (f) edge node[left] {$l$} (g);
\path (f) edge node[below] {$x$} (h);
\path (g) edge node[below] {$l-x$} (i);
\path (f) edge node[below] {$x$} (j);
\path (g) edge node[below] {$l-x$} (k);
\end{tikzpicture}
\caption{$M \sqcup \overline{Z_2(l,x,x) \sharp \overline{Z_2}(l,l-x,l-x)}$}\label{f30-1}
\vfill
\end{subfigure}
\begin{subfigure}[b][6.6cm][s]{.49\textwidth}
\centering
\vfill
\begin{tikzpicture}[state/.style ={circle, draw}]
\node[state] (b) at (0.1,4.5) {$q_5'',+$};
\node[state] (c) at (2.1,0) {$q_6',-$};
\node[state] (d) at (2.1,4.5) {$q_6'',-$};
\node[state] (e) at (2.1,1.5) {$q_3',+$};
\node[state] (l) at (2.1,3) {$q_3'',-$};
\path (e) edge node[above] {$x$} (j);
\path (c) edge node[above] {$x$} (h);
\path (b) edge node[below] {$l-x$} (d);
\path (b) edge node[below] {$l-x$} (l);
\node[state] (g) at (-1.6,4.5) {$p_2,-$};
\node[state] (h) at (-3,0) {};
\node[state] (i) at (-3,4.5) {};
\node[state] (j) at (-3,1.5) {};
\node[state] (k) at (-3,3) {};
\path (g) edge node[below] {$l-x$} (i);
\path (g) edge node[below] {$l-x$} (k);
\path (g) edge node[below] {$l$} (b);
\end{tikzpicture}
\caption{Connected sum at $p_1$ and $q_5'$}\label{f30-2}
\vfill
\end{subfigure}
\begin{subfigure}[b][6.6cm][s]{.49\textwidth}
\centering
\vfill
\begin{tikzpicture}[state/.style ={circle, draw}]
\node[state] (c) at (2,0) {$q_6',-$};
\node[state] (d) at (2,4.5) {$q_6'',-$};
\node[state] (e) at (2,1.5) {$q_3',+$};
\node[state] (l) at (2,3) {$q_3'',-$};
\path (e) edge node[above] {$x$} (j);
\path (c) edge node[above] {$x$} (h);
\path (d) edge node[below] {$l-x$} (i);
\path (l) edge node[below] {$l-x$} (k);
\node[state] (h) at (-3,0) {};
\node[state] (i) at (-3,4.5) {};
\node[state] (j) at (-3,1.5) {};
\node[state] (k) at (-3,3) {};
\end{tikzpicture}
\caption{(Self) connected sum at $p_2$ and $q_5''$}\label{f30-3}
\vfill
\end{subfigure}
\caption{Case (2-b)} \label{f30}
\end{figure}

Suppose that Case (2-b) holds. We have $\Sigma_{p_1}=[+,l,x,x]$ and $\Sigma_{p_2}=[-,l,l-x,l-x]$ with $2x \neq l$. In Example \ref{e74} of the action on $Z:=Z_2(a,e,e) \sharp \overline{Z_2}(a,a-e,a-e)$ we take $a=l$ and $e=x$ and reverse its orientation; it has 10 fixed points $q_2',\cdots,q_6',q_2'',\cdots,q_6''$ that have fixed point data
\begin{center}
$[+,l-x,l-2x,x]$, $[+,l-x, x, x]$, $[+,l-x, x, x]$, $[-, l, x, x]$, $[-,l-2x, x, x]$, $[+, x, l-2x, l-x]$, $[-,x,l-x,l-x]$, $[-,x,l-x,l-x]$, $[+, l, l-x, l-x]$, $[-, l-2x, l-x, l-x]$.
\end{center}
We have $\Sigma_{p_1}=[+,l,x,x]=-\Sigma_{q_5'}$ and $\Sigma_{p_2}=-\Sigma_{q_5''}$. We take a connected sum at $p_1$ and $p_2$ of $M$ and $q_5'$ and $q_5''$ of $\overline{Z}$ to construct another $S^1$-manifold $M'$ with fixed point data 
\begin{center}
$\Sigma_{M'}=\{\Sigma_M \setminus ([+,l,x,x] \cup [-,l,l-x,l-x])\} \cup ([+,l-x,l-2x,x] \cup [+,l-x, x, x] \cup [+,l-x, x, x] \cup [-,l-2x, x, x] \cup [+, x, l-2x, l-x] \cup [-,x,l-x,l-x] \cup [-,x,l-x,l-x] \cup [-, l-2x, l-x, l-x])$. 
\end{center}
This corresponds to Operation (5) of Theorem \ref{t62} with $l=C$ and $x=A$. The manifold $M'$ has two less $l$ than $M$. The other weights $x,l-x,l-2x$ added have magnitudes strictly smaller than $l$ (the weight $l-2x$ may be negative).

As in Case (2-a), for a multigraph we do this in two steps, first we take a connected sum of $M$ ($\Gamma$) and $Z$ (Figure \ref{f28}) at $p_1$ and $q_5'$ to construct another $S^1$-manifold $M_1$ (another multigraph $\Gamma_1$, which is Figure \ref{f30-2}), and then take a self connected sum of $M_1$ ($\Gamma_1$) at $p_2$ and $q_5''$ to construct $M'$ ($\Gamma'$, which is Figure \ref{f30-3}).

As explained in Example \ref{e74}, Figure \ref{f28} describes the manifold $Z$ and satisfies Property A, and hence $\Gamma'$ describes $M'$ with Property A, by Propositions \ref{p47} and \ref{p48}.

\textbf{Case (2-c) - Self connected sum, Operation (1)}

Suppose that Case (2-c) holds. Since $p_1$ has weights $\{2x(=l),x,x(=y)\}$ and the action is effective, we have that $x=1$, that is, the biggest weight $l=2x$ is 2. This case is then Case (0-a); thus proceed as in Case (0-a).

\textbf{Case (3) - We divide into more cases}

Assume that Case (3) holds. We consider the following three cases.
\begin{enumerate}[(a)]
\item $x \neq y$.
\item $x=y$ and $2x \neq l$.
\item $x=y$ and $2x=l$.
\end{enumerate}

\textbf{Case (3-a) - Connected sum with $Z_1(l,y,x)$, Operation (3)}

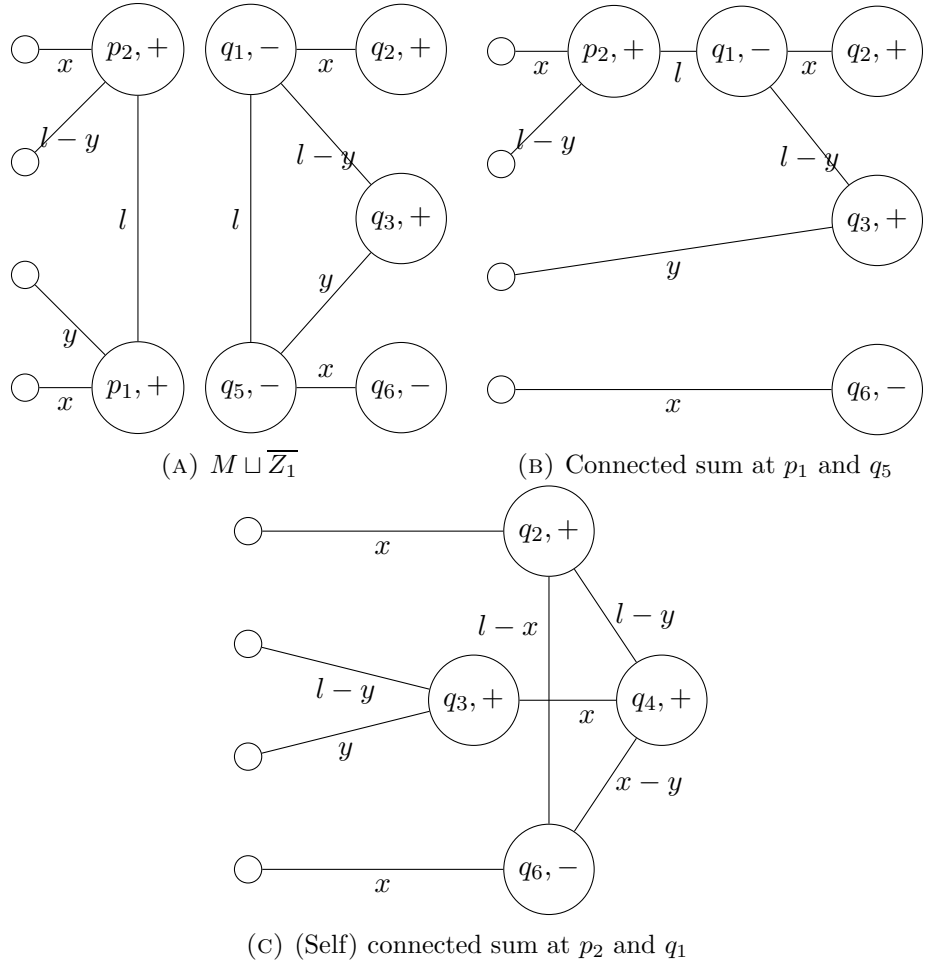
\begin{figure}
\centering
\begin{subfigure}[b][6.4cm][s]{.49\textwidth}
\centering
\vfill
\begin{tikzpicture}[state/.style ={circle, draw}]
\node[state] (a) at (0,0) {$q_5,-$};
\node[state] (b) at (0,4.5) {$q_1,-$};
\node[state] (c) at (2,0) {$q_6,-$};
\node[state] (d) at (2,4.5) {$q_2,+$};
\node[state] (e) at (2,2.25) {$q_3,+$};
\path (a) edge node[left] {$l$} (b);
\path (a) edge node[above] {$x$} (c);
\path (a) edge node[above] {$y$} (e);
\path (b) edge node[below] {$x$} (d);
\path (b) edge node[below] {$l-y$} (e);
\node[state] (f) at (-1.5,0) {$p_1,+$};
\node[state] (g) at (-1.5,4.5) {$p_2,+$};
\node[state] (h) at (-3,0) {};
\node[state] (i) at (-3,4.5) {};
\node[state] (j) at (-3,1.5) {};
\node[state] (k) at (-3,3) {};
\path (f) edge node[left] {$l$} (g);
\path (f) edge node[below] {$x$} (h);
\path (g) edge node[below] {$x$} (i);
\path (f) edge node[below] {$y$} (j);
\path (g) edge node[below] {$l-y$} (k);
\end{tikzpicture}
\caption{$M \sqcup \overline{Z_1}$}\label{f31-1}
\vfill
\end{subfigure}
\begin{subfigure}[b][6.4cm][s]{.49\textwidth}
\centering
\vfill
\begin{tikzpicture}[state/.style ={circle, draw}]
\node[state] (b) at (0.2,4.5) {$q_1,-$};
\node[state] (c) at (2,0) {$q_6,-$};
\node[state] (d) at (2,4.5) {$q_2,+$};
\node[state] (e) at (2,2.25) {$q_3,+$};
\path (b) edge node[below] {$x$} (d);
\path (b) edge node[below] {$l-y$} (e);
\node[state] (g) at (-1.5,4.5) {$p_2,+$};
\node[state] (h) at (-3,0) {};
\node[state] (i) at (-3,4.5) {};
\node[state] (j) at (-3,1.5) {};
\node[state] (k) at (-3,3) {};
\path (c) edge node[below] {$x$} (h);
\path (g) edge node[below] {$x$} (i);
\path (e) edge node[below] {$y$} (j);
\path (g) edge node[below] {$l-y$} (k);
\path (g) edge node[below] {$l$} (b);
\end{tikzpicture}
\caption{Connected sum at $p_1$ and $q_5$}\label{f31-2}
\vfill
\end{subfigure}
\begin{subfigure}[b][6.2cm][s]{.49\textwidth}
\centering
\vfill
\begin{tikzpicture}[state/.style ={circle, draw}]
\node[state] (c) at (2,0) {$q_6,-$};
\node[state] (d) at (2,4.5) {$q_2,+$};
\node[state] (e) at (1,2.25) {$q_3,+$};
\node[state] (a) at (3.5,2.25) {$q_4,+$};
\node[state] (h) at (-2,0) {};
\node[state] (i) at (-2,4.5) {};
\node[state] (j) at (-2,1.5) {};
\node[state] (k) at (-2,3) {};
\path (i) edge node[below] {$x$} (d);
\path (k) edge node[below] {$l-y$} (e);
\path (c) edge node[below] {$x$} (h);
\path (e) edge node[below] {$y$} (j);
\path (c) edge node[pos=.8, left] {$l-x$} (d);
\path (a) edge node[right] {$x-y$} (c);
\path (a) edge node[right] {$l-y$} (d);
\path (a) edge node[pos=.3, below] {$x$} (e);
\end{tikzpicture}
\caption{(Self) connected sum at $p_2$ and $q_1$}\label{f31-3}
\vfill
\end{subfigure}
\caption{Case (3-a)} \label{f31}
\end{figure}

Suppose that Case (3-a) holds. The fixed points $p_1$ and $p_2$ have fixed point data $[+,l,x,y]$ and $[+,l,x,l-y]$ with $x \neq y$. In Example \ref{e72} of the action on $Z_1(a,b,c)$ we take $a=l$, $b=y$, and $c=x$, and reverse its orientation; its fixed point data is
\begin{center}
$[-, l-y, l, x]$, $[+, l-y, l-x, x]$, $[+, l-y, y, x]$, $[+, l-y, x-y, x]$, $[-, l, y, x]$, $[-, l-x, x-y, x]$.
\end{center}
We take a connected sum at $p_1$ and $p_2$ of $M$ and $q_5$ and $q_1$ of $\overline{Z_1(l,y,x)}$ to construct another $S^1$-manifold $M'$ with fixed point data
\begin{center}
$\Sigma_{M'}=\{\Sigma_M \setminus ([+,l,x,y] \cup [+,l,x,l-y])\} \cup ([+, l-y, l-x, x] \cup [+, l-y, y, x] \cup [+, l-y, x-y, x] \cup [-, l-x, x-y, x])$.
\end{center}
This corresponds to Operation (3) of Theorem \ref{t62} with $x=A$, $y=B$, and $l=C$. The manifold $M'$ has two less $l$ than $M$. The other weights $l-y$, $l-x$, $x$, $x-y$ added have magnitudes strictly smaller than $l$.

For a multigraph, first take a connected sum of $M$ ($\Gamma$) and $Z_1$ (right of Figure \ref{f31-1}) at $p_1$ and $q_5$ to construct another $S^1$-manifold $M_1$ (another multigraph $\Gamma_1$, Figure \ref{f31-2}), and then take a self connected sum of $M_1$ ($\Gamma_1$) at $p_2$ and $q_1$ to construct $M'$ ($\Gamma'$, Figure \ref{f31-3}).

\textbf{Case (3-b) - Connected sum with $Z_2(l,x,x)$, Operation (4)}

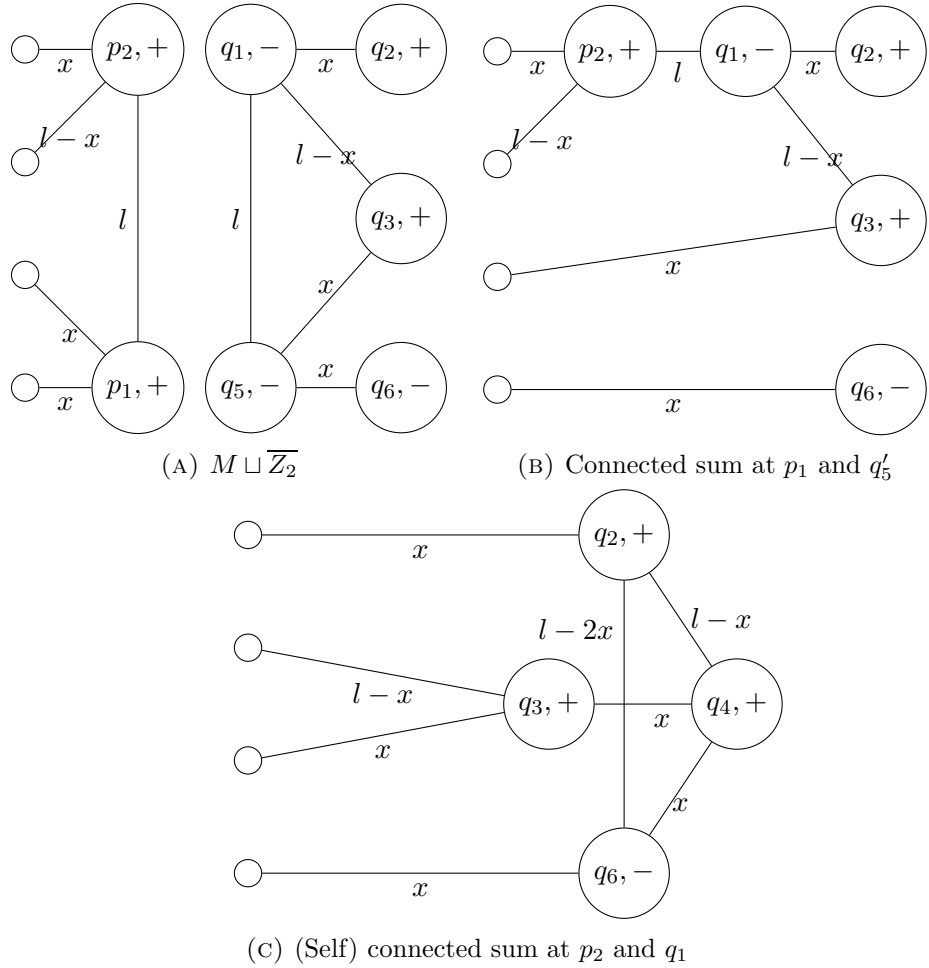
\begin{figure}
\centering
\begin{subfigure}[b][6.4cm][s]{.49\textwidth}
\centering
\vfill
\begin{tikzpicture}[state/.style ={circle, draw}]
\node[state] (a) at (0,0) {$q_5,-$};
\node[state] (b) at (0,4.5) {$q_1,-$};
\node[state] (c) at (2,0) {$q_6,-$};
\node[state] (d) at (2,4.5) {$q_2,+$};
\node[state] (e) at (2,2.25) {$q_3,+$};
\path (a) edge node[left] {$l$} (b);
\path (a) edge node[above] {$x$} (c);
\path (a) edge node[above] {$x$} (e);
\path (b) edge node[below] {$x$} (d);
\path (b) edge node[below] {$l-x$} (e);
\node[state] (f) at (-1.5,0) {$p_1,+$};
\node[state] (g) at (-1.5,4.5) {$p_2,+$};
\node[state] (h) at (-3,0) {};
\node[state] (i) at (-3,4.5) {};
\node[state] (j) at (-3,1.5) {};
\node[state] (k) at (-3,3) {};
\path (f) edge node[left] {$l$} (g);
\path (f) edge node[below] {$x$} (h);
\path (g) edge node[below] {$x$} (i);
\path (f) edge node[below] {$x$} (j);
\path (g) edge node[below] {$l-x$} (k);
\end{tikzpicture}
\caption{$M \sqcup \overline{Z_2}$}\label{f32-1}
\vfill
\end{subfigure}
\begin{subfigure}[b][6.4cm][s]{.49\textwidth}
\centering
\vfill
\begin{tikzpicture}[state/.style ={circle, draw}]
\node[state] (b) at (0.3,4.5) {$q_1,-$};
\node[state] (c) at (2.1,0) {$q_6,-$};
\node[state] (d) at (2.1,4.5) {$q_2,+$};
\node[state] (e) at (2.1,2.25) {$q_3,+$};
\path (b) edge node[below] {$x$} (d);
\path (b) edge node[below] {$l-x$} (e);
\node[state] (g) at (-1.5,4.5) {$p_2,+$};
\node[state] (h) at (-3,0) {};
\node[state] (i) at (-3,4.5) {};
\node[state] (j) at (-3,1.5) {};
\node[state] (k) at (-3,3) {};
\path (c) edge node[below] {$x$} (h);
\path (g) edge node[below] {$x$} (i);
\path (e) edge node[below] {$x$} (j);
\path (g) edge node[below] {$l-x$} (k);
\path (g) edge node[below] {$l$} (b);
\end{tikzpicture}
\caption{Connected sum at $p_1$ and $q_5'$}\label{f32-2}
\vfill
\end{subfigure}
\begin{subfigure}[b][6.4cm][s]{.49\textwidth}
\centering
\vfill
\begin{tikzpicture}[state/.style ={circle, draw}]
\node[state] (a) at (3.5,2.25) {$q_4,+$};
\node[state] (c) at (2,0) {$q_6,-$};
\node[state] (d) at (2,4.5) {$q_2,+$};
\node[state] (e) at (1,2.25) {$q_3,+$};
\path (a) edge node[right] {$l-x$} (d);
\path (i) edge node[below] {$x$} (d);
\path (a) edge node[below] {$x$} (c);
\path (a) edge node[pos=.3, below] {$x$} (e);
\path (k) edge node[below] {$l-x$} (e);
\path (c) edge node[pos=.8, left] {$l-2x$} (d);
\node[state] (h) at (-3,0) {};
\node[state] (i) at (-3,4.5) {};
\node[state] (j) at (-3,1.5) {};
\node[state] (k) at (-3,3) {};
\path (c) edge node[below] {$x$} (h);
\path (e) edge node[below] {$x$} (j);
\end{tikzpicture}
\caption{(Self) connected sum at $p_2$ and $q_1$}\label{f32-3}
\vfill
\end{subfigure}
\caption{Case (3-b)} \label{f32}
\end{figure}

Suppose that Case (3-b) holds. The fixed points $p_1$ and $p_2$ have fixed point data $[+,l,x,x]$ and $[+,l,x,l-x]$ with $2x \neq l$. In Example \ref{e73} of the action on $Z_2(a,d,d)$ we take $a=l$ and $d=x$ and reverse its orientation; its fixed point data is
\begin{center}
$[-, l-x, l, x]$, $[+,l-x,l-2x,x]$, $[+,l-x, x, x]$, $[+,l-x, x, x]$, $[-, l, x, x]$, $[-,l-2x, x, x]$.
\end{center}
We take a connected sum at $p_1$ and $p_2$ of $M$ and $q_5$ and $q_1$ of $\overline{Z_2(l,x,x)}$ to construct another $S^1$-manifold $M'$ with fixed point data
\begin{center}
$\Sigma_{M'}=\{\Sigma_M \setminus ([+,l,x,x] \cup [+,l,x,l-x])\} \cup ([+,l-x,l-2x,x] \cup [+,l-x, x, x] \cup [+,l-x, x, x] \cup [-,l-2x, x, x] )$.
\end{center}
This corresponds to Operation (4) of Theorem \ref{t62} with $x=A$ and $l=C$. The manifold $M'$ has two less $l$'s than $M$. The other weights $x$, $l-x$, $l-2x$ added have magnitudes strictly smaller than $l$.

For a multigraph, first take a connected sum of $M$ ($\Gamma$) and $Z_1$ (right of Figure \ref{f32-1}) at $p_1$ and $q_5$ to construct another $S^1$-manifold $M_1$ (another multigraph $\Gamma_1$, which is Figure \ref{f32-2}), and then take a self connected sum of $M_1$ ($\Gamma_1$) at $p_2$ and $q_1$ to construct $M'$ ($\Gamma'$, which is Figure \ref{f32-3}).

\textbf{Case (3-c) - Self connected sum, Operation (1)}

Suppose that Case (3-c) holds. Since $p_1$ has weights $\{2x(=l),x,x(=y)\}$ and the action is effective, this implies that $x=1$, that is, the biggest weight $l=2x$ is 2. This case is then Case (0-a); proceed as in Case (0-a).

\textbf{Completing the proof}

In any of the cases above, by taking connected sums with itself, $\mathbb{CP}^3$, $Z_1$, $Z_2$, and $Z_2 \sharp \overline{Z_2}$ (and these with opposite orientations), we can remove fixed points (vertices) that have the biggest weight (whose edge has the biggest label) to construct another compact connected oriented $S^1$-manifold $M'$ with isolated fixed points (another multigraph $\Gamma'$). On the manifold $M'$ and the multigraph $\Gamma'$, go to Step 1 and repeat the above procedure.
After a finite number of repeating the above steps, we can remove all fixed points to construct a fixed point free circle action on a compact connected oriented manifold (we can convert the multigraph $\Gamma$ describing $M$ with Property A to the empty multigraph).

All multigraphs Figures \ref{f15-2}, \ref{f25-2}, \ref{f26-2}, and \ref{f28} describing $\mathbb{CP}^3$, $Z_1$, $Z_2$, and $Z_2 \sharp \overline{Z_2}$ satisfy Property A, respectively. Since $\Gamma$ also satisfies Property A, by Propositions \ref{p47} and \ref{p48}, any multigraph appearing in this proof satisfies Property A. Thus we can repeat the above process. Repeating the steps, the theorems all follow. \end{proof}

\section*{Funding}
This work was supported by the National Research Foundation of Korea(NRF) grant funded by the Korea government(MSIT) (2021R1C1C1004158).

\end{document}